\documentclass[11pt]{article}
\usepackage{fullpage}
\usepackage{amsmath,amsthm,amsfonts,amssymb,dsfont,bm}
\usepackage{enumerate,color,xcolor}
\usepackage{graphicx}
\usepackage{subfigure}
\usepackage[colorlinks,linkcolor=cyan,citecolor=blue]{hyperref}
\usepackage{url}

\numberwithin{equation}{section}
\theoremstyle{plain}

\newtheorem{theorem}{Theorem}[section]
\newtheorem{corollary}[theorem]{Corollary}

\newtheorem{lemma}[theorem]{Lemma}
\newtheorem{proposition}[theorem]{Proposition}

\newtheorem{remark}[theorem]{Remark}

\begin{document}
\begin{center}
{\bf\Large Renewal Hawkes Processes: Expectations and Applications}
\end{center}

\author{}
\begin{center}
{Lirong Cui}\,\footnote{Corresponding author.
 School of Management, Beijing Institute of Technology, No. 5, Zhongguancun South Street Haidian area, Beijing, 100081, People's Republic of China;
  lirongcui@bit.edu.cn},
  {Yongji Zhang}\,\footnote{Corresponding author. School of Management, Beijing Institute of Technology, No. 5, Zhongguancun South Street Haidian area, Beijing, 100081, People's Republic of China;
  yjzhang@bit.edu.cn},
  Lingjiong Zhu\,\footnote{Department of Mathematics, Florida State University, Tallahassee, Florida, United States of America; zhu@math.fsu.edu
 }
\end{center}

\begin{center}
 \today
\end{center}





\begin{abstract}
Hawkes processes are point processes with self-exciting and clustering properties that are popular in applications.
In recent years, renewal Hawkes processes have gained attention, due to their versatility such as the capability of capturing dependence between clusters. In this paper,
three classes of novel renewal Hawkes processes are introduced after incorporating exogenous and endogenous renewal factors.
The relationship among five related Hawkes processes is studied.
The expectations of the three classes of renewal Hawkes processes are
derived by establishing a set of integral equations.
A general common
renewal equation for these expectations is derived and further discussions are provided.
The special case of constant exogenous factor is discussed as well for the three proposed Hawkes processes.
Finally, we apply our proposed models to study the optimization problems that arise in a periodic replacement policy for systems with cascading failures
and provide numerical illustrations. The numerical solutions rely on computing the expectations of the proposed renewal Hawkes processes
that can be efficiently obtained by using the direct Riemann integration method.
\end{abstract}


\section{Introduction}\label{sec:intro}

The \textit{Hawkes process} is a simple point process $N(t)$, first introduced by Alan Hawkes in 1971 \cite{Hawkes1971}
with the stochastic intensity given by:
\begin{equation}\label{intensity}
\lambda(t)=\mu+\eta\int_{0}^{t-}h(t-s)\mathrm{d}N(s),
\end{equation}
where $\mu>0$ is the \textit{baseline intensity}, $\eta>0$ is known as the \textit{branching ratio}
and $h:\mathbb{R}_{+}\rightarrow\mathbb{R}_{+}$
is locally bounded which is known as the \textit{offspring density function} and $\eta h(t)$ is called an \textit{excitation function}
or \textit{kernel function} in the literature. We assume that $N(0)=0$, that is, the Hawkes process
starts at time zero with empty history and $\int_{0}^{\infty}h(t)\mathrm{d}t<\infty$.
Without loss of generality, we further assume $\int_{0}^{\infty}h(t)\mathrm{d}t=1$.
Br\'{e}maud and Massouli\'{e} \cite{Bremaud1996} generalized
the dynamics \eqref{intensity} by replacing the formula for the intensity \eqref{intensity}
by a nonlinear function of $\int_{0}^{t-}h(t-s)\mathrm{d}N(s)$,
and they are known as the \textit{nonlinear Hawkes processes}, whereas
the classical model \eqref{intensity} proposed by Hawkes \cite{Hawkes1971}
is often referred to as the \textit{linear Hawkes process} in the literature.

Hawkes processes are \textit{self-excited} as the occurrence of a jump increases the intensity
of the point process, and thus increases the likelihood of more future jumps.
On the other hand, the intensity declines when there are no occurrences of new jumps.
The self-exciting and clustering property makes the Hawkes process
very appealing in applications.
Since they were introduced by Alan Hawkes in 1971 \cite{Hawkes1971}, there have been many applications of Hawkes processes
in various fields, such as seismology \cite{Ogata1988, Kanazawa2021}, finance \cite{Bacry2015, Hawkes2018}, neuroscience \cite{Chornoboy1988}, insurance \cite{Swishchuk2021},
criminology \cite{Park2021}, social media \cite{Zhao2015}, reliability \cite{Ertekin2015, Cui2018}, machine learning \cite{Zhou2013, Zhang2020}. In the meantime,
there have been extensive studies on the theory of Hawkes processes; for example,  stability \cite{Bremaud1996}, immigration-birth representation \cite{Hawkes1974}, limit theorems including
the law of large numbers, central limit theorems, large deviations
\cite{Bordenave2007,Bacry2013,ZhuCLT,Zhu2014,Zhu2015,GaoZhu2021Bernoulli},
scaling limit theorems \cite{JaissonRosenbaum2015AAP,JaissonRosenbaum2016AAP},
mean-field limit \cite{Delattre2016},
regenerative properties \cite{Graham2021AAP},
and conditional uniformity \cite{Daw2024}. Because of the vast literature on Hawkes processes, here we recommend some recent related literature reviews on Hawkes processes,
  for example, the works of Worrall et al. \cite{Worrall2022}, Laub et al. \cite{Laub2024} and their book \cite{Laub2022}.

Expectations and moments are basic characteristics for both probabilistic analysis and statistical inferences of Hawkes processes; however, these measures are not easy to derive in simple closed form in general. For the linear Hawkes processes, \cite{Cui2020} presented a general method for getting their moments, and \cite{Cui2022} applied this method to derive the moments of Hawkes process with
 gamma decay functions. \cite{Daw2023} proposed a matrix method to get the moments of linear Markovian Hawkes processes. \cite{Cui2024} used the field theory to give a closed-form expression of the
 expectation for a non-linear Hawkes process. \cite{Li2020} presented a numerical method based on the Laplace transform and inverse Laplace transform for expectations of linear Hawkes processes.

The statistical inference on parameters in Hawkes processes is an interesting and important issue, the Maximum Likelihood Estimator (MLE) is the most common
technique for parametric estimation of Hawkes processes, which has been first introduced by Ogata \cite{Ogata1982}. Since then, a lot of related work has been carried out, see for example, \cite{Veen2008,Wang2025}.
Most recently, \cite{Zhang2025} presented and analyzed an away-step Frank-Wolfe method for a class of nonconvex optimization problems,
that is applied to the inference of multivariate Hawkes processes by using the MLE.

Driven by applications and the enrichment of theory, many new models have been proposed that extend the classical Hawkes process.
Dassios and Zhao \cite{Dassios2011} studied a dynamic contagion process, where in addition to the self-exciting jumps in the classical Hawkes process,
 the baseline intensity is modeled by a shot noise Poisson process with external jumps.
Zhu \cite{ZhuCIR} incorporated Hawkes jumps into the Cox-Ingersoll-Ross process, extending the classical Hawkes process with exponential kernel function
and studied the limit theorems.
Zhang \cite{Zhang2015} studied affine point processes, where a diffusion term is added to the intensity process of a Markovian Hawkes process.
Delattre and Fournier \cite{DF} studied Hawkes processes on a graph with two nodes whether or not influence each other modeled
by i.i.d. Bernoulli random variables. Chevallier \cite{Chevallier} studied a generalized Hawkes process model with an inclusion of the dependence on
the age of the process.
Dittlevsen and L\"{o}cherbach \cite{DL} considered
a multi-class systems of interacting nonlinear Hawkes processes modeling several
large families of neurons.
Chevallier et al. \cite{CDLO} studied spatially extended systems of interacting nonlinear Hawkes processes modeling large systems of neurons.
Cui et al. \cite{Cui2019} introduced a partially self-exciting Hawkes process and the maximum likelihood estimation of its parameters was provided, where the occurrence of self-excitation has a certain probability;  when this probability is one, it reduces to the classical Hawkes process.
Cui and Shen \cite{Cui2021} proposed an ephemeral self-exciting Hawkes process.

Recently, a renewal Hawkes process was proposed by \cite{Wheatley2016}, which is
more versatile than the classical Hawkes process as such processes can capture dependence between clusters in the immigration-birth representation.
An integral equation was derived for the expectation of the renewal Hawkes process; but they did not pursue the solution of the integral equation in \cite{Wheatley2016}.
\cite{Chen2018} presented likelihood evaluation methods for the univariate and multivariate
renewal Hawkes processes, respectively.
But the probabilistic analyses for the renewal Hawkes processes are yet to be explored.
Renewal Hawkes processes are a flexible class of
self-exciting point processes where the baseline intensity process follows a renewal process while still incorporating the self-exciting properties of Hawkes processes, that is to say, the renewal Hawkes processes
possess the following two important features:
(i) The events in renewal Hawkes processes can be divided into two classes: immigrants and offspring, i.e., the events result from the basic intensity $\mu (t)$ are called as immigrants, instead, the events result from the
 exciting intensity $\eta \int_0^t h(t - s){\rm{d}}N(s)$  are for offspring.
 (ii) The basic intensity process must be renewable, i.e, it can be controlled in terms of some specified rule. On the other hand, the exciting intensity cannot be controlled; it can develop as time goes by.
 It is worth mentioning that the renewal Hawkes process considered by \cite{Wheatley2016} and \cite{Chen2018} only involves a simple endogenous factor which determines the
 renewal instants for immigrant intensity or baseline intensity. As a result, the situation they considered might be too simplistic for practical needs. We are interested in filling this gap by developing novel renewal Hawkes processes
that have the versatility to model various practical scenarios.
In particular, we propose three novel renewal Hawkes processes.

First, we introduce \textit{renewal Hawkes process with exogenous variable} ($\textup{R}_1$Hawkes process) in Section~\ref{sec:R1}.
This is motivated by practical scenarios.
For example, the customer arrivals can be modeled by a Hawkes process in which the customers can be divided into two classes, one class is due to some social influence or advertisement, the other class is
due to a notice from a friend who has been a customer. However, because of changes of social influence or advertisement over time, the rule of customer arrivals for first class might be renewed. The scenario can be modeled by
the renewal Hawkes process with exogenous variable ($\textup{R}_1$Hawkes process).

Next, we introduce \textit{renewal Hawkes process with minimum of endogenous
and exogenous variables} ($\textup{R}_2$Hawkes process) in Section~\ref{sec:R2}.
This model is motivated by practical scenarios.
Consider a system that consists of a main part and an auxiliary part.
The main part suffers from some shocks depending on the external environment and needs to do a maintenance (replacement) after
a shock arrives, and the auxiliary part needs to be repaired due to the shock effects when it breakdowns. Thus the total number of maintenances (replacements and repairs) can be modeled by a renewal Hawkes process with
minimum of endogenous and exogenous variables ($\textup{R}_2$Hawkes process).

Finally, we introduce \textit{renewal Hawkes process with maximum of endogenous and exogenous variables} ($\textup{R}_3$Hawkes process) in Section~\ref{sec:R3}, which is motivated by practical scenarios.
Consider the context of group maintenance.
A set of operating devices suffer from cascade failures, but the maintenance can be done in some specified instants. If the specified instant arrives without any failure, no maintenance
is needed; otherwise, even when there are some failures, the maintenance action cannot be provided. Thus the maintenance can be done in a group when some failures occur and specified instants arrive. It can be modeled by a
renewal Hawkes process with maximum of endogenous and exogenous variables ($\textup{R}_3$Hawkes process) rather than a renewal process.

The contributions of our paper can be summarized as follows.
\begin{itemize}
\item[(i)]
We propose three novel renewal Hawkes processes: the renewal Hawkes process with exogenous variable ($\textup{R}_1$Hawkes process), the renewal Hawkes process with minimum of endogenous and exogenous variables ($\textup{R}_2$Hawkes process), and the renewal Hawkes process with maximum of endogenous and exogenous variables ($\textup{R}_3$Hawkes process). Note, the classical renewal Hawkes processes can be viewed as a renewal Hawkes process with endogenous variable.
The relationship among five related Hawkes, our three new classes of renewal Hawkes processes, the classical Hawkes process and the renewal Hawkes process in \cite{Wheatley2016}, are presented, which draws a clear boundary line for these five classes of renewal Hawkes processes.
\item[(ii)]
The expectations of three renewal Hawkes processes, including expectations of their intensity processes, are derived via a set of integral equations,
which are complicated and difficult to set up. We propose a decomposition method represented by a graph for the renewal Hawkes processes to overcome
 the difficulty to obtain the analytical solutions.
A general formula for all four classes of renewal Hawkes processes is developed via an integral equation like a classical renewal equation, and further explanations are provided.
The special case of constant exogenous factor is also discussed.
\item[(iii)]
We study an application of three renewal Hawkes processes in the context of a periodic replacement policy for systems with cascading failures,
by proposing and solving two optimization problems to obtain the optimal replacement times.
In our model, minimal repairs and replacements are nested within the entire system that will perform preventive replacements at periodic intervals. 
That is, the main subsystem has its own minimal repairs and regular renewals, whereas the minor subsystem has all minimal repairs; but the entire system will have periodic replacements. To the best of our knowledge, past research has never touched upon such issues.
The nesting of maintenance strategies in the system is more in line with the actual complex systems, especially for those with primary and auxiliary cascading failures. This nested maintenance strategy for the main subsystem is more targeted for failure prevention.
We provide numerical illustrations that rely on computing the expectations of the proposed renewal Hawkes processes
via a direct Riemann integration method to compute the expectations based on our theory.
The numerical method goes beyond our particular application and is applicable for all three proposed renewal Hawkes processes,
which is not only a computational method but also leads to an efficient analytical approximation when the time horizon is small.
The computational results based on analytical solutions build a solid baseline for simulations, which can lead to future research.
\end{itemize}

The rest of the paper is structured as follows. In Section~\ref{sec:three:types}, the definitions of three novel classes of renewal Hawkes processes are given by considering the renewal issues. The relationship among five related Hawkes processes is considered, in which four classes of renewal Hawkes processes and a classical Hawkes process are included. Expectations of the three classes of renewal Hawkes processes are derived and presented as solutions to a set of integral equations in Section~\ref{sec:expectations}, and the special case of constant exogenous factor is also studied. In Section~\ref{sec:applications}, applications of our proposed models to a periodic replacement policy for systems with cascading failures are presented, and two optimization problems arise in this context are studied, which involves numerically computing the expectations of the proposed renewal Hawkes processes, for which a numerical procedure is proposed, and numerical illustrations are provided.
Section~\ref{sec:conclusions} concludes the paper.
Technical proofs are provided in Appendix~\ref{sec:proofs}.
Further illustrations are given in Appendix~\ref{sec:illustration}.

\section{Three Classes of Renewal Hawkes Processes}\label{sec:three:types}

Before we proceed to introducing new classes of renewal Hawkes processes,
we first review the immigration-birth representation property of the classical Hawkes process.
Hawkes and Oakes \cite{Hawkes1974} first discovered that a linear Hawkes process has an immigration-birth representation.
The immigrants arrive according to a standard Poisson process with intensity $\mu>0$.
For an immigrant that arrives at time $s$, children are generated
according to a Poisson process with intensity $h(t-s)$ at time $t\geq s$,
and each child will generate children independently following the same mechanism.
The total number of children generated by the immigrant
follows a Poisson distribution with parameter $\eta>0$
and conditional on the total number of children generated by the immigrant, say $K=k$,
then the birth times of the children subtracting the arrival time
of the immigrant are distributed as the order statistics of $k$ i.i.d.
random variables with density $h(\cdot)$.
It turns out that the Hawkes process $N(t)$ is the number of all the immigrants and their descendants
that arrive on the time interval $[0,t]$.

In 2016, \cite{Wheatley2016} introduced a renewal Hawkes process, which is an extension of the classical Hawkes process by replacing the homogeneous Poisson process of the immigrant arrival process
 in the original Hawkes process by a general renewal process. The renewal times are the immigrant arriving times, and more succinctly, the renewals are only determined by the Hawkes process itself.
 The renewal Hawkes process was introduced in \cite{Wheatley2016} by considering the endogenous factor.
 Indeed, we can simultaneously or separately take the endogenous and exogenous factors into
 account for the renewal processes on the immigrant arrival processes. Thus three classes of renewal Hawkes processes are proposed in the following, in which the first class of renewal Hawkes processes is
 introduced by considering only the exogenous issue on the renewal immigrant arrival processes, and the next two classes result from considering the endogenous and exogenous factors together.

\subsection{Renewal Hawkes Process with Exogenous Variable}\label{sec:R1}

We first introduce a \textit{renewal Hawkes process with exogenous variable}, in which the renewal times depend only on the exogenous factor.
For simplicity, we call this renewal Hawkes process the $\textup{R}_1$Hawkes process in our paper.
Let $N_1(t)$ be a simple point process on the interval $[0,\infty)$, i.e., its points are all distinct. Its events can be classified into two types: immigrant and offspring events, whose types cannot be observed except that
 the first event is an immigrant event. Let $Y_1$, $Y_2,\ldots$, be the positive independent and identically distributed (i.i.d.) random variables with a distribution function $F(y)$ (or its probability density function is
 denoted by $f(y)$), and they are independent with the point process $N_1(t)$. The renewal times will be sequentially determined for the immigration events by an exogenous factor represented by  $Y_1$, $Y_2,\ldots$, through
 the following way. Let $I(t) = \max\left(i : \sum_{l=1}^{i} Y_{l} < t\right), \mathcal{F}_{1,t} = \sigma(N_1(s), 0 \leq s \leq t)$ denotes the natural filtration of $N_1(t)$.
 The point process $N_1(t)$ is called a renewal Hawkes process with exogenous variable,
 $\textup{R}_1$Hawkes process for short, if its intensity process $\left\{ \lambda_1(t), {t \geq 0}\right\}$ relative to the
 enlarged filtration $\tilde{\mathcal{F}}_{1,t} = \sigma(N_1(s), Y_1, Y_2, \ldots, Y_{I(t)}, 0 \leq s \leq t)$ takes the following form
\begin{equation}\label{lambda:1:defn}
\begin{aligned}
 \lambda_{1}(t) &  = \frac{\mathbb{E}[\mathrm{d}{N}_{1}(t)|\tilde{\mathcal{F}}_{1,t^{-}}]}{\mathrm{d}t}=\frac{\mathbb{E}[\mathrm{d}{N}_{1}(t)|\mathcal{F}_{1,t^{-}},Y_{1},Y_{2},\ldots ,Y_{I(t)}]}{\mathrm{d}t} \\
  &=\mu \left(t-\sum\nolimits_{i=1}^{I(t)}Y_{i}\right)+\eta \int_{0}^{t-}h(t-v)\mathrm{d}N_{1}(v),
\end{aligned}
\end{equation}
for a positive function $\mu(t)$, $t \in [0, \infty)$, $\eta > 0$ and positive function $h(t)$, $t \in [0, \infty)$, such that $\int_{0}^{\infty} h(t) \mathrm{d}t = 1$, and we can also re-write the intensity function of $\textup{R}_1$Hawkes process in \eqref{lambda:1:defn} as
\begin{equation}
   {\lambda }_{1}(t) =\mu \left(t-\sum\nolimits_{i=1}^{I(t)}{{{Y}_{i}}}\right)+\eta \sum_{i=1,\tau _{i}^{(1)}<t}^{{{N}_{1}}(t)}h\left(t-\tau _{i}^{(1)}\right),
\end{equation}
where $\tau_1^{(1)} < \tau_2^{(1)} < \cdots < \tau_{N_1(t)}^{(1)}$ are the event times
of $N_1(t)$ in the interval $[0,t]$.

Let $T_1=\tau_1^{(1)}$ be the first immigration virtual event time, and in general, let $T_i, i=1,2,...$ be the successive immigration virtual event times, which come from the renewal points for a renewal process with an intensity function $\mu(t)$, i.e., the immigration event may occur or not which depends upon the occurrence or not of event $\{ Y_1 - T_1 > 0 \}$. More precisely, if $Y_1>T_1$, then the intensity function of immigration events is renewed at time $Y_1$ (the immigrant event happened); otherwise, the immigration event does not happen by the renewal time $Y_1$, this is why we call $T_1$ as a virtual time. As a result, we call the immigrant event times $T_1<T_2<\cdots$ immigration virtual event times. The function $\mu(t)$ is called a background event intensity; it can be thought of as the failure rate function of immigrant event virtual time $T_1$ which is a random variable with probability distribution function:
\begin{equation}
G(t) = 1 - e^{-\int_{0}^{t} \mu(v) \mathrm{d}v},
\end{equation}
whose probability density function is given by $g(t) = \mu(t) e^{-\int_{0}^{t} \mu(v) \mathrm{d}v}$.
We recall from \eqref{intensity} that the positive constant parameter $\eta$ is called the branching ratio, and the function $h(t)$ is called the offspring density function, $\eta h(t)$ is called an excitation function.

\subsection{Renewal Hawkes Process with Minimum of Endogenous and Exogenous Variables}\label{sec:R2}

Next, we introduce a \textit{renewal Hawkes process with minimum of endogenous and exogenous variables},
in which the renewal times are determined by the minimum of endogenous and exogenous variables..
For simplicity, we call this renewal Hawkes process $\textup{R}_2$Hawkes process in our paper,
Let $N_2(t)$ be a simple point process on the interval $[0,\infty)$, i.e., its points are all distinct, $Y_1$, $Y_2,\ldots$, be the positive independent and identically distributed random variables with a distribution function $F(y)$ (or its probability density function is denoted by $f(y)$), and they are independent with the point process $N_2(t)$. The renewal times will be sequentially determined for the immigration events by exogenous and endogenous factors together through the following way. The immigration event is firstly renewed at time
\begin{equation}
V_1 = \min(Y_i, T_1),
\end{equation}
where $Y_i$'s and $T_i$'s are the same as defined in Section~\ref{sec:R1}.
The second renewal time for immigration event is recursively defined by
\begin{equation}
V_2 = \min(Y_2+V_1, T_2),
\nonumber
\end{equation}
where $T_2$ is the first immigration virtual event time after the renewal point $V_1$, and similarly, it represents that the first immigration event after the renewal point $V_1$ that may or may not occur.

In general, $T_i$ is the first immigration virtual event time after the renewal point $V_{i-1}$, the $i$-th renewal time is recursively defined by
\begin{equation}
V_i = \min(Y_i + V_{i-1}, T_i), \qquad\text{for every $i\geq 2$.}
\end{equation}

Let $\textit{J}(t)=\max(i:V_i<t)$ and $\mathcal{F}_{2,t}=\sigma(N_2(s), 0 \leq s \leq t)$ denote the natural filtration of $N_2(t)$. The point process $N_2(t)$ is called a renewal Hawkes process with minimum of endogenous
and exogenous variables, $\textup{R}_2$Hawkes process for short, if its intensity process $\{\lambda_2(t), t \geq 0\}$ relative to the enlarged filtration ${\tilde{\mathcal{F}}_{2,t}} = \sigma(N_2(s), V_{J(s)}, 0 \leq s \leq t)$ takes the following form
\begin{equation}\label{lambda:2:defn}
\begin{aligned}
\lambda_2(t) & = \frac{\mathbb{E}[\mathrm{d}N_2(t) \mid \tilde{\mathcal{F}}_{2,t^{-}}]}{\mathrm{d}t} = \frac{\mathbb{E}[\mathrm{d}N_2(t) \mid \mathcal{F}_{2,t^{-}}, V_{J(t)}]}{\mathrm{d}t} \\
& = \mu\left(t - V_{J(t)}\right) + \eta \int_{0}^{t-} h(t - v) \mathrm{d}N_2(v),
\end{aligned}
\end{equation}
for a positive function  $\mu(t), t \in [0, \infty)$,
$\eta > 0$,
and positive function $h(t)$, $t \in [0, \infty)$, such that $\int_{0}^{\infty} h(t) \mathrm{d}t = 1$,
and we can also re-write the intensity function of $\textup{R}_2$Hawkes process in \eqref{lambda:2:defn}  as
\begin{equation}
\lambda_2(t) = \mu\left(t - V_{J(t)}\right) + \eta \sum_{i=1, \tau_i^{(2)} < t}^{N_2(t)} h\left(t - \tau_i^{(2)}\right),
\end{equation}
where $\tau_1^{(2)} < \tau_2^{(2)} < \cdots < \tau_{N_2(t)}^{(2)}$ are event times of $N_2(t)$ in the interval $[0,t]$, which may include some of $T_1, T_2, \ldots, T_{J(t)}$.

\subsection{Renewal Hawkes Process with Maximum of Endogenous and Exogenous Variables}\label{sec:R3}

Next, we introduce a \textit{renewal Hawkes process with maximum of endogenous and exogenous variables}, in which the renewal times for the immigration events are determined by the maximum of endogenous and exogenous variables.
For simplicity, we call this renewal Hawkes process the $\textup{R}_3$Hawkes process in our paper.
In the definition of $\textup{R}_2$Hawkes process, $V_i(i\geq1)$ are used to denote the renewal times for the intensity function of immigration events, in which a minimal operator is used for determining these renewal times. In the definition of a $\textup{R}_3$Hawkes process, a maximal operator will be used for determination of renewal times of the intensity function of immigration events. More precisely, let $N_3(t)$ be a simple point process on the interval $[0,\infty)$ and $W_i(i\geq1)$ be the renewal times which are recursively defined by
\begin{equation}
W_0 = 0, \qquad W_i=\max(Y_i+W_{i-1}, T_i), \qquad\text{for every  $i \geq 1$},
\end{equation}
where $Y_i$'s and $T_i$'s are the same as defined in Section~\ref{sec:R1}.

Let $K(t)=\max(i:W_i<t)$ and $\mathcal{F}_{3,t}=\sigma(N_3(s),0\leq s \leq t)$ denote the natural filtration of $N_3(t)$. The point process $N_3(t)$ is called a renewal Hawkes process with maximum of endogenous
and exogenous variables, $\textup{R}_3$Hawkes process for short, if its intensity process $\{\lambda_3(t),t \leq 0\}$ relative to the enlarged filtration $\tilde{\mathcal{F}}_{3,t} = \sigma(N_3(s), W_{K(s)}, 0 \leq s \leq t)$ takes the following form
\begin{equation}\label{lambda:3:defn}
\begin{aligned}
\lambda_3(t) & = \frac{\mathbb{E}[\mathrm{d}N_3(t) \mid \tilde{\mathcal{F}}_{3,t^{-}}]}{\mathrm{d}t} = \frac{\mathbb{E}[\mathrm{d}N_3(t) \mid \mathcal{F}_{3,t^{-}}, W_{K(t)}]}{\mathrm{d}t} \\
& = \mu\left(t - W_{K(t)}\right) + \eta \int_{0}^{t-} h(t - v) \mathrm{d}N_3(v),
\end{aligned}
\end{equation}
and we can also re-write the intensity function of $\text{R}_3$Hawkes process in \eqref{lambda:3:defn} as
\begin{equation}
\lambda_3(t) = \mu\left(t - W_{K(t)}\right) + \eta \sum_{i=1, \tau_i^{(3)} < t}^{N_3(t)} h\left(t - \tau_i^{(3)}\right),
\end{equation}
where $\tau_1^{(3)} < \tau_2^{(3)} < \cdots < \tau_{N_3(t)}^{(3)}$ are event times of $N_3(t)$ in the interval $[0,t]$, which may include some of $T_1, T_2, \ldots, T_{K(t)}$.

To help understand and illustrate the three proposed classes of renewal Hawkes processes, we will plot their realized intensity function trajectories
through some examples and provide discussions in Appendix~\ref{sec:illustration}.

\subsection{Relationship Among Models}

Since three novel classes of renewal Hawkes processes are introduced in Sections~\ref{sec:R1}, \ref{sec:R2} and \ref{sec:R3}, a question is naturally raised regarding the relationship among the models proposed by \cite{Wheatley2016}, \cite{Hawkes1971} and our proposed models. 
In this section, we study the relationship among these models in detail.
There are five related Hawkes processes to be taken into consideration, which are: (i) the classical Hawkes process $N(t)$, denoted as HP, i.e., with intensity $\mu(t) + \int_0^{t-} \eta h(t-v) \mathrm{d}N(v)$, (ii) the renewal Hawkes process proposed in \cite{Wheatley2016}, denoted as $\textup{R}_{\mathrm{WFS}}$HP, (iii) $\textup{R}_1$Hawkes process introduced in the paper, denoted as $\textup{R}_1$HP, (iv) $\textup{R}_2$Hawkes process introduced in the paper, denoted as $\textup{R}_2$HP, and (v) $\textup{R}_3$Hawkes process introduced in the paper, denoted as $\textup{R}_3$HP.

We shall study the relationship among the five Hawkes processes according to the following four cases/conditions.

\begin{itemize}
\item Case (i): $\mu(t)=$ constant.

It is easy to know that $\textup{HP} = \textup{R}_{\mathrm{WFS}}\text{HP} = \text{R}_1\text{HP} = \text{R}_2\text{HP} = \text{R}_3\text{HP}$, because all renewal points are for the background intensity function.

\item Case (ii): $Y_1=Y_2=\cdots=0$, i.e., $F(0)=1$.

We can see that $\text{R}_1\text{HP}$ and $\text{R}_2\text{HP}$ are ill-posed, whereas $\text{R}_{\mathrm{WFS}}\text{HP} = \text{R}_\text{3}\text{HP}$.

\item Case (iii): $Y_1=Y_2=\cdots=\infty$, i.e., $F(t)=0$ for $t\in [0,\infty)$.

Similarly, we can get that $\text{R}_3\text{HP}$ are ill-posed, whereas $\text{R}_1\text{HP} = \text{HP}$, $\text{R}_{\mathrm{WFS}}\text{HP} = \text{R}_2\text{HP}$.

\item Case (iv): $Y_1=Y_2=\cdots= T_1$, i.e., $ \frac{f(t)}{1-F(t)}=\mu(t)$ for $ t\in[0,\infty) $.

It is clear that $\text{R}_1\text{HP} = \text{R}_2\text{HP} = \text{R}_3\text{HP} = \text{R}_{\mathrm{WFS}}\text{HP}$, because all renewals for the background intensity function in a classical Hawkes process are at the times of immigration events.
\end{itemize}

Based on the discussions on four cases above, we summarize the relationship among the five Hawkes processes in Figure~\ref{fig:hawkes_processes}.

\begin{figure}[htbp]
    \centering
    \includegraphics[width=0.8\textwidth]{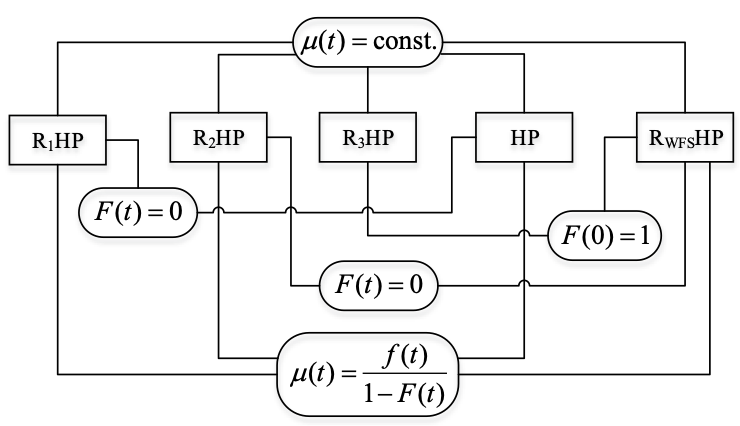}
    \caption{The schematic diagram for the relationship representation of 5 related Hawkes processes. Note: Rectangles are for related Hawkes processes, and the oval rectangles are for conditions. Any two related Hawkes processes connected via lines through a condition rectangle are equivalent under the condition.}
    \label{fig:hawkes_processes}
\end{figure}

\section{Expectations of Renewal Hawkes Processes}\label{sec:expectations}

The expectation is one of the important features of point processes, which cannot only be used in the related statistical inference, but also can be used in probabilistic analysis on the point processes. The expectations of the proposed renewal Hawkes processes in Section~\ref{sec:three:types} will be given by some integral equations, and the corresponding expectations of intensity functions can also be derived based on these expectations of renewal Hawkes processes by using Theorem~1 in \cite{Cui2020}.

\subsection{$\textup{R}_1$Hawkes Process}\label{sec:R1:expectation}

Let $m_1(t)=\mathbb{E}[N_1(t)]$ and $0<U_1<U_2<$$\textellipsis$ be the renewal times for $\textup{R}_1$Hawkes process $N_1(t)$, and more precisely, $U_1 = Y_1$, $U_2 = U_1 + Y_2$, $U_3 = U_2 + Y_3$, $\ldots$. The expectations of the $\textup{R}_1$Hawkes process and its intensity process will be derived based on two-step conditioning on $U_1$, $Y_1$ and $T_1$, respectively. We have the following result.

\begin{theorem}\label{thm:1}
For every $t\geq 0$,
\begin{equation}\label{m:1:eqn}
m_1(t) = \Psi_1(t) + \int_0^t m_1(t-y) n_1(y) \, \mathrm{d}y,
\end{equation}
where $ n_1(y) = f(y) $ and
\begin{equation}\label{Psi:1:eqn}
\begin{aligned}
\Psi_1(t) &= \int_0^t G(y) f(y) \, \mathrm{d}y + \int_0^t k_{1,0}(t-\tau) F_g(\tau) \, \mathrm{d}\tau - [1 - F(t)] \int_0^t k_{1,0}(t-\tau) g(\tau) \, \mathrm{d}\tau \\
&\quad + \int_0^t \int_0^y k_{T_1,\tau}(y-\tau) g(\tau) f(y) \, \mathrm{d}\tau \, \mathrm{d}y + \int_0^t \int_0^y k_{1,y,0,\tau}(t-y) g(\tau) f(y) \, \mathrm{d}\tau \, \mathrm{d}y  \\
&\quad\quad + m(t)\int_t^\infty G(y) f(y) \, \mathrm{d}y,
\end{aligned}
\end{equation}
where
$F_g(y) := [1 - F(y)] g(y)$, and $k_{1,0}(t)$,
$k_{T_1,\tau}(t)$, $k_{1,y,0,\tau}(t)$ satisfy the integral equations:
\begin{equation}
\left\{
\begin{aligned}\label{k:equations}
k_{1,0}(t) &= \int_0^t [1 + k_{1,0}(t-v)] \eta h(v) \, \mathrm{d}v, \quad \text{with } k_{1,0}(0) = 0, \\
k_{T_1,\tau}(t) &= \int_0^t \mu(v+\tau) \, \mathrm{d}v + \int_0^t \eta h(v) k_{T_1,\tau}(t-v) \, \mathrm{d}v, \quad \text{with } k_{T_1,\tau}(0) = 0, \\
k_{1,y,0,\tau}(t) &= \eta H(t) k_{T_1,\tau}(y-\tau) - \int_0^{y-\tau} \eta [h(y-\tau-v) - h(t+y-\tau-v)] k_{T_1,\tau}(v) \, \mathrm{d}v \\
&\quad + \int_0^t \eta h(t-v) k_{1,y,0,\tau}(v) \, \mathrm{d}v, \quad \text{with } k_{1,y,0,\tau}(0) = 0,
\end{aligned}
\right.
\end{equation}
where $H(t):=\int_{0}^{t}h(v)\mathrm{d}v$ and $m(t)$ satisfies the integral equation:
\begin{equation}\label{m:t:eqn}
m(t) = \int_0^t \mu(v) \, \mathrm{d}v + \int_0^t \eta h(v) m(t-v) \, \mathrm{d}v.
\end{equation}
\end{theorem}

\begin{remark}\label{remark:1}
The expectation of the intensity function of the $\text{R}_1$Hawkes process can be given based on Theorem~1 in \cite{Cui2020}, which is
$\mathbb{E}[\lambda_1(t)] = \frac{\mathrm{d}}{\mathrm{d}t} m_1(t)$.
It follows from Theorem~\ref{thm:1} that $\mathbb{E}[\lambda_1(t)]$ satisfies the integral equation:
\begin{equation}
\mathbb{E}[\lambda_1(t)]  = \Psi'_1(t) + \int_0^t \mathbb{E}[\lambda_1(t-y)] n_1(y) \, \mathrm{d}y.
\end{equation}
\end{remark}

Next, let us consider a special case $f(y)=\gamma e^{-\gamma y}$ for some $\gamma>0$,
that is $Y_{1},Y_{2},\ldots$ follow an exponential distribution. For this special case,
we have the following result.

\begin{corollary}\label{cor:1}
Assume $f(y)=\gamma e^{-\gamma y}$ for some $\gamma>0$.
Then, for any $t\geq 0$,
\begin{equation}
m_{1}(t)=\Psi_{1}(t)+\gamma\int_{0}^{t}\Psi_{1}(s)\mathrm{d}s.
\end{equation}
\end{corollary}


Furthermore, let us consider the special case for the offspring density function $h(t)=\beta e^{-\beta t}$ for some $\beta>0$.
In the classical Hawkes process, this corresponds to the case when the intensity process
is Markovian.
We obtain the following corollary.

\begin{corollary}\label{cor:1:exp}
Let the offspring density function $h(t)=\beta e^{-\beta t}$ for some $\beta>0$. Then, for any $t\geq 0$,
\begin{equation}
k_{1,0}(t)=\frac{\eta(e^{\beta(\eta-1)t}-1)}{\eta-1},
\end{equation}
and
\begin{equation}
k_{1,y,0,\tau}(t)=\frac{\eta(e^{\beta(\eta-1)t}-1)}{\eta-1}\left[k_{T_1,\tau}(y-\tau) +\int_0^{y-\tau} \eta\beta^{2}e^{-\beta(y-\tau-v)} k_{T_1,\tau}(v) \, \mathrm{d}v\right],
\end{equation}
where
\begin{equation}
k_{T_1,\tau}(t) =\int_{0}^{t}e^{(\eta-1)\beta(t-s)}\left[\mu(s+\tau)+\eta\beta\int_0^s \mu(v+\tau) \, \mathrm{d}v\right]\mathrm{d}s.
\end{equation}
Moreover,
\begin{equation}
m(t)=\int_{0}^{t}e^{(\eta-1)\beta(t-s)}\left[\mu(s)+\beta\int_{0}^{s}\mu(v)\mathrm{d}v\right]\mathrm{d}s.
\end{equation}
\end{corollary}

In Corollary~\ref{cor:1:exp}, we consider the special case for the offspring density function $h(t)=\beta e^{-\beta t}$ for some $\beta>0$.
Next, we specialize the model further to consider $f(y)=\gamma e^{-\gamma y}$
for some $\gamma>0$ and $\mu(t)=\mu_{i}$ for every $t_{i-1}\leq t<t_{i}$. Then, we obtain the following corollary.

\begin{corollary}\label{cor:1:exp:special}
Let $h(t)=\beta e^{-\beta t}$ for some $\beta>0$,
$f(y)=\gamma e^{-\gamma y}$
for some $\gamma>0$ and $\mu(t)=\mu_{i}$ for every $t_{i-1}\leq t<t_{i}$.
Then, When $y<t_{i_{\ast}}$, we have
\begin{align}
k_{1,y,0,\tau}(t)&=\frac{\eta(e^{\beta(\eta-1)t}-1)}{\eta-1}(1-\beta)
\nonumber
\\
&\qquad\qquad\cdot\left[\frac{\eta\mu_{t_{i_{\ast}}}(e^{(\eta-1)\beta(y-\tau)}-1)}{(\eta-1)\beta}+\eta\beta\mu_{t_{i_{\ast}}}
\left(\frac{y-\tau}{(1-\eta)\beta}-\frac{1-e^{(\eta-1)\beta(y-\tau)}}{(1-\eta)^{2}\beta^{2}}\right)\right]
\nonumber
\\
&\qquad
+\frac{\eta(e^{\beta(\eta-1)t}-1)}{\eta-1}\eta\beta(y-\tau)\mu_{i_{\ast}}.
\end{align}
When $t_{i}\leq y<t_{i+1}$ for some $i\geq i_{\ast}$, we have
\begin{align}
&k_{1,y,0,\tau}(t)=\frac{\eta(e^{\beta(\eta-1)t}-1)}{\eta-1}(1-\beta)
\Bigg\{
\eta\mu_{t_{i_{\ast}}}e^{(\eta-1)\beta(y-\tau)}\frac{e^{(1-\eta)\beta(t_{i_{\ast}}-\tau)}-1}{(1-\eta)\beta}
\nonumber
\\
&+\eta\beta\mu_{t_{i_{\ast}}}e^{(\eta-1)\beta(y-\tau)}\left(\frac{(t_{i_{\ast}}-\tau)e^{(1-\eta)\beta(t_{i_{\ast}}-\tau)}}{(1-\eta)\beta}-\frac{e^{(1-\eta)\beta(t_{i_{\ast}}-\tau)}-1}{(1-\eta)^{2}\beta^{2}}\right)
\nonumber
\\
&\qquad
+\sum_{j=i_{\ast}}^{i-1}e^{(\eta-1)\beta(y-\tau)}\frac{e^{(1-\eta)\beta(t_{j+1}-\tau)}-e^{(1-\eta)\beta(t_{j}-\tau)}}{(1-\eta)\beta}
\nonumber
\\
&\qquad\qquad\qquad\cdot\left[\eta\mu_{j+1}+\eta\beta(t_{i_{\ast}}-\tau)\mu_{i_{\ast}}
+\eta\beta\sum_{\ell=i_{\ast}+1}^{j}(t_{\ell}-t_{\ell-1})\mu_{\ell}\right]
\nonumber
\\
&\qquad\qquad
+\sum_{j=i_{\ast}}^{i-1}e^{(\eta-1)\beta(y-t_{j})}\eta\beta\mu_{j+1}
\left[\frac{(t_{j+1}-t_{j})e^{(1-\eta)\beta(t_{j+1}-t_{j})}}{(1-\eta)\beta}
-\frac{e^{(1-\eta)\beta(t_{j+1}-t_{j})}-1}{(1-\eta)^{2}\beta^{2}}\right]
\nonumber
\\
&\qquad\qquad\qquad
+\frac{1-e^{(1-\eta)\beta(t_{i}-y)}}{(1-\eta)\beta}\left[\eta\mu_{i+1}+\eta\beta(t_{i_{\ast}}-\tau)\mu_{i_{\ast}}
+\eta\beta\sum_{\ell=i_{\ast}+1}^{i}(t_{\ell}-t_{\ell-1})\mu_{\ell}\right]
\nonumber
\\
&\qquad\qquad\qquad\qquad
+e^{(\eta-1)\beta(y-t_{i})}\eta\beta\mu_{i+1}
\left[\frac{(y-\tau-t_{i})e^{(1-\eta)\beta(y-\tau-t_{i})}}{(1-\eta)\beta}
-\frac{e^{(1-\eta)\beta(t-t_{i})}-1}{(1-\eta)^{2}\beta^{2}}\right]\Bigg\}
\nonumber
\\
&\qquad
+\frac{\eta(e^{\beta(\eta-1)t}-1)}{\eta-1}\eta\beta\left[(t_{i_{\ast}}-\tau)\mu_{i_{\ast}}+\sum_{j=i_{\ast}}^{i-1}(t_{j+1}-t_{j})\mu_{j+1}
+(y-t_{i})\mu_{i+1}\right],
\end{align}
where $i_{\ast}$ is the unique positive value such that $t_{i_{\ast}-1}\leq\tau<t_{i_{\ast}}$.

Moreover, for any $t_{i}\leq t<t_{i+1}$,
\begin{align}
m(t)
&=e^{(\eta-1)\beta t}\sum_{j=1}^{i}\left[\mu_{j}+\beta\sum_{\ell=1}^{j-1}\mu_{\ell}(t_{\ell}-t_{\ell-1})\right]\frac{e^{(1-\eta)\beta t_{j}}-e^{(1-\eta)\beta t_{j-1}}}{(1-\eta)\beta}
\nonumber
\\
&\qquad
+\frac{e^{(\eta-1)\beta t}}{\beta(\eta-1)^{2}}\sum_{j=1}^{i}\mu_{j}\left[1+(\eta-1)\beta t_{j-1}-\left[1+(\eta-1)\beta t_{j}\right] e^{(1-\eta)\beta (t_{j}-t_{j-1})}\right]
\nonumber
\\
&\qquad\qquad
+e^{(\eta-1)\beta t}\left[\mu_{i+1}+\beta\sum_{\ell=1}^{i}\mu_{\ell}(t_{\ell}-t_{\ell-1})\right]
\frac{e^{(1-\eta)\beta t}-e^{(1-\eta)\beta t_{i}}}{(1-\eta)\beta}
\nonumber
\\
&\qquad\qquad\qquad
+\frac{e^{(\eta-1)\beta t}}{\beta(\eta-1)^{2}}\mu_{i+1}\left[1+(\eta-1)\beta t_{i}-\left[1+(\eta-1)\beta t\right] e^{(1-\eta)\beta (t-t_{i})}\right].
\end{align}
\end{corollary}

\subsection{$\textup{R}_2$Hawkes Process}\label{sec:R2:expectation}

The expectation of the $\textup{R}_2$Hawkes process will also be derived based on two-step conditioning on $V_1$, $Y_1$ and $T_1$, respectively.
First, consider the expectation of the $\textup{R}_2$Hawkes process. Let $m_2(t)=\mathbb{E}[N_2(t)]$ be the expectation of the $\textup{R}_2$Hawkes process.
We have the following result.

\begin{theorem}\label{thm:2}
For every $t\geq 0$, \( m_2(t) \) satisfies the following integral equation:
\begin{equation}\label{m:2:eqn}
m_2(t) =\Psi_2(t)+ \int_0^t m_2(t-v)n_2(v)\mathrm{d}v,
\end{equation}
where $\Psi_2(t) := \int_0^t [1 + k_{2,0}(t-y)] F_g(y) \, \mathrm{d}y$ and $n_2(y) := G_f(y) + F_g(y)$,
with $G_f(y) := [1 - G(y)] f(y)$ and $F_g(y):=[1-F(y)]g(y)$ for every $y\geq 0$,
where $k_{2,0}(t)$ satisfies the integral equation:
\begin{equation}
k_{2,0}(t) = \int_0^t [1 + k_{2,0}(t-v)] \eta h(v) \mathrm{d}v.
\end{equation}
Moreover, the Laplace transform of $m_2(t)$ is given by
\begin{equation}
m_2^*(s) = \frac{F_g^*(s)}{s[1 - \eta h^*(s)][1 - G_f^*(s) - F_g^*(s)]}.
\end{equation}
\end{theorem}

\begin{remark}\label{remark:2}
We have
\begin{equation}
\mathbb{E}[\lambda_2(t)] = \frac{\mathrm{d}}{\mathrm{d}t} m_2(t),
\end{equation}
or the form of Laplace transform
\begin{equation}
\ell\{\mathbb{E}[\lambda_2(t)]\}(s) = s m_2^*(s),
\end{equation}
where \( \ell\{\mathbb{E}[\lambda_2(t)]\}(s) = \int_0^\infty e^{-st} \mathbb{E}[\lambda_2(t)] \mathrm{d}t \).
\end{remark}

Let us consider the special case $f(y)=\gamma e^{-\gamma y}$
for some $\gamma>0$. That is, $Y_{1},Y_{2},\ldots$ are exponentially distributed.
We also assume that $\mu(t)=\mu_{i}$ for every $t_{i-1}\leq t<t_{i}$.
That is $\mu(t)$ is piecewise constant.
We have the following result.

\begin{corollary}\label{cor:2}
Assume $f(y)=\gamma e^{-\gamma y}$
for some $\gamma>0$ and $\mu(t)=\mu_{i}$ for every $t_{i-1}\leq t<t_{i}$.
Then, for any $t_{i}\leq t<t_{i+1}$:
\begin{equation}
m_{2}(t)=e^{(1-\alpha_{i})(\gamma+\mu_{i+1})(t_{i}-t)}m_{2}(t_{i})
+\int_{t_{i}}^{t}(\Psi'_{2}(s)+(\gamma+\mu_{i+1})\Psi_{2}(s))e^{(1-\alpha_{i})(\gamma+\mu_{i+1})(s-t)}\mathrm{d}s,
\end{equation}
for any $i=0,1,2,\ldots$ with $m_{2}(0)=0$.
\end{corollary}

\begin{remark}
Let us consider the special case for the offspring density function $h(t)=\beta e^{-\beta t}$ for some $\beta>0$.
In the classical Hawkes process, this corresponds to the case when the intensity process
is Markovian. It is easy to see that $k_{2,0}(t)=k_{1,0}(t)$ such that
\begin{equation}
k_{2,0}(t)=\frac{\eta(e^{\beta(\eta-1)t}-1)}{\eta-1},
\end{equation}
with the understanding that $k_{2,0}(t)=\beta t$ with $\eta=1$.
\end{remark}

\begin{remark}
Consider the special case for the offspring density function $h(t)=\beta e^{-\beta t}$ for some $\beta>0$.
Assume $f(y)=\gamma e^{-\gamma y}$
for some $\gamma>0$ and $\mu(t)=\mu_{i}$ for every $t_{i-1}\leq t<t_{i}$.
Then, for any $t_{i}\leq t<t_{i+1}$:
\begin{equation}
F_{g}(t)=[1-F(t)]g(t)=\alpha_{i}\mu_{i+1}e^{-(\gamma+\mu_{i+1})t}.
\end{equation}
Therefore, for any $t_{i}\leq t<t_{i+1}$:
\begin{align}
\Psi_{2}(t)
&=\sum_{j=1}^{i}\int_{t_{j-1}}^{t_{j}}[1+k_{2,0}(t-y)]F_{g}(y)\mathrm{d}y
+\int_{t_{i}}^{t}[1+k_{2,0}(t-y)]F_{g}(y)\mathrm{d}y
\nonumber
\\
&=\sum_{j=1}^{i}\int_{t_{j-1}}^{t_{j}}\frac{\eta e^{\beta(\eta-1)(t-y)}-1}{\eta-1}\alpha_{j-1}\mu_{j}e^{-(\gamma+\mu_{j})y}\mathrm{d}y
+\int_{t_{i}}^{t}\frac{\eta e^{\beta(\eta-1)(t-y)}-1}{\eta-1}\alpha_{i}\mu_{i+1}e^{-(\gamma+\mu_{i+1})y}\mathrm{d}y
\nonumber
\\
&=\sum_{j=1}^{i}\frac{\eta\alpha_{j-1}\mu_{j}e^{\beta(\eta-1)t}}{(\eta-1)(\gamma+\mu_{j}+\beta(\eta-1))}
\left(e^{-(\gamma+\mu_{j}+\beta(\eta-1))t_{j-1}}-e^{-(\gamma+\mu_{j}+\beta(\eta-1))t_{j}}\right)
\nonumber
\\
&\qquad\qquad
-\sum_{j=1}^{i}\frac{\alpha_{j-1}\mu_{j}}{(\eta-1)(\gamma+\mu_{j})}\left(e^{-(\gamma+\mu_{j})t_{j-1}}-e^{-(\gamma+\mu_{j})t_{j}}\right)
\nonumber
\\
&\qquad\qquad\qquad
+\frac{\eta\alpha_{i}\mu_{i+1}e^{\beta(\eta-1)t}}{(\eta-1)(\gamma+\mu_{i+1}+\beta(\eta-1))}
\left(e^{-(\gamma+\mu_{i+1}+\beta(\eta-1))t_{i}}-e^{-(\gamma+\mu_{i+1}+\beta(\eta-1))t}\right)
\nonumber
\\
&\qquad\qquad\qquad\qquad
-\frac{\alpha_{i}\mu_{i+1}}{(\eta-1)(\gamma+\mu_{i+1})}\left(e^{-(\gamma+\mu_{i+1})t_{i}}-e^{-(\gamma+\mu_{i+1})t}\right).
\end{align}
\end{remark}

\subsection{$\textup{R}_3$Hawkes Process}\label{sec:R3:expectation}

For the $\textup{R}_3$Hawkes process, let $m_3(t)=\mathbb{E}[N_3(t)]$ be the expectation of the $\textup{R}_3$Hawkes process.
We have the following result.

\begin{theorem}\label{thm:3}
For every $t\geq 0$,
\begin{equation}\label{m:3:eqn}
m_3(t) = \Psi_3(t) + \int_0^t m_3(t-y) n_3(y) \mathrm{d}y,
\end{equation}
where $n_3(y) = G(y) f(y) + F(y) g(y)$ and
\begin{equation}\label{Psi:3:eqn}
\begin{aligned}
\Psi_3(t) &= \int_0^t G(y) f(y) \mathrm{d}y + \int_0^t k_{3,0}(t-\tau) F_g(\tau) \mathrm{d}\tau - [1 - F(t)] \int_0^t k_{3,0}(t-\tau) g(\tau) \mathrm{d}\tau \\
&\quad + \int_0^t  \int_0^y k_{T_1,\tau}(y-\tau) g(\tau) f(y) \mathrm{d}\tau \mathrm{d}y + \int_0^t \int_0^y k_{3,y,0,\tau} (t-y) g(\tau) f(y) \, \mathrm{d}\tau \, \mathrm{d}y \\
&\quad\quad + m(t) \int_t^\infty G(y) f(y) \mathrm{d}y,
\end{aligned}
\nonumber
\end{equation}
where $k_{3,0}(t)$, $k_{1,0}(t)$ and $k_{3,y,0,\tau}(t)$ satisfy the following set of integral equations:
\begin{equation}
\begin{cases}
k_{3,0}(t) = \int_0^t [1 + k_{3,0}(t-v)] \eta h(v) \mathrm{d}v, \\
k_{T_1,\tau}(t) = \int_0^t \mu(v+\tau) \mathrm{d}v + \eta \int_0^t k_{T_1,\tau}(t-v) h(v) \mathrm{d}v, \\
k_{3,y,0,\tau}(t) = \eta H(t) k_{T_3,\tau}(y-\tau) - \int_0^{y-\tau} \eta [h(y-\tau-v) - h(t+y-\tau-v)] k_{T_1,\tau}(v) \, \mathrm{d}v \\
\quad\quad + \int_0^t \eta h(t-v) k_{3,y,0,\tau}(v) \, \mathrm{d}v, \quad \text{with } k_{3,y,0,\tau}(0) = 0,
\end{cases}
\end{equation}
with $k_{3,0}(0) = k_{T_1,\tau}(0) = k_{3,y,0,\tau}(0) = 0$,
and $m(t)$ satisfies the integral equation \eqref{m:t:eqn}.
\end{theorem}

\begin{remark}\label{remark:3}
We have
$\mathbb{E}[\lambda_3(t)] = \frac{\mathrm{d}}{\mathrm{d}t} m_3(t)$
which satisfies the integral equation:
\begin{equation}
\mathbb{E}[\lambda_3(t)] = \Psi'_3(t) + \int_0^t \mathbb{E}[\lambda_3(t-y)]  n_3(y) \mathrm{d}y.
\end{equation}
\end{remark}

Let us consider the special case $f(y)=\gamma e^{-\gamma y}$
for some $\gamma>0$. That is, $Y_{1},Y_{2},\ldots$ are exponentially distributed.
We also assume that $\mu(t)=\mu_{i}$ for every $t_{i-1}\leq t<t_{i}$.
That is $\mu(t)$ is piecewise constant.
Let us define:
\begin{equation}
M_{3}(t):=
\left[\begin{array}{c}
m_{3}(t)
\\
m'_{3}(t)
\\
m''_{3}(t)
\end{array}\right],
\qquad
\qquad
U_{3}(t):=
\left[\begin{array}{c}
\Psi_{3}(t)
\\
\Psi'_{3}(t)
\\
\Psi''_{3}(t)
\end{array}\right].
\end{equation}
Then, we have the following result.

\begin{corollary}\label{cor:3}
Assume that $f(y)=\gamma e^{-\gamma y}$
for some $\gamma>0$ and $\mu(t)=\mu_{i} $ for every $t_{i-1}\leq t<t_{i}$.
Then, for any $t_{i}\leq t<t_{i+1}$,
\begin{align}
M_{3}(t)&=e^{(B_{i}+D_{i}C_{i}^{-1}-D_{i}C_{i}^{-1}A_{i})(t-t_{i})}M_{3}(t_{i})
\nonumber
\\
&\qquad\qquad\qquad
+\int_{t_{i}}^{t}e^{(B_{i}+D_{i}C_{i}^{-1}-D_{i}C_{i}^{-1}A_{i})(t-s)}\left(U'_{3}(s)-D_{i}C_{i}^{-1}U_{3}(s)\right)\mathrm{d}s,
\end{align}
for any $i=0,1,2,\ldots$ with
$M_{3}(0)=
\left[0, \Psi'_{3}(0), \Psi''_{3}(0)\right]^{\top}$ and
\begin{equation}
A_{i}:=
\left[
\begin{array}{ccc}
0 & 0 & 0
\\
\gamma(1-\alpha_{i}) & 0 & 0
\\
(\alpha_{i}-1)\gamma^{2}+2\gamma\mu_{i+1}\alpha_{i} & \gamma(1-\alpha_{i}) & 0
\end{array}
\right],
C_{i}:=
\left[
\begin{array}{ccc}
1 & 1 & -1
\\
-\gamma & -\mu_{i+1} & \gamma+\mu_{i+1}
\\
\gamma^{2} & \mu_{i+1}^{2} & -(\gamma+\mu_{i+1})^{2}
\end{array}
\right],
\end{equation}
with $\mu_{i+1}\neq \gamma$ and
\begin{align}
B_{i}:=
\left[
\begin{array}{ccc}
\gamma(1-\alpha_{i}) & 0 & 0
\\
(\alpha_{i}-1)\gamma^{2}+2\gamma\mu_{i+1}\alpha_{i} & \gamma(1-\alpha_{i}) & 0
\\
\gamma^{3}+\mu_{i+1}^{3}\alpha_{i}-(\gamma+\mu_{i+1})^{3}\alpha_{i} & (\alpha_{i}-1)\gamma^{2}+2\gamma\mu_{i+1}\alpha_{i} & \gamma(1-\alpha_{i})
\end{array}
\right],
\end{align}
and
\begin{equation}
D_{i}:=
\left[
\begin{array}{ccc}
-\gamma & -\mu_{i+1} & \gamma+\mu_{i+1}
\\
\gamma^{2} & \mu_{i+1}^{2} & -(\gamma+\mu_{i+1})^{2}
\\
-\gamma^{3} & -\mu_{i+1}^{2} & (\gamma+\mu_{i+1})^{3}
\end{array}
\right].
\end{equation}
\end{corollary}

\begin{remark}
Let us consider the special case for the offspring density function $h(t)=\beta e^{-\beta t}$ for some $\beta>0$.
In the classical Hawkes process, this corresponds to the case when the intensity process
is Markovian. It is easy to see that $k_{3,0}(t)=k_{1,0}(t)$ such that
\begin{equation}
k_{3,0}(t)=\frac{\eta(e^{\beta(\eta-1)t}-1)}{\eta-1},
\end{equation}
with the understanding that $k_{3,0}(t)=\beta t$ with $\eta=1$.
Moreover, we have
\begin{equation}
k_{T_1,\tau}(t) =\int_{0}^{t}e^{(\eta-1)\beta(t-s)}\left[\mu(s+\tau)+\eta\beta\int_0^s \mu(v+\tau) \, \mathrm{d}v\right]\mathrm{d}s.
\end{equation}
Furthermore, $k_{3,y,0,\tau}(t)$ satisfies the integral equation:
\begin{equation}
k_{3,y,0,\tau}(t) = \eta\left(1-e^{-\beta(t-y+\tau)}\right)+ \int_0^{y-\tau} \eta\beta^{-\beta(t-v)} k_{T_1,\tau}(v) \mathrm{d}v
+ \int_0^t \eta \beta e^{-\beta(t-v)} k_{3,y,0,\tau}(v) \mathrm{d}v.
\end{equation}
Similar as before, we can solve this integral equation to obtain:
\begin{align}
k_{3,y,0,\tau}(t) &= \int_{0}^{t}e^{(\eta-1)\beta(t-s)}\left[\eta\beta e^{-\beta(s-y+\tau)}-\beta\int_0^{y-\tau} \eta\beta^{-\beta(s-v)} k_{T_1,\tau}(v) \mathrm{d}v\right]\mathrm{d}s\nonumber
\\
&\qquad
+\int_{0}^{t}e^{(\eta-1)\beta(t-s)}\beta\left[ \eta\left(1-e^{-\beta(s-y+\tau)}\right)+ \int_0^{y-\tau} \eta\beta^{-\beta(s-v)} k_{T_1,\tau}(v) \mathrm{d}v \right]\mathrm{d}s
\nonumber
\\
&=\int_{0}^{t}e^{(\eta-1)\beta(t-s)}\eta\beta \mathrm{d}s
\nonumber
\\
&=\frac{\eta(e^{\beta(\eta-1)t}-1)}{\eta-1},
\end{align}
with the understanding that $k_{3,y,0,\tau}(t)=\beta t$ with $\eta=1$.
Moreover, we have
\begin{equation}
m(t)=\int_{0}^{t}e^{(\eta-1)\beta(t-s)}\left[\mu(s)+\beta\int_{0}^{s}\mu(v)\mathrm{d}v\right]\mathrm{d}s.
\end{equation}
\end{remark}

\begin{remark}
Indeed, we have
\begin{equation}
k_{3,0}(t) = k_{1,0}(t), \quad
k_{T_3,\tau}(t) = k_{T_1,\tau}(t),\quad
k_{3,y,0,\tau}(t) = k_{1,y,0,\tau}(t), \quad
\Psi_3(t) = \Psi_1(t).
\end{equation}
\end{remark}

\subsection{Discussions}

As a summary, the expectations of three classes of renewal Hawkes processes derived in Section~\ref{sec:R1:expectation}, Section~\ref{sec:R2:expectation} and Section~\ref{sec:R3:expectation}
satisfy the integral equations of the form:
\begin{equation}\label{eqn:43}
m_i(t) = \Psi_i(t) + \int_0^t m_i(t-y) n_i(y) \mathrm{d}y, \qquad i = 1, 2, 3,
\end{equation}
which implies that (i) the expectation of a renewal Hawkes process can be given by an integral equation in which a convolution is employed, (ii) the expectation of renewal Hawkes process can be regarded as a sum of expectation of renewal part and expectation of a process formed by immigrants and their offspring before the renewal time. This point may be useful in both theoretical and practical research.

We can rewrite the expectation of renewal Hawkes process proposed by \cite{Wheatley2016} as:
\begin{equation}\label{eqn:44}
m_{\mathrm{WFS}}(t) = \Psi_{\mathrm{WFS}}(t) + \int_0^t m_{\mathrm{WFS}}(t-y) n_{\mathrm{WFS}}(y) \mathrm{d}y,
\end{equation}
where $m_{\mathrm{WFS}}(t) = \mathbb{E}[N_{\mathrm{WFS}}(t)]$, $\Psi_{\mathrm{WFS}}(t) = \int_0^t [1 + k_{1,0}(t-y)] g(y) \mathrm{d}y$, and $n_{\mathrm{WFS}}(y) = g(y)$.

To better understand Equation~\eqref{eqn:43}, we provide the following explanations. The terms $\Psi_i(t)$, $i=1,2,3,$ are the sums of expectations of immigrants by the first renewal times and expectations of their offspring by time $t$ (before and after the first renewal times) for three renewal Hawkes processes, respectively. The terms $\int_0^t m_i(t-y) n_i(y) \mathrm{d}y$, $i=1,2,3,$ respectively are expectations of three renewal Hawkes processes after the first renewal times, and the terms $n_i(y)$, $i=1,2,3,$ are the probability density functions of the first renewal times for three renewal Hawkes processes, respectively. All explanations for Equation~\eqref{eqn:43} are applicable also to Equation~\eqref{eqn:44}.
Finally, the expectation of intensity function for each renewal Hawkes process can be given by the unified form: a derivative of the expectation of renewal Hawkes process.

\subsection{Special Case: Constant Exogenous Factor}\label{sec:special:case}

In this section, we study the special case when the exogenous factor, random variable $Y$, degenerates to a positive constant.
For this special case, the expressions for the corresponding expectations can be further simplified
for three novel renewal Hawkes processes that are derived in Section~\ref{sec:R1:expectation}, Section~\ref{sec:R2:expectation} and Section~\ref{sec:R3:expectation}.

For $\textup{R}_1$Hawkes process $N_1(t)$, when random variables $Y_1 = Y_2 = \cdots = c_1 < \infty$, i.e., the background intensity of the renewal Hawkes processes can be renewed at times $\{i c_1\}_{i=1,2,\ldots}$, and the probability density function and distribution function of $Y_i \ (i \geq 1)$ can be respectively expressed by using Dirac delta function as
\begin{equation*}
f(y) = \delta(y - c_1) \quad \text{and} \quad F(y) = 1 \ \text{for} \ y \geq c_1.
\end{equation*}

Based on Equations \eqref{m:1:eqn} and \eqref{Psi:1:eqn} and properties of Dirac delta function, we have
\begin{equation*}
m_1^{(c)}(t) = \Psi_1^{(c)}(t) + \int_0^t m_1^{(c)}(t-y) \delta(y - c_1) \mathrm{d}y = \Psi_1^{(c)}(t) + m_1^{(c)}(t - c_1) I_{\{t \geq c_1\}},
\end{equation*}
where $m_1^{(c)}(t) = \mathbb{E}[N_1(t) \mid Y = c_{1}]$,
$I_A$ is an indicator function of event $A$, i.e., $I_A = 1$ when event $A$ occurs, $I_A = 0$, otherwise, and
\begin{equation*}
\Psi_1^{(c)}(t) =
\begin{cases}
\mathbb{E}[N(t)] G(c_1), & \text{if } t < c_1, \\
\int_0^{c_1} [1 + k_{1,0}(t-\tau) + k_{T_1,\tau}(c_1-\tau) + k_{1,y,0,\tau}(t-c_1)] g(\tau) \mathrm{d}\tau, & \text{if } t \geq c_1.
\end{cases}
\end{equation*}

Similarly, for $\textup{R}_2$Hawkes process $N_2(t)$, when $Y_1 = Y_2 = \cdots = c_1 < \infty$, i.e., the background intensity of the renewal Hawkes processes can be renewed in which the interval renewal length is not greater than $c_1$. Based on Equation~\eqref{m:2:eqn}, we get
\begin{equation*}
m_2^{(c)}(t) =
\begin{cases}
\int_0^t [1 + k_{2,0}(t-\tau)] g(\tau) \mathrm{d}\tau + \int_0^t m_2^{(c)}(t-\tau) g(\tau) \mathrm{d}\tau, & \text{if } t < c_1, \\
 \int_0^{c_1} [1 + k_{2,0}(t-\tau) + m_2^{(c)}(t-\tau)] g(\tau) \mathrm{d}\tau + m_2^{(c)}(t-c_1) [1 - G(c_1)], & \text{if } t \geq c_1,
\end{cases}
\end{equation*}
where $m_2^{(c)}(t) = \mathbb{E}[N_2(t) \mid Y = c_{1}]$.

Similarly, for $\text{R}_3$Hawkes process $N_3(t)$, when $Y_1 = Y_2 = \cdots = c_1 < \infty$, i.e., the background intensity of the renewal Hawkes processes can be renewed in which the interval renewal length is not less than $c_1$. Based on Equations~\eqref{m:3:eqn} and \eqref{Psi:3:eqn}, we get
\begin{equation*}
\begin{aligned}
m_3^{(c)}(t) & = \Psi_3^{(c)}(t) + \int_0^t m_3^{(c)}(t-y) n_3^{(c)}(y) \mathrm{d}y \\
& = \Psi_3^{(c)}(t) + m_3^{(c)}(t-c_1)G(c_1) I_{\{t \geq c_1\}} + \int_{c_1}^t m_3^{(c)}(t-y) g(y) \mathrm{d}y,
\end{aligned}
\end{equation*}
where
\begin{equation*}
n_3^{(c)}(y) =
\begin{cases}
G(y) \delta(y - c_1), & \text{if } y < c_1, \\
G(y) \delta(y - c_1) + g(y), & \text{if } y \geq c_1,
\end{cases}
\end{equation*}
and
\begin{equation*}
\Psi_3^{(c)}(t) =
\begin{cases}
\mathbb{E}[N(t)] G(c_1), & \text{if } t < c_1, \\
\int_0^{c_1} [1+k_{3,0}(t-\tau)+ k_{T_1,\tau}(c_1-\tau)+k_{3,c_1,0,\tau}(t - c_1)] g(\tau) \mathrm{d}\tau, & \text{if } t \geq c_1.
\end{cases}
\end{equation*}


\section{Applications and Numerical Illustrations}\label{sec:applications}

The proposed three classes of renewal Hawkes processes are motivated by real applications as we discussed in Section~\ref{sec:intro}.
In the next section, we study a particular application in detail to periodic replacement policy with minimal repairs for systems with cascading failures.

\subsection{Periodic Replacement Policy for Systems with Cascading Failures}\label{sec:periodic:replacement}

The periodic replacement policy is one of the most fundamental maintenance policies in practice; see the classic fundamental monograph written
by reliability pioneers Barlow and Proschan in 1965 \cite{Barlow1965}. Since then, there has been an extensive literature on the periodic replacement; see, for example, \cite{Nakagawa1991,Zhao2022,Lima2025}. Poisson processes have been widely used in the periodic replacement policy.; see e.g. \cite{Dohi2001,Sheu2010,Mai2025}. However, such works failed
to model the cascading failures that can be captured by Hawkes processes; see e.g. \cite{Saichev2011}.
In this section, we study the applications of the proposed three classes of renewal Hawkes processes
to periodic replacement policy with partial minimal repairs for transmission device systems with cascading failures.
We make the following assumptions.

\begin{enumerate}
\item[(1)]
A system consists of two parts: parts A and B. Part A is a single-unit device, whereas part B contains multiple devices.
Part A undergoes a randomly periodical environment change, and each change duration follows a random variable
$Y\sim F(y)$, and part B suffers from cascading failures resulting from part A and itself during the operating process of the system.
\item[(2)]
All failures in the system are treated with minimal repairs. But the system is under a periodic replacement policy, i.e.,
the preventive replacements for the system at instants $iT$ ($i=1,2,\ldots$) are carried out, and each repair or replacement
is done within a small time, i.e., the repair time or replacement time is negligible.
\item[(3)]
Part A is naturally renewed at each environmental change cycle times $Y_{1},Y_{2},\ldots$ i.e.,
$Y_{1},Y_{2},\ldots$ are samples of $Y\sim F(y)$, and it is also renewed at preventive replacement instants $iT$ ($i=1,2,\ldots$).
Here, $Y_{i}$ are exogenous variables, which are the same $Y_{i}$'s as in the definitions of $\textup{R}_1$Hawkes, $\textup{R}_2$Hawkes, $\textup{R}_3$Hawkes processes in Section~\ref{sec:three:types}.
\item[(4)]
Each minimal repair needs cost $c_{f}$, and each replacement needs costs $c_{Ip}$ and $c_{p}$ for part A and the whole system respectively.
\item[(5)]
The failure rate function of part A is $\mu(t)$, and the whole cascading failures follow a linear Hawkes process $N(t)$ with
the intensity process $\lambda(t)=\mu(t)+\eta\int_{0}^{t-}h(t-s)dN(s)$, where $N(t)$ is the Hawkes process
that records the number of failures on the time interval $[0,t]$ for the system. Note that, under the different corresponding periodic replacement policies, $N(t)$
is denoted by $N_{i}(t), i=1,2,3,$ respectively.
\item[(6)]
Part B itself is highly reliable and only malfunctions under the influence of part A.
\end{enumerate}

In our model, minimal repairs and replacements are nested within the entire system, which will perform preventive replacements at periodic intervals. 
That is, part A, the main subsystem, has its own minimal repairs and regular renewals, whereas part B, the minor subsystem, has all minimal repairs. 
But the entire system will have periodic replacements. To the best of our knowledge, past research has never touched upon such issues.
The nesting of maintenance strategies in the system is more in line with the actual complex systems, especially for those with primary and auxiliary cascading failures. This nested maintenance strategy for the main subsystem is more targeted for failure prevention.

Our model setup is motivated from the real-world application.
For example, there is cascade fault behavior in the transmission device
system and energy supply system. The cascade failures can be modeled by one of our proposed renewal Hawkes processes, which depends on the considerations of exogenous and endogenous factors.

For the oil feeding system, due to the failure of the fuel supply system, it can cause cascade failures in other gearboxes, such as bearing transmission systems of wind turbines, high-speed railways, nuclear power plants and Marine power units. The number of failures in the oil feeding system can be modeled by our proposed renewal Hawkes processes.
See Figure~\ref{fig:oil} for an illustration of an oil feeding system.
The cascading failure system consists of the oil feeding system and gear-box systems, and the failures in oil feeding system with intensity function $\mu(t)$ can stimulate the occurrences of failures in gear-box systems.

The cascade failure behaviors in energy supply systems have been further considered in, for example, Liang et al. \cite{Liang2023}, who studied cascading failures for multiple energies (i.e., electricity, gas, cold, and heat) in a regional integrated energy system (RIES), in which failures in an energy subsystem can easily propagate to other energy subsystems via multi-energy coupling devices and result in cascading failures.
See Figure~\ref{fig:energy} for an illustration of an energy supply system.
The cascading failure system consists of the energy supply system (part A) and device systems $1$ to $k$ (part B), and the failures in part A with intensity function $\mu(t)$ can stimulate the occurrences of failures in device systems $1$ to $k$.
Because of the feature of cascading failures, the Hawkes processes are ideal models to characterize cascading failures, and in particular, the maintenance actions for systems with cascading failures can naturally be modeled by the renewal Hawkes processes.

\begin{figure}[htbp]
    \centering
    \includegraphics[width=0.6\textwidth]{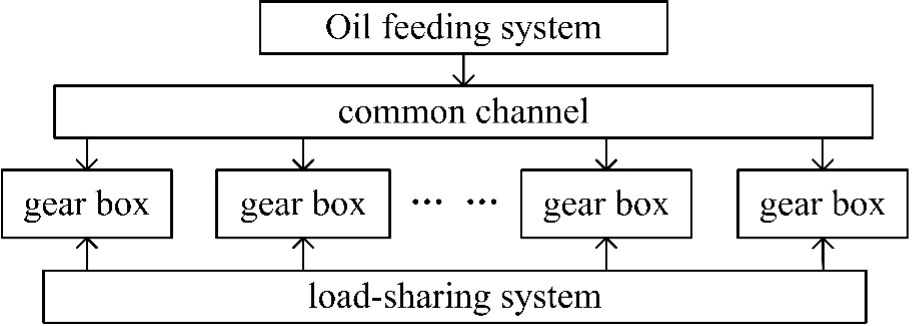}
   \caption{Illustration of an oil feeding system.}
      \label{fig:oil}
\end{figure}

\begin{figure}[htbp]
    \centering
    \includegraphics[width=0.6\textwidth]{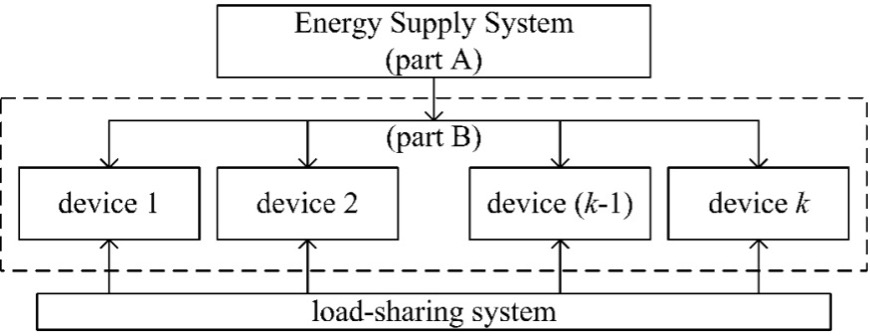}
   \caption{Illustration of an energy supply system.}
      \label{fig:energy}
\end{figure}

The customer arrivals can be modeled by a Hawkes process in which the customers can be divided into two classes, one class is due to some social influence or advertisement, the other class is due to a notice from a friend who has been a customer.
However, because of changes of social influence (e.g., advertisement) over time, the rule of customer arrives for first class must be renewed. The scenario can be modeled by the renewal Hawkes process with exogenous variable ($\textup{R}_1$Hawkes process).
Depending on whichever occurs first, i.e. social influence (e.g., advertisement), or the occurrence of customer arrival,
then the arrival of customers for first class can be modeled by a renewal process. The scenario can be modeled by the renewal Hawkes process with minimum of endogenous and exogenous variables ($\textup{R}_2$Hawkes process). Similarly,  depending on whichever occurs last, i.e. social influence (e.g., advertisement), or the occurrence of customer arrival, then the arrival of customers for first class can be modeled by a renewal process. The renewal Hawkes process with maximum of endogenous and exogenous variables ($\textup{R}_3$Hawkes process) can be used to model this scenario.
We illustrate the discussions on the customer arrivals above in Figure~\ref{fig:customer:arrivals}.

\begin{figure}[htbp]
    \centering
    \includegraphics[width=0.5\textwidth]{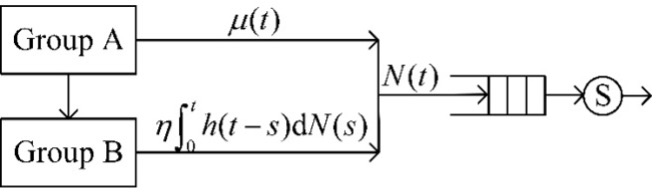}
   \caption{Illustration of two classes of customer arrivals.}
      \label{fig:customer:arrivals}
\end{figure}

In the following, we will propose and study two novel optimization problems.
Although in the studies of maintenance, there are many works on periodic replacements for minimizing the expected loss per unite time for the long-run average,
to the best of our knowledge,
no such optimization problems have been studied in the reliability literature.

Now, we introduce formally two optimization problems. Based on the result in renewal reward processes (see page 133 \cite{Ross1996}),
the first problem is to find an optimal $T_{1}^{\ast}$ that minimizes
the expected loss per unit time for the long-run average:
\begin{equation}\label{problem:1}
T_{1}^{\ast}:=\arg\min C_{1}(T),
\qquad
C_{1}(T):=\lim_{t\rightarrow\infty}\frac{\mathbb{E}[L_{1}(t)]}{t},
\end{equation}
where $L_{1}(t)$ stands for the loss during the time interval $[0,t]$, and $T$ is a renewal instant where we recall
that the preventive replacements for the system are conducted at instants $iT$ ($i=1,2,\ldots)$.
We have the following proposition.

\begin{proposition}\label{prop:C:1}
$T_{1}^{\ast}:=\arg\min C_{1}(T)$, where
\begin{equation}
C_{1}(T)=\frac{c_{f}\mathbb{E}[N(T)]+c_{Ip}\mathbb{E}[N_{R}(T)]+c_{p}}{T},
\end{equation}
where $N_{R}(t)$ is a renewal process for immigrants with interarrivals $Y_{1},Y_{2},\ldots$,
such that $\mathbb{E}[N_{R}(t)]=\sum_{n=1}^{\infty}F^{(n)}(t)$, where $F^{(n)}(t)$ is the $n$-fold convolution of $F$ with itself,
where $F$ is the distribution function of interarrivals $Y_{1},Y_{2},\ldots$,
and $N(t)$ is a renewal Hawkes process, and it is represented by $N_{i}(t)$ for $\textup{R}_i$Hawkes process, $i=1,2,3$,
that are introduced in Section~\ref{sec:three:types}.
\end{proposition}

In Proposition~\ref{prop:C:1}, $\mathbb{E}[N(T)]$
equal to $\mathbb{E}[N_{i}(T)]$ for $\textup{R}_i$Hawkes process, $i=1,2,3$,
and they are computed in Theorems~\ref{thm:1}, \ref{thm:2} and \ref{thm:3}.
Moreover, $\mathbb{E}[N_{R}(T)]$ is also explicitly computed in Proposition~\ref{prop:C:1}.
For special cases, it has an even more explicit form. For example, if $F(y)=1-e^{-\gamma y}$, then
$\mathbb{E}[N_{R}(t)]=\gamma t$.
We will numerically compute $\mathbb{E}[N(T)]$ in Section~\ref{sec:example:1} based on the numerical method developed in Section~\ref{sec:num:procedure}.
We will further numerically compute $C_{1}(T)$ and the optimal placement time $T_{1}^{\ast}$ in Section~\ref{sec:example:2}.

The next optimization problem has the same setup as in the first problem
except that the cost $c_{f}$ is divided into two issues:
Each minimal repair need costs $c_{I}$ and $c_{O}$ for parts A and B, respectively,
where it is assumed that $c_{I}\geq c_{O}$. Here, the subscripts $I$ and $O$ stand for ``important'' and ``ordinary'' respectively.
Then, based on the result in renewal reward processes (see page 133 \cite{Ross1996}), we have
\begin{equation}\label{problem:2}
T_{2}^{\ast}:=\arg\min C_{2}(T),
\qquad
C_{2}(T):=\lim_{t\rightarrow\infty}\frac{\mathbb{E}[L_{2}(t)]}{t},
\end{equation}
where $L_{2}(t)$ is the loss during the time interval $[0,t]$, $T$ is a renewal constant, where we recall
that the preventive replacements for the system are conducted at instants $iT$ ($i=1,2,\ldots)$ and
\begin{equation}
C_{2}(T)=\frac{c_{I}\mathbb{E}[N_{\mathrm{NHP}}(Y_{1})]\mathbb{E}[N_{R}(T)]+c_{I}O_{I}(T)+c_{O}\mathbb{E}[N_{O}(T)]+c_{Ip}\mathbb{E}[N_{R}(T)]+c_{p}}{T},
\end{equation}
where $O_{I}(t):=\int_{0}^{t}\mathbb{E}[N_{\mathrm{NHP}}(t-y)]\mathbb{E}\left[f^{(N_{R}(t))}(y)\right]\mathrm{d}y$,
$N_{R}(t)$ is a renewal process for immigrants with interarrivals $Y_{1},Y_{2},\ldots$,
and we have the decomposition:
\begin{equation}\label{eqn:decomposition}
N(t)=N_{I}(t)+N_{O}(t),
\end{equation}
where $N_{I}(t)$ is the number of failures in part A by time $t$, which is modeled by a point process for immigrants with intensity function
$\mu(t)$ renewed at times $Y_{1}$, $\sum_{i=1}^{2}Y_{i}$, $\sum_{i=1}^{3}Y_{i}$, etc.
and $N_{O}(t)$ is the number
of failures in part B by time $t$, which is modeled by a self-exciting process for offspring of $N(t)$ with intensity function
$\eta\int_{0}^{t-}h(t-s)dN(s)$, and $N_{\mathrm{NHP}}(t)$ is a non-homogeneous Poisson process
with intensity function $\mu(t)$.
Intuitively, the decomposition \eqref{eqn:decomposition} is due to the fact that
the total number of failures in the system is the sum of failure numbers in parts A and B.
Indeed, we have the following proposition.

\begin{proposition}\label{prop:C:2}
$T_{2}^{\ast}:=\arg\min C_{2}(T)$, where
\begin{equation}
C_{2}(T)=\frac{(c_{I}-c_{O})\mathbb{E}[N_{I}(T)]+c_{O}\mathbb{E}[N(T)]+c_{Ip}\mathbb{E}[N_{R}(T)]+c_{p}}{T},
\end{equation}
where $N_{R}(t)$ is a renewal process for immigrants with interarrivals $Y_{1},Y_{2},\ldots$ such that
$\mathbb{E}[N_{R}(t)]=\sum_{n=1}^{\infty}F^{(n)}(t)$, where $F^{(n)}(t)$ is the $n$-fold convolution of $F$ with itself,
where $F$ is the distribution function of interarrivals $Y_{1},Y_{2},\ldots$, and
$N_{I}(t)$ is a point process for immigrants with intensity function
$\mu(t)$ renewed at times $Y_{1}$, $\sum_{i=1}^{2}Y_{i}$, $\sum_{i=1}^{3}Y_{i}$, etc. and we have
\begin{equation}\label{N:I:T:original}
\mathbb{E}[N_{I}(T)]
=\mathbb{E}[N_{\mathrm{NHP}}(Y_{1})]\mathbb{E}[N_{R}(T)]+\int_{0}^{T}\mathbb{E}\left[N_{\mathrm{NHP}}(T-y)\right]\mathbb{E}\left[f^{(N_{R}(T))}(y)\right]\mathrm{d}y,
\end{equation}
where $N_{\mathrm{NHP}}(t)$ is a non-homogeneous Poisson process with intensity function $\mu(t)$ with
\begin{equation}\label{NHP:Y:1:formula}
\mathbb{E}[N_{\mathrm{NHP}}(Y_{1})]=\int_{0}^{\infty}\int_{0}^{y}\mu(v)\mathrm{d}v\mathrm{d}F(y),
\end{equation}
and $N(t)$ is a renewal Hawkes process, and it is represented by $N_{i}(t)$ for $\textup{R}_i$Hawkes process, $i=1,2,3$,
that are introduced in Section~\ref{sec:three:types}.
\end{proposition}

In Proposition~\ref{prop:C:2}, $\mathbb{E}[N(T)]$
equal to $\mathbb{E}[N_{i}(T)]$ for $\textup{R}_i$Hawkes process, $i=1,2,3$,
and they are computed in Theorems~\ref{thm:1}, \ref{thm:2} and \ref{thm:3}.
We will numerically compute $\mathbb{E}[N(T)]$ in Section~\ref{sec:example:1} based on the numerical method developed in Section~\ref{sec:num:procedure}.
Moreover, the terms $\mathbb{E}[N_{\mathrm{NHP}}(Y_{1})]$, $\mathbb{E}[N_{R}(T)]$
are computed out explicitly in Proposition~\ref{prop:C:2}.
The term $\mathbb{E}[N_{I}(T)]$ in Proposition~\ref{prop:C:2}
is more involved to compute. We provide
the following computation.

\begin{corollary}\label{N:I:T:complicated}
We have
\begin{align}
\mathbb{E}[N_{I}(T)]
&=\sum_{n=1}^{\infty}F^{(n)}(t)\int_{0}^{\infty}\int_{0}^{y}\mu(v)\mathrm{d}v\mathrm{d}F(y)
\nonumber
\\
&\qquad
+\sum_{n=0}^{\infty}\int_{0}^{T}\int_{0}^{T-y}\int_{0}^{T}\mu(s)
f^{(n)}(y)\left(f^{(n)}(v)-f^{(n+1)}(v)\right)\mathrm{d}v\mathrm{d}s\mathrm{d}y.
\end{align}
\end{corollary}

As we can see from Corollary~\ref{N:I:T:complicated},
the formula for $\mathbb{E}[N_{I}(T)]$ is complicated
and hard to compute because $f^{(n)}(t)$ is essentially
an $(n-1)$-tuple-integral, which is costly to compute
when $n$ becomes large.
For some special cases though, one
can get a much easier to compute expression for
$\mathbb{E}[N_{I}(T)]$.
For example, when $\mu(t)=2t$ and $F(t)=1-e^{-\gamma t}$ with $\gamma>0$,
we can compute out $\mathbb{E}[N_{I}(T)$] in closed-form
in the following corollary

\begin{corollary}\label{N:I:T}
Assume $\mu(t)=2t$ and $F(t)=1-e^{-\gamma t}$ with $\gamma>0$.
Then, we have
\begin{align}
\mathbb{E}[N_{I}(T)]&=\frac{4T}{\gamma}-T^{2}\sum_{k=0}^{\infty}\sum_{n=0}^{k-1}\frac{e^{-2\gamma T}(\gamma T)^{k+n}}{k!n!}
\nonumber
\\
&\qquad
+\frac{2T}{\gamma}\sum_{k=0}^{\infty}\sum_{n=0}^{k}\frac{e^{-2\gamma T}(\gamma T)^{k+n}k}{k!n!}
-\frac{1}{\gamma^{2}}\sum_{k=0}^{\infty}\sum_{n=0}^{k+1}\frac{e^{-2\gamma T}(\gamma T)^{k+n}k(k+1)}{k!n!}.
\end{align}
\end{corollary}

However, in general, as we have seen in Corollary~\ref{N:I:T:complicated},
$\mathbb{E}[N_{I}(T)]$ in Proposition~\ref{prop:C:2}
seems to be hard to evaluate analytically.
Nevertheless, in the following corollary,
we derive an upper bound and a lower bound for
$\mathbb{E}[N_{I}(T)]$, and as a result,
an upper bound and a lower bound for $C_{2}(T)$
that will be shown to be quite tight in practice.

\begin{proposition}\label{cor:C:2}
We have the following upper and lower bounds:
\begin{equation}
\mathbb{E}\left[N_{\mathrm{NHP}}(Y_{1})\right]\mathbb{E}[N_{R}(t)]
\leq\mathbb{E}[N_{I}(t)]
\leq\mathbb{E}\left[N_{\mathrm{NHP}}(Y_{1})\right]\left(\mathbb{E}[N_{R}(t)]+1\right),
\end{equation}
and as a result, we have
\begin{equation}
\underline{C}_{2}(T)\leq C_{2}(T)\leq\overline{C}_{2}(T),
\end{equation}
where
\begin{align}
&\underline{C}_{2}(T):=\frac{(c_{I}-c_{O})\mathbb{E}[N_{\mathrm{NHP}}(Y_{1})]\mathbb{E}[N_{R}(T)]+c_{O}\mathbb{E}[N(T)]+c_{Ip}\mathbb{E}[N_{R}(T)]+c_{p}}{T},\label{underline:C:2}
\\
&\overline{C}_{2}(T):=\frac{(c_{I}-c_{O})\mathbb{E}[N_{\mathrm{NHP}}(Y_{1})]\left(\mathbb{E}[N_{R}(T)]+1\right)+c_{O}\mathbb{E}[N(T)]+c_{Ip}\mathbb{E}[N_{R}(T)]+c_{p}}{T}.\label{overline:C:2}
\end{align}
\end{proposition}

We will numerically compute $C_{2}(T)$ and the optimal placement time $T_{2}^{\ast}$ in Section~\ref{sec:example:2}.
We will see in Section~\ref{sec:example:2} that the lower and upper bounds for $C_{2}(T)$ provided in \eqref{underline:C:2}-\eqref{overline:C:2}
in Proposition~\ref{cor:C:2} are very tight in practice.
Indeed, we can quantify the gap between the upper bound in \eqref{overline:C:2}
and the lower bound in \eqref{underline:C:2}; it follows from \eqref{underline:C:2}-\eqref{overline:C:2}
and the formula in \eqref{NHP:Y:1:formula} that
\begin{align}\label{eqn:gap}
\overline{C}_{2}(T)-\underline{C}_{2}(T)
=\frac{(c_{I}-c_{O})\mathbb{E}[N_{\mathrm{NHP}}(Y_{1})]}{T}=\frac{c_{I}-c_{O}}{T}\int_{0}^{\infty}\int_{0}^{y}\mu(v)\mathrm{d}v\mathrm{d}F(y),
\end{align}
which is small if $c_{I}-c_{O}$, $\mathbb{E}[N_{\mathrm{NHP}}(Y_{1})]$ are small,
or $T$ is large. Note that $\mathbb{E}[N_{\mathrm{NHP}}(Y_{1})]$ can often
be computed more explicitly in special cases. For example,
when $\mu(t)=2t$ and $F(y)=1-e^{-\gamma y}$ with $\gamma>0$, one can compute from \eqref{NHP:Y:1:formula} that
$\mathbb{E}[N_{\mathrm{NHP}}(Y_{1})]
=\int_{0}^{\infty}y^{2}\gamma e^{-\gamma y}\mathrm{d}y=\frac{2}{\gamma^{2}}$,
which is small when $\gamma$ is large.

We also note from Proposition~\ref{prop:C:1} and Proposition~\ref{prop:C:2}
that $C_{2}(T)=C_{1}(T)$ when $c_{I}=c_{O}=c_{f}$, i.e. both optimization problems
are equivalent when minimal repair costs are the same each time
for parts A and B, which is consistent with the intuition.

\subsection{Numerical Experiments}\label{sec:numerical}

In this section, we conduct numerical experiments to illustrate the application to periodic replacement policy for systems with cascading failures
we introduced in Section~\ref{sec:periodic:replacement}, which relies on computing $\mathbb{E}[N_{i}(T)]$, $i=1,2,3$, i.e. the expectations of three proposed renewal Hawkes processes
derived in Section~\ref{sec:expectations}. In Section~\ref{sec:num:procedure}, we will develop efficient numerical methods for solving the expectations
of these renewal Hawkes processes and they are applicable beyond the particular application in Section~\ref{sec:periodic:replacement}.

\subsubsection{Numerical Procedure}\label{sec:num:procedure}

In this section, we will provide numerical solutions of expectations of renewal Hawkes processes.
As the results showed in Equation~\eqref{eqn:43}, the expectations of all three renewal Hawkes processes satisfy integral equations, which in general are difficult to give analytic solutions, even for the case of $\textup{R}_2$Hawkes process expressed by the form of Laplace transform (see Equation~\eqref{eqn:34}). But numerical solutions to these integral equations can be presented by using the direct Riemann-Stieltjes integration method (see \cite{Xie1989} and \cite{Xie2003}) with a given error. On the other hand, for the case of the $\text{R}_2$Hawkes process, we can use the numerical Laplace transform inversion formula for presenting the numerical solution. But in the present paper, we only show the results using the direct Riemann integration method as follows.

Using the direct Riemann integration method, we can get the following procedure for solving these integral equations.
Let the interval $[0, t]$ be divided into $n$ equal parts, denoted by $\Delta t = t/n$ their length (note here $\Delta t$ is enough small or $n$ is enough large). Based on Equation~\eqref{eqn:43}, we have
\begin{equation}
\begin{cases}
\text{for } i = 1,2,3, \\
m_i(\Delta) = \Psi_i(\Delta), \\
m_i(2\Delta) = \Psi_i(2\Delta) + \Delta \sum_{k=1}^2 m_i((2-k)\Delta) n_i(k\Delta), \\
m_i(j\Delta) = \Psi_i(j\Delta) + \Delta \sum_{k=1}^j m_i((j-k)\Delta) n_i(k\Delta), \ j = 2,3,\ldots, n.
\end{cases}
\end{equation}

For computation of $\Psi_i(j\Delta) \ (i = 1,2,3, \ j = 1,2,\ldots, n)$, some set of equations are also recursively established, which are given as follows.
\begin{equation}\label{eqn:46}
\begin{cases}
k_{i,0}(j\Delta) = \eta \int_0^{j\Delta} h(v) \mathrm{d}v + \eta \Delta \sum_{l=1}^j h(l\Delta) k_{i,0}((j-l)\Delta), \ i = 1,2,3, \ \text{with} \ k_{i,0}(0) = 0, \\
k_{T_1,l\Delta}(j\Delta) = \int_0^{j\Delta} \mu(v+l\Delta) \mathrm{d}v + \eta \Delta \sum_{k=1}^j h(k\Delta) k_{T_1,l\Delta}((j-k)\Delta), \ l \leq j, \text{with} \,\, k_{T_1,l\Delta}(0)=0, \\
k_{i,r\Delta,0,l\Delta}(j\Delta) = \eta H(j\Delta)k_{T_1,l\Delta}((r-l)\Delta)
\\
\quad\qquad\qquad\qquad\qquad- \eta \Delta \sum_{k=1}^{r-l}\left[h((r-l-k)\Delta)- h((j+r-l-k)\Delta)\right] k_{T_1,l\Delta}(k\Delta) \\
\qquad\quad\quad\quad\quad\quad\quad +\eta \Delta \sum_{k=1}^j h(k\Delta) k_{i,r\Delta,0,l\Delta}((j-k)\Delta), \ i = 1,3, \ \text{with} \ k_{i,r\Delta,0,l\Delta}(0) = 0, \\
\mathbb{E}[N(j\Delta)] = \int_0^{j\Delta} \mu(v) \mathrm{d}v + \eta \Delta \sum_{k=1}^j h(k\Delta) \mathbb{E}[N((j-k)\Delta)], \ \text{with} \,\,  \mathbb{E}[N(0)] = 0.
\end{cases}
\end{equation}

Based on Equation~\eqref{eqn:46} and formulas of $\Psi_i(t) \ (i = 1,2,3)$ presented in Section~\ref{sec:expectations}, the similar recursive equations can be given for $\Psi_i(j\Delta) \ (i = 1,2,3, \ j = 1,2,\ldots, n)$ in which the integrals reduce into summations; more precisely, they are
\begin{equation}
\begin{aligned}
\Psi_i(j\Delta) &= \int_0^{j\Delta} G(y) f(y) \mathrm{d}y + \Delta \sum_{k=1}^j k_{i,0}((j-k)\Delta) F_g(k\Delta) - [1 - F(j\Delta)] \Delta \sum_{s=1}^j k_{i,0}((j-s)\Delta) g(s\Delta) \\
&\quad + \Delta^2 \sum_{k=1}^j \sum_{l=1}^k k_{T_1,l\Delta}((k-l)\Delta) g(l\Delta) f(k\Delta) + \Delta^2 \sum_{k=1}^j \sum_{r=1}^k k_{i,r\Delta,0,l\Delta}((j-k)\Delta) g(l\Delta) f(k\Delta) \\
&\quad\quad + \mathbb{E}[N(j\Delta)] \int_{j\Delta}^\infty G(y) f(y) \mathrm{d}y, \quad i = 1,3,
\end{aligned}
\end{equation}
and
\begin{equation}
\Psi_2(j\Delta) = \int_0^{j\Delta} F_g(y) \mathrm{d}y + \Delta \sum_{i=1}^j k_{2,0}((j-i)\Delta) F_g(i\Delta).
\end{equation}

The expectations of intensity processes for three renewal Hawkes processes can be recursively computed at points $0, \Delta, 2\Delta, \ldots, j\Delta, \ldots, n\Delta = t$ by
\begin{equation}
\begin{cases}
\mathbb{E}[\lambda_i(0)] = \mu(0), \\
\mathbb{E}[\lambda_i(\Delta)] = m_i(\Delta) / \Delta, \\
\mathbb{E}[\lambda_i(j\Delta)] = \frac{1}{\Delta} [m_i(j\Delta) - m_i((j-1)\Delta)], \quad j = 2,3,\ldots
\end{cases}
\end{equation}

\subsubsection{Numerical Illustration: Expectations of Renewal Hawkes Processes}\label{sec:example:1}

In this section, we provide explicit numerical examples to illustrate the computations
of expectations of three proposed renewal Hawkes processes via the numerical procedure
introduced in Section~\ref{sec:num:procedure}.

For an $\textup{R}_1$Hawkes process $N_1(t)$ with intensity function
\begin{equation*}
\lambda_1(t) = \mu(t - U_{I(t)}) + \int_0^{t-} \eta h(t-v) \mathrm{d}N_1(v),
\end{equation*}
where $\mu(t) = 2t, \ h(t) = 2e^{-2t}, \eta = 2$, the random variables $Y_1, Y_2, \ldots$ have a distribution function and probability density function, respectively,
\begin{equation*}
F(y) = 1 - e^{-y} \quad \text{and} \quad f(y) = e^{-y}, \quad y \geq 0.
\end{equation*}
For the other two renewal Hawkes processes, the same $\mu(t), h(t)$, $\eta$ and $F(y)$ are used.
In the meantime, we have the distribution function and probability density function of the virtual time $T_1$, respectively,
\begin{equation*}
G(t) = 1 - e^{-t^2} \quad \text{and} \quad g(t) = 2te^{-t^2}, \quad t \geq 0.
\end{equation*}

By using Maple software, we obtain Figures~\ref{fig:3} and \ref{fig:4} for expectations $m_i(t)=\mathbb{E}[N_{i}(t)]$, to illustrate the expectations computed
for renewal Hawkes processes $\textup{R}_i$Hawkes process in Theorem~\ref{thm:1}, Theorem~\ref{thm:2}, and Theorem~\ref{thm:3}, for $i=1,2,3$, respectively, under specified computational parameters $\Delta = \frac{1}{100}$, $\frac{1.5}{100}$,
and $\Delta=\frac{2}{200}$, $\frac{2.5}{200}$, where the vertical axis stands for expectations $m_i(t)=\mathbb{E}[N_{i}(t)]$ and the horizontal axis stands for $t$ in each figure.

\begin{figure}[htbp]
    \begin{minipage}[t]{0.5\linewidth}
        \centering
        \includegraphics[scale=0.35]{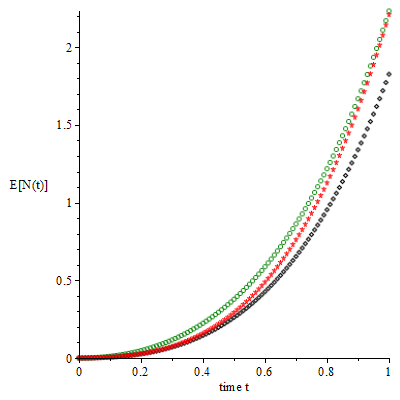}
        \centerline{(a) $\Delta = 1/100$}
    \end{minipage}%
    \begin{minipage}[t]{0.5\linewidth}
        \centering
        \includegraphics[scale=0.35]{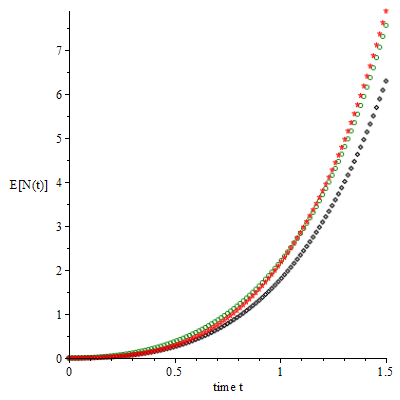}
        \centerline{(b) $\Delta = 1.5/100$}
    \end{minipage}
    \caption{The expectations of renewal Hawkes processes $N_i(t) \ (i=1,2,3)$ with small $\Delta$. Red for $m_1(t)$, black for $m_2(t)$, green for $m_3(t)$.}\label{fig:3}
\end{figure}

\begin{figure}[htbp]
    \begin{minipage}[t]{0.5\linewidth}
        \centering
        \includegraphics[scale=0.3]{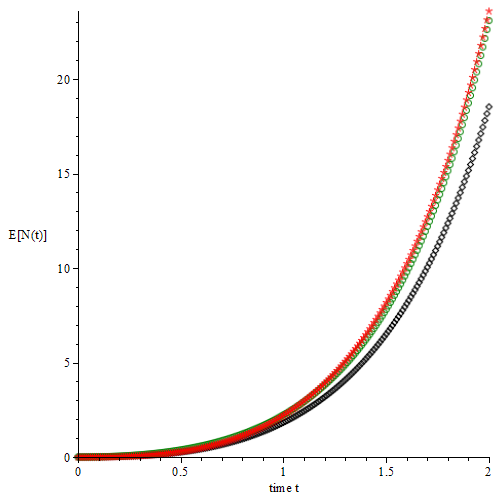}
        \centerline{(a) $\Delta = 2/200$}
    \end{minipage}%
    \begin{minipage}[t]{0.5\linewidth}
        \centering
        \includegraphics[scale=0.3]{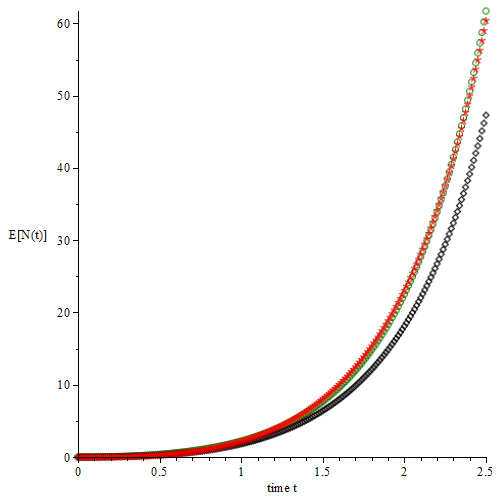}
        \centerline{(b) $\Delta = 2.5/200$}
    \end{minipage}
    \caption{The expectations of renewal Hawkes processes $N_i(t) \ (i=1,2,3)$ with large $\Delta$. Red for $m_1(t)$, black for $m_2(t)$, green for $m_3(t)$.}\label{fig:4}
\end{figure}

Similarly, we get Figures~\ref{fig:5} and \ref{fig:6} for $\mathbb{E}[\lambda_i(t)]$, the expectations of the intensity of renewal Hawkes processes $\textup{R}_i$Hawkes process
for $i=1,2,3$,
that are computed in Remark~\ref{remark:1}, Remark~\ref{remark:2} and Remark~\ref{remark:3}, respectively, under specified computational parameters $\Delta = \frac{1}{100}$, $\frac{1.5}{100}$,
and $\Delta=\frac{2}{200}$, $\frac{2.5}{200}$, where the vertical axis stands for expectations of the intensity $\mathbb{E}[\lambda_i(t)]$ and the horizontal axis stands for $t$ in each figure.

\begin{figure}[htbp]
    \begin{minipage}[t]{0.5\linewidth}
        \centering
        \includegraphics[scale=0.35]{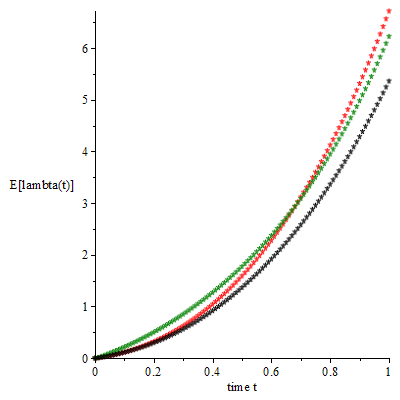}
        \centerline{(a) $\Delta = 1/100$}
    \end{minipage}%
    \begin{minipage}[t]{0.5\linewidth}
        \centering
        \includegraphics[scale=0.35]{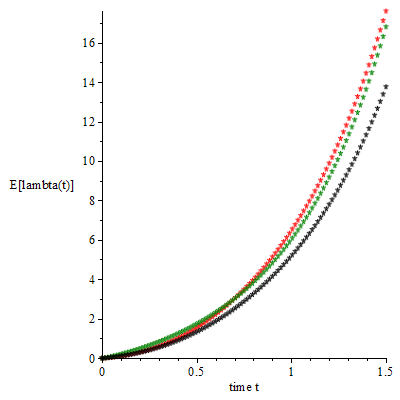}
        \centerline{(b) $\Delta = 1.5/100$}
    \end{minipage}
    \caption{The expectations of intensity processes $\lambda_i(t) \ (i=1,2,3)$ with small $\Delta$. Red for $\mathbb{E}[\lambda_1(t)]$, black for $\mathbb{E}[\lambda_2(t)]$, green for $\mathbb{E}[\lambda_3(t)]$.
    }\label{fig:5}
\end{figure}

\begin{figure}[htbp]
    \begin{minipage}[t]{0.5\linewidth}
        \centering
         \includegraphics[scale=0.3]{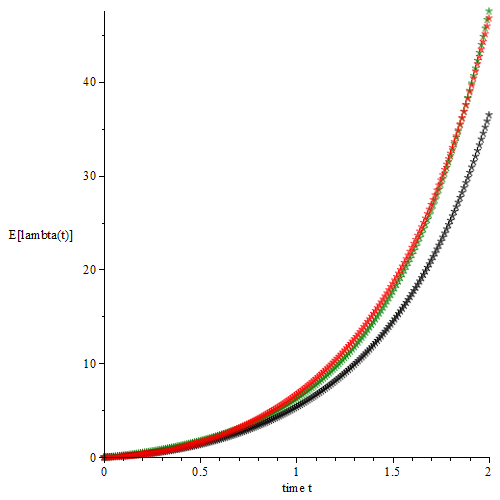}
        \centerline{(a) $\Delta = 2/200$}
    \end{minipage}%
    \begin{minipage}[t]{0.5\linewidth}
        \centering
        \includegraphics[scale=0.3]{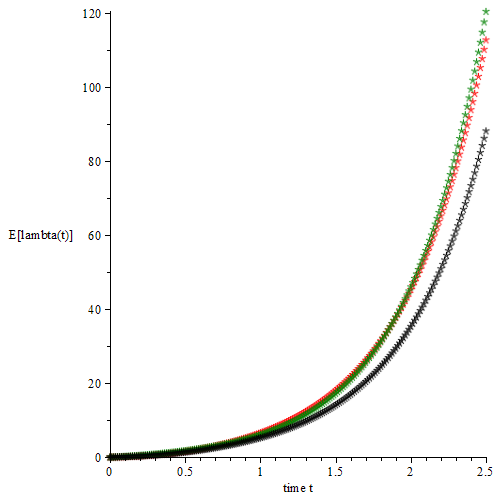}
        \centerline{(b) $\Delta = 2.5/200$}
    \end{minipage}
     \caption{The expectations of intensity processes $\lambda_i(t) \ (i=1,2,3)$ with small $\Delta$. Red for $\mathbb{E}[\lambda_1(t)]$, black for $\mathbb{E}[\lambda_2(t)]$, green for $\mathbb{E}[\lambda_3(t)]$.
     }\label{fig:6}
\end{figure}

We can see from these figures that all expectations increase as time increases, which is consistent with our intuition since
$m_{i}(t)$ is the expectation of the number of arrivals for renewal Hawkes processes on $[0,t]$ which is always increasing in $t$
and $\mathbb{E}[\lambda_{i}(t)]$ is also increasing in $t$ due to the choice of $\mu(t)=2t$ that is increasing in time $t$.
Moreover, since $\mathbb{E}[\lambda_{i}(t)]$ is increasing in $t$, we expect that $m_{i}(t)$ has a super-linear growth in time $t$,
which is also confirmed by Figure~\ref{fig:3} and Figure~\ref{fig:4}.
Meanwhile, the accuracy of computations depends on the parameter $\Delta$, i.e., the smaller $\Delta$ results in more accuracy; but it increases the load of computations.

\subsubsection{Numerical Illustration: Optimal Replacement Times}\label{sec:example:2}

In this section, we provide numerical illustration to compute the optimal replacement times
that are studied in Section~\ref{sec:periodic:replacement} for the application of periodic replacement policy with minimal repairs for systems with cascading failures.
We numerical solve two optimization problems introduced in Section~\ref{sec:periodic:replacement}
by relying on numerically computing the expectations of three proposed renewal Hawkes processes
which is illustrated in Section~\ref{sec:example:1} based on the numerical method developed in Section~\ref{sec:num:procedure}.

We assume that $F(y)=1-e^{-y}$, $h(t)=2e^{-2t}$, $\eta=2$, $\mu(t)=2t$,
The related costs are assumed to be $c_{f}=1.5$, $c_{Ip}=5$, $c_{p}=10$, $c_{I}=2$, $c_{O}=1$.
For two optimization problems \eqref{problem:1} and \eqref{problem:2} introduced in Section~\ref{sec:periodic:replacement},
we can get the corresponding optimal replacement times for the whole system under three renewal Hawkes processes; see Proposition~\ref{prop:C:1} and Proposition~\ref{prop:C:2}.
Next, we provide detailed descriptions and graphic illustrations for failure processes of the system under different renewal Hawkes processes.

\paragraph{(i) Renewal Hawkes Process with Exogenous Variable}

\begin{figure}[htbp]
    \centering
    \includegraphics[width=0.7\textwidth]{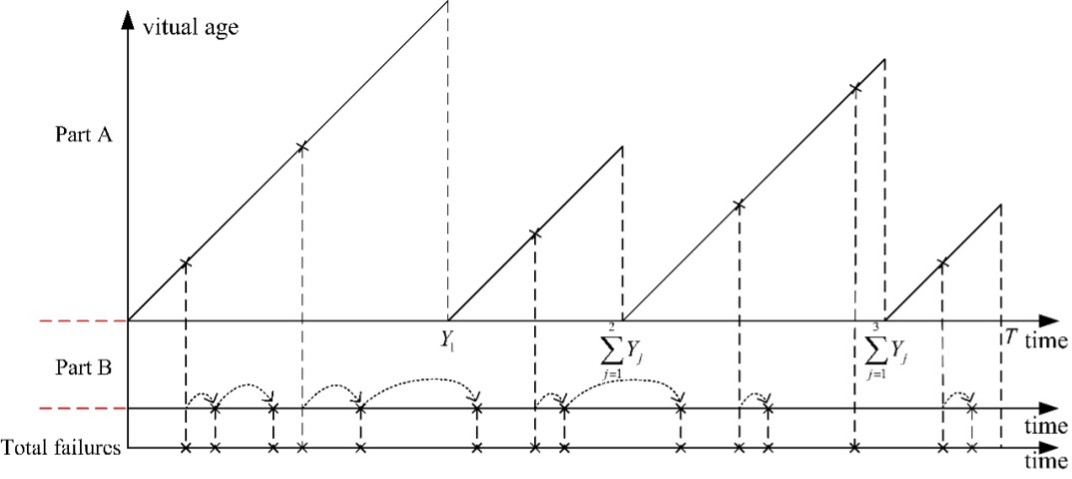}
    \caption{A realization of the failure process of the system in a replacement cycle under renewal Hawkes process with exogenous variable. The dotted arc with an arrow denotes the relationship between parent and its offspring.}
  \label{fig:realization:failure}
\end{figure}

As shown in Figure~\ref{fig:realization:failure}, for a renewed cycle, failures in part A of the system result from its age, and in the intervals
$[0,Y_{1})$, $[Y_{1},\sum_{j=1}^{2}Y_{j})$ and $[\sum_{j=1}^{2}Y_{j},\sum_{j=1}^{3}Y_{j})$, there are two, one and two
failures respectively and they are treated with minimal repairs immediately with negligible time after each failure occurs, and part A is renewed at times $Y_{1}$, $\sum_{j=1}^{2}Y_{j}$,$\sum_{j=1}^{3}Y_{j}$, respectively. Failures in part B of the system result
from failures occurred in part A and itself with cascading effects, and in the intervals
$[0,Y_{1})$, $[Y_{1},\sum_{j=1}^{2}Y_{j})$, $[\sum_{j=1}^{2}Y_{j},\sum_{j=1}^{3}Y_{j})$ and $[\sum_{j=1}^{3}Y_{j}, T)$, there are three (including one second generation offspring), two (including one second generation offspring), two (including one second generation offspring) and one
failures and they are treated with minimal repairs immediately with negligible time after each failure occurs, and the whole system is restored preventively at time $T$, i.e. the whole system is renewed at time $T$.

The curves of $C_{1}(T)$ and $C_{2}(T)$ are given in Figure~\ref{fig:T11:T21}, where the vertical axis standards for $C_{i}(T)$, $i=1,2$,
and the horizontal axis stands for $T$ and the corresponding
optimal replacement times are $T_{11}^{\ast}=1.12$ and $T_{21}^{\ast}=1.26$, where
$T_{ij}^{\ast}$ denotes the optimal replacement time for optimization problem $i$ under $\textup{R}_j$Hawkes process.
On the right hand side in Figure~\ref{fig:T11:T21}, the blue curve denotes the upper bound and red curve denotes the lower bound for $C_{2}(T)$
obtained from Proposition~\ref{cor:C:2}.
We can see that the blue and red curves on the right hand side in Figure~\ref{fig:T11:T21}, and hence the corresponding optimal replacement times
are very close to each other. This illustrates that the lower and upper bounds in Proposition~\ref{cor:C:2} are very tight.

\begin{figure}[htbp]
    \centering
    \includegraphics[width=0.35\textwidth]{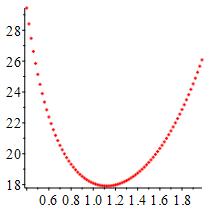}
     \includegraphics[width=0.35\textwidth]{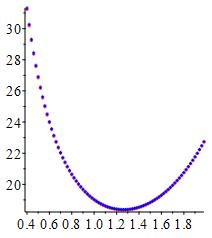}
     \caption{Plots of $C_{1}(T)$ (left) and $C_{2}(T)$ (right) with minimizers being the optimal replacement times $T_{11}^{\ast}$ (left) and $T_{21}^{\ast}$ (right),
     where the blue curve denotes the upper bound $\overline{C}_{2}(T)$ and red curve denotes the lower bound $\underline{C}_{2}(T)$ for $C_{2}(T)$ (right).}
  \label{fig:T11:T21}
\end{figure}

\paragraph{(ii) Renewal Hawkes Process with Minimum of Endogenous and Exogenous Variables}

\begin{figure}[htbp]
    \centering
    \includegraphics[width=0.7\textwidth]{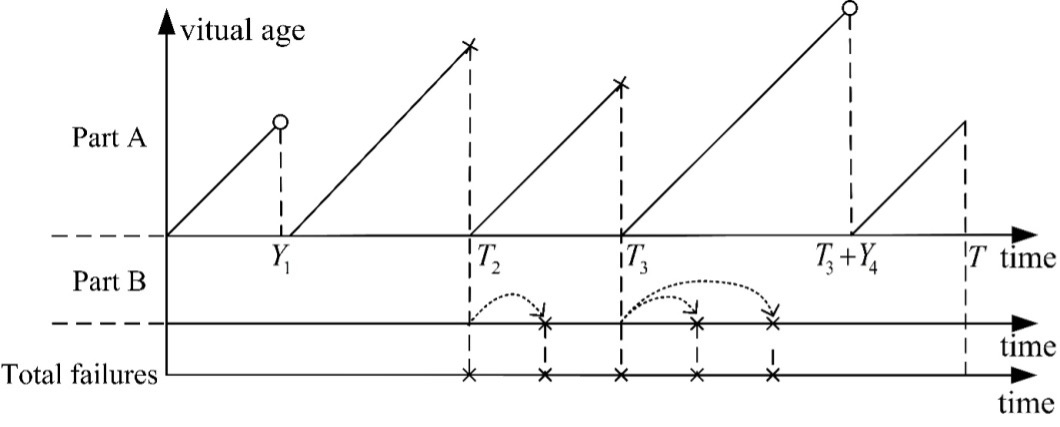}
    \caption{A realization of the failure process of the system in a replacement cycle under renewal Hawkes process with minimum of endogenous and exogenous variables. The dotted arc with an arrow denotes the relationship between parent and its offspring.}
  \label{fig:realization:failure:2}
\end{figure}

As shown in Figure~\ref{fig:realization:failure:2}, for a renewed cycle, failures in part A of the system also result from its age, and in the interval $[0,T_{3}+Y_{4}]$, there are two failures in total at $t=T_{2}$ and $t=T_{3}$, respectively, and they are treated with replacements with negligible times as part A is renewed. 
Failures in part B of the system result from failures occurred in part A and itself with cascading effects, and in the intervals
$[T_{2},T_{3})$ and $[T_{3},T_{3}+Y_{4})$; there are one and two failures respectively, and they are treated with minimal repairs immediately with negligible time after each failure occurs. The whole system is renewed at time $T$; but part A are renewed at times $v_{1}=\min(T_{1},Y_{1})=Y_{1}, v_{2}=\min(T_{2},v_{1}+Y_{2})=T_{2}, v_{3}=\min(T_{3},v_{2}+Y_{3})=T_{3}$ and $v_{4}=\min(T_{4},v_{3}+Y_{4})=T_{3}+Y_{4}$, respectively, as shown in Figure~\ref{fig:realization:failure:2}. Note that, we used $v_{i}, i=1,2,3,$ for the realizations of renewal instants $V_{i}$ in $\textup{R}_2$HP.

The curves of $C_{1}(T)$ and $C_{2}(T)$ are given in Figure~\ref{fig:T12:T22},
where the vertical axis standards for $C_{i}(T)$, $i=1,2$,
and the horizontal axis stands for $T$, and the corresponding
optimal replacement times are $T_{12}^{\ast}=1.16$ and $T_{22}^{\ast}=1.3$.
On the right hand side in Figure~\ref{fig:T12:T22}, the blue curve denotes the upper bound and red curve denotes the lower bound for $C_{2}(T)$
obtained from Proposition~\ref{cor:C:2}.
We can see that the blue and red curves on the right hand side in Figure~\ref{fig:T12:T22}, and hence the corresponding optimal replacement times
are very close to each other. This illustrates that the lower and upper bounds in Proposition~\ref{cor:C:2} are very tight.

\begin{figure}[htbp]
    \centering
    \includegraphics[width=0.35\textwidth]{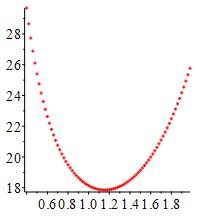}
     \includegraphics[width=0.35\textwidth]{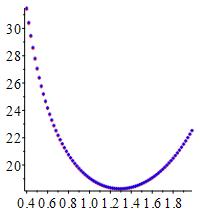}
     \caption{Plots of $C_{1}(T)$ (left) and $C_{2}(T)$ (right) with minimizers being the optimal replacement times $T_{12}^{\ast}$ (left) and $T_{22}^{\ast}$ (right), where the blue curve denotes the upper bound $\overline{C}_{2}(T)$ and red curve denotes the lower bound $\underline{C}_{2}(T)$ for $C_{2}(T)$ (right).}
  \label{fig:T12:T22}
\end{figure}

\paragraph{(iii) Renewal Hawkes Process with Maximum of Endogenous and Exogenous Variables}

\begin{figure}[htbp]
    \centering
    \includegraphics[width=0.7\textwidth]{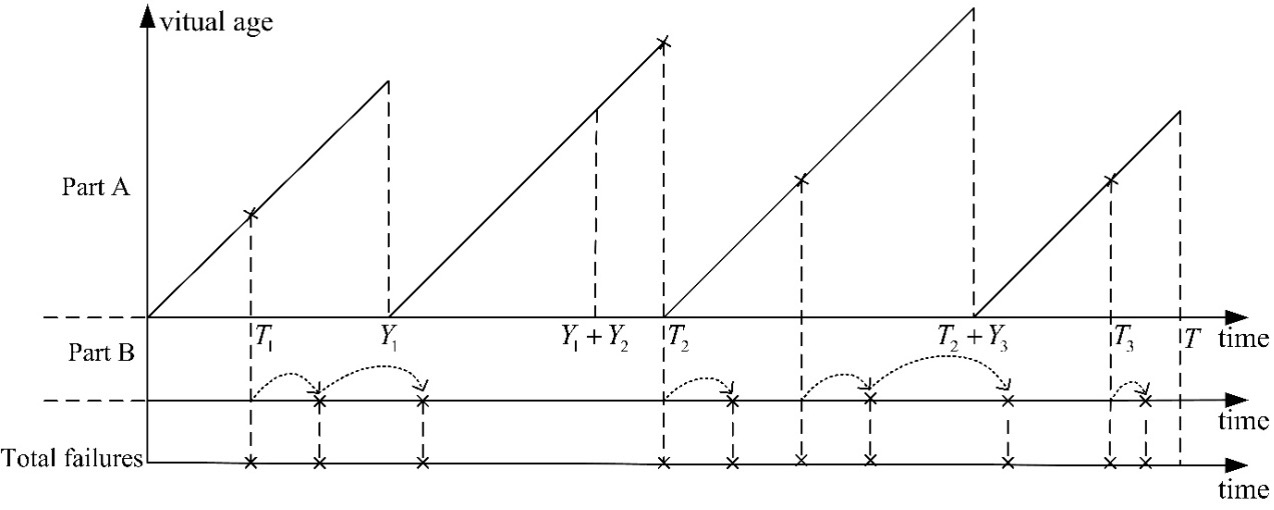}
    \caption{A realization of the failure process of the system in a replacement cycle under renewal Hawkes process with maximum of endogenous and exogenous variables. The dotted arc with an arrow denotes the relationship between parent and its offspring.}
  \label{fig:realization:failure:3}
\end{figure}

As shown in Figure~\ref{fig:realization:failure:3} for a renewed cycle, failures in part A of the system also result from its age, and there are four failures in the interval $[0,T_{3})$. 
Except for the second failure occurred at time $t=T_{2}$ with replacement, they are treated with minimal repairs with negligible times. Failures in part B of the system result from failures occurred in part A and itself with cascading effects, in the intervals $[0,Y_{1}), [Y_{1},T_{2}),[T_{2},T_{2}+Y_{3})$ and after time $t=T_{2}+Y_{3}$,
there are one, one (which is the second generation offspring), two and two (including one second generation offspring) failures respectively, and they are treated with minimal repairs immediately with negligible time after each failure occurs. The whole system is renewed at time $T$, but part A is renewed at times $w_{1}=\max(T_{1},Y_{1})=Y_{1}, w_{2}=\max(T_{2},w_{1}+Y_{2})=T_{2}$ and $w_{3}=\max(T_{3},w_{2}+Y_{3})=T_{2}+Y_{3}$, respectively, before the system renewal time $T$, as shown in Figure~\ref{fig:realization:failure:3}. Note that, we used $w_{i}, i=1,2,3,$ for the realizations of renewal instants $W_{i}$ in $\textup{R}_3$HP.

The curves of $C_{1}(T)$ and $C_{2}(T)$ are given in Figure~\ref{fig:T13:T23},
where the vertical axis standards for $C_{i}(T)$, $i=1,2$,
and the horizontal axis stands for $T$,
and the corresponding
optimal replacement times are $T_{13}^{\ast}=1.20$ and $T_{23}^{\ast}=1.36$.
On the right hand side in Figure~\ref{fig:T13:T23}, the blue curve denotes the upper bound and red curve denotes the lower bound for $C_{2}(T)$
obtained from Proposition~\ref{cor:C:2}.
We can see that the blue and red curves on the right hand side in Figure~\ref{fig:T13:T23}, and hence the corresponding optimal replacement times
are very close to each other. This illustrates that the lower and upper bounds in Proposition~\ref{cor:C:2} are very tight.

\begin{figure}[htbp]
    \centering
    \includegraphics[width=0.35\textwidth]{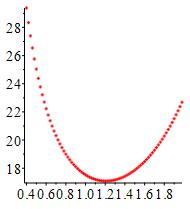}
     \includegraphics[width=0.35\textwidth]{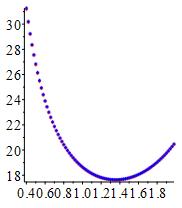}
     \caption{Plots of $C_{1}(T)$ (left) and $C_{2}(T)$ (right) with minimizers being the optimal replacement times $T_{13}^{\ast}$ (left) and $T_{23}^{\ast}$ (right), where the blue curve denotes the upper bound $\overline{C}_{2}(T)$ and red curve denotes the lower bound $\underline{C}_{2}(T)$ for $C_{2}(T)$ (right).}
  \label{fig:T13:T23}
\end{figure}

\paragraph{Further Discussions}

We can see from Figure~\ref{fig:T11:T21}, Figure~\ref{fig:T12:T22} and Figure~\ref{fig:T13:T23}
that under all the three renewal Hawkes processes, $C_{1}(T)$ and $C_{2}(T)$ are convex in $t$,
and they are both decreasing when $t$ is small and increasing when $t$ is large so that
there exist unique optimal placement times $T_{1j}^{\ast}$ and $T_{2j}^{\ast}$ under $\textup{R}_j$Hawkes process for $j=1,2,3$.
This is consistent with the intuition.
We recall from Proposition~\ref{prop:C:1} that $C_{1}(T)=\frac{c_{f}\mathbb{E}[N(T)]+c_{Ip}\mathbb{E}[N_{R}(T)]+c_{p}}{T}$.
Since $\mathbb{E}[N(T)],\mathbb{E}[N_{R}(T)]\rightarrow 0$ as $T\rightarrow 0$, we have $C_{1}(T)\rightarrow\infty$ as $T\rightarrow 0$.
Moreover, under the choice that $\mu(t)=2t$, the expectations of intensities $\mathbb{E}[\lambda(t)]$ increase in $t$ (see Figure~\ref{fig:5} and Figure~\ref{fig:6}),
which implies that $\mathbb{E}[N(T)]$ has super linear growth in $T$ (see Figure~\ref{fig:3} and Figure~\ref{fig:4}) and therefore
$C_{1}(T)\rightarrow\infty$ as $T\rightarrow\infty$. This explains why we observe that $C_{1}(T)$ is decreasing when $T$ is small
and increasing when $T$ is large in Figure~\ref{fig:T11:T21}, Figure~\ref{fig:T12:T22} and Figure~\ref{fig:T13:T23}.
By recalling the formula of $C_{2}(T)$ from Proposition~\ref{prop:C:2}, we can similarly explain why
we observe that $C_{2}(T)$ is decreasing when $T$ is small
and increasing when $T$ is large in Figure~\ref{fig:T11:T21}, Figure~\ref{fig:T12:T22} and Figure~\ref{fig:T13:T23}.
Since $C_{1}(T)$ and $C_{2}(T)$ are both decreasing when $T$ is small and increasing when $T$ is large,
the optimal placement times must exist and are finite. In addition, both $T_{1j}^{\ast}$ and $T_{2j}^{\ast}$ are increasing in $j=1,2,3$, that is,
under $\textup{R}_3$Hawkes process, one gets longer optimal replacement time
than under $\textup{R}_2$Hawkes process, which has longer optimal replacement time
than under $\textup{R}_1$Hawkes process.

All these numerical results presented in Figure~\ref{fig:T11:T21}, Figure~\ref{fig:T12:T22} and Figure~\ref{fig:T13:T23} can provide the following managerial insights.
First, in general, the optimal periodic replacement time exists, and the optimal periodic replacements tell the managers when should the whole system be replaced that is optimal for expected loss per unit time in the long-run average. Second, under mild conditions, the expected loss per unit time in the long-run average tends to infinity as time approaches to infinity, i.e., the minimal repairs are not always the best strategy for maintenance of systems with cascading failures.

\section{Conclusion}\label{sec:conclusions}

In the paper, three novel renewal Hawkes processes are proposed by considering exogenous and endogenous renewal issues, which are named as: renewal Hawkes process with exogenous
variable ($\textup{R}_1$Hawkes process), renewal Hawkes process with minimum of endogenous and exogenous variables ($\textup{R}_2$Hawkes process), renewal Hawkes process with maximum of endogenous and exogenous variables ($\textup{R}_3$Hawkes process), respectively.
 The relationship among five related Hawkes
 processes is discussed. Some frameworks or main ideas for these novel renewal Hawkes processes are discussed, which are main sources of our proposed renewal Hawkes processes.
The expectations of the three renewal Hawkes processes are derived via respectively
 establishing three sets of integral equations. Meanwhile, a general common renewal equation for three renewal Hawkes processes is presented, which is consistent with the intuitions.
 The special case of constant exogenous factor is also studied.
 All these results not only extend the theory of renewal Hawkes processes, but also enrich the field of Hawkes processes.
 
As an application, we apply the proposed models to optimization problems that arise in
periodic
 replacement policy for systems with cascading failures.
In our model setup, minimal repairs and replacements are nested within the entire system, which performs preventive replacements at periodic intervals, 
where the main subsystem has its own minimal repairs and regular renewals, whereas the minor subsystem has all minimal repairs; but the entire system conducts periodic replacements. To the best our knowledge, past literature has never attempted to model such issues.
The nesting of maintenance strategies in the system is more in line with the actual complex systems, especially for those with primary and auxiliary cascading failures. This nested maintenance strategy for the main subsystem is more targeted for failure prevention.
 We solve the optimization problems analytically to obtain the optimal replacement times.
 The numerical illustrations rely on numerically computing the expectations
 of the three proposed renewal Hawkes processes via the set
 of integral equations derived earlier in our paper.
 The set of integral equations can be numerically solved by using the direct Riemann integration method. Three sets of recursive equations are given for numerical
 solutions of the expectations.
 
We believe that these novel renewal Hawkes processes can be used
other practical scenarios that appear in management, operations research and other fields.
The future research work includes for example statistical inferences and moments computations on these novel renewal Hawkes processes, which may be driven by
both in theory and practice.

\section*{Acknowledgments}
The authors would like to thank Zheming Cao, Fengming Kang and Yingchen Zhao for research assistance.
Lirong Cui is partially supported by the National Natural Science Foundation of China under grant 72271134.


\bibliographystyle{plain}
\bibliography{Hawkes}


\appendix

\section{Technical Proofs}\label{sec:proofs}

\subsection{Proofs of the Results in Section~\ref{sec:R1:expectation}}

\begin{proof}[Proof of Theorem~\ref{thm:1}]
First, we have the decomposition:
\begin{equation}
m_1(t) = \mathbb{E}\left[N_{1}(t)1_{U_{1}<T_{1}}\right]
+\mathbb{E}\left[N_{1}(t)1_{U_{1}\geq T_{1}}\right]= \Delta_{11} + \Delta_{12},
\end{equation}
where
\[
\Delta_{11} := \mathbb{E}\left[N_{1}(t)1_{U_{1}<T_{1}}\right],
\]
\[
\Delta_{12} := \mathbb{E}\left[N_{1}(t)1_{U_{1}\geq T_{1}}\right].
\]

Furthermore, we have
\begin{equation}
\begin{aligned}
\Delta_{11} & = \mathbb{E}\left[N_{1}(t)1_{U_{1}<T_{1}}\right]= \int_0^{\infty} \mathbb{E}[N_1(t) | U_1 = y < T_1][1-G(y)] f(y) \, \mathrm{d}y \\
& = \int_0^{t} \mathbb{E}[N_1(t-y)][1-G(y)] f(y) \, \mathrm{d}y = \int_0^{t} m_1(t-y) G_f(y) \, \mathrm{d}y,
\end{aligned}
\end{equation}
where \( G_f(y) := [1 - G(y)] f(y) \).

On the other hand, for the term \( \Delta_{12} =\mathbb{E}\left[N_{1}(t)1_{U_{1}\geq T_{1}}\right] \), conditioning on the event \( \{U_1 = y\} \), then we have
\begin{equation}
\begin{aligned}
\Delta_{12} & =  \mathbb{E}\left[N_{1}(t)1_{U_{1}\geq T_{1}}\right]= \int_0^{\infty} \mathbb{E}[N_1(t) | y \geq T_1] \mathbb{P}\{y \geq T_1\} f(y) \, \mathrm{d}y \\
& = \int_0^{t} \mathbb{E}[N_1(t) | y \geq T_1] \mathbb{P}\{y \geq T_1\} f(y) \, \mathrm{d}y + \int_t^{\infty} \mathbb{E}[N_1(t) | y \geq T_1] \mathbb{P}\{y \geq T_1\} f(y) \, \mathrm{d}y
\nonumber \\
 & = \Delta_{121} + \Delta_{121},
\nonumber
\end{aligned}
\nonumber
\end{equation}
where
\begin{equation}
\Delta_{121} := \int_0^{t} \mathbb{E}[N_1(t) | y \geq T_1] \mathbb{P}\{y \geq T_1\} f(y) \, \mathrm{d}y,
\nonumber
\end{equation}
and
\begin{equation}
\Delta_{122} := \int_t^{\infty} \mathbb{E}[N_1(t) | y \geq T_1] \mathbb{P}\{y \geq T_1\} f(y) \, \mathrm{d}y.
\nonumber
\end{equation}

Furthermore, we can decompose the counting process \( N_1(t) \) under the conditions \( \{U_1 = y \geq T_1\} \) and \( \{T_1 = \tau\} \) by time \( t \geq 0 \), which is denoted as
\begin{equation}\label{eqn:13}
N_1(t) | [U_1 =y \geq T_1 = \tau] = 1 + N_{1,0}(t-\tau) + N_{T_1,\tau}( y - \tau) + N_{1,y,0,\tau}(t-y) I_{\{N_{T_1,\tau}(y-\tau) \geq 1\}} + N_{1,1}(t-y),
\end{equation}
where \( N_*(\bullet) \) are all point processes with respective intensity functions, denoted as \( N_*(\bullet) \triangleright \lambda_*(\bullet) \), i.e., the point process \( N_*(\bullet) \) has an intensity function \( \lambda_*(\bullet) \), that is
\begin{equation}
\begin{aligned}
\nonumber
N_{1,0}(t) &\triangleright \eta h(t) + \int_0^t \eta h(t-v) \, \mathrm{d}N_{1,0}(v), \\
\nonumber
N_{T_1,\tau}(t) &\triangleright \mu(t+\tau) + \int_0^t \eta h(t-v) \, \mathrm{d}N_{T_1,\tau}(v), \\
\nonumber
N_{1,y,0,\tau}(t) &\triangleright \int_0^{y-\tau} \eta h(t+y-\tau-v) \, \mathrm{d}N_{T_1,\tau}(v) + \int_0^t \eta h(t-v) \, \mathrm{d}N_{1,y,0,\tau}(v),
\nonumber
\end{aligned}
\nonumber
\end{equation}
and \( N_1(t) \) and \( N_{1,1}(t) \) have the same distribution, denoted as \( N_1(t) \stackrel{d}{=} N_{1,1}(t) \),
and \( \tau \) is a realization of \( T_1 \) with probability density function: \( \mu(t) e^{-\int_0^t \mu(v) \mathrm{d}v} \).

The reasons for Equation~\eqref{eqn:13} are that (i) when the first immigrant occurs at the time \( \tau \) which results in the term \( 1 + N_{1,0}(t-\tau) \),
 and \( N_{1,0}(t) \) denotes a point process formed by the offspring of the immigrant at the time \( \tau \); (ii) there are \( N_{T_1,\tau}(y-\tau) \) events during the interval \( (\tau, y] \),
 which results in the point process \( N_{T_1,\tau}(t) \) formed by the immigrants after time \( \tau \) and their offspring; (iii) after the first renewal time \( y \),
  all offspring of individuals occurring during the interval \( (\tau, y] \) from a point process \( N_{1,y,0,\tau}(t-y) \) (here it requires that \( N_{T_1,\tau}(y-\tau) \geq 1 \),
   otherwise, there are not any offspring) ; (iv) the immigrants after the renewal time \( y \) and their offspring from a point process \( N_{1,1}(t-y) \) which has the same distribution with the original process \( N_1(t) \).
  The decomposition of $N_1(t) | [U_1 =y \geq T_1 = \tau]$ is shown schematically in Figure~\ref{fig7}.
  The process and idea for this method is that it gradually decomposes the RHP into some known point processes by analyzing the evolution process of RHP. Then, by integrating these known point processes, the original RHP can be represented as a sum of known point processes, and the expectation of RHP can be given based on this result.
The feature of this method is that it is based on the deep evolution mechanism of RHP, gradually decomposing it to some known point processes, expressing RHP as the sum of the known point processes, and then giving the expectation of RHP.
The decomposition method may be applied for other Hawkes processes and its related problems.

\begin{figure}[h]
   \centering
   \includegraphics[width=0.7\textwidth]{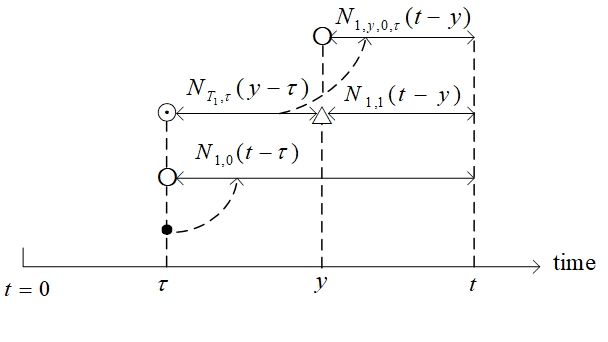}
   \caption{The schematic diagram for decomposition of $N_1(t) | [U_1 =y \geq T_1 = \tau]$.}
   \label{fig7}
\end{figure}

Before our derivation on expectations, based on Lemma~\ref{lem:zero} presented below, we can simplify one term in Equation~\eqref{eqn:13} as
\begin{equation}\label{eqn:14}
N_{1,y,0,\tau}(t-y) I_{\{N_{T_1,\tau}(y-\tau) \geq 1\}} \equiv N_{1,y,0,\tau}(t-y).
\end{equation}

\begin{lemma}\label{lem:zero}
For a non-negative continued function \( \varphi(t) \), if the following equation holds
\[
\varphi(t) = \int_0^t \eta h(t-v) \varphi(v) \, \mathrm{d}v,
\]
then \(\varphi(t) \equiv 0\).
\end{lemma}

\begin{proof}
It is assumed \(\varphi(t) > 0\) for \(t\) in some interval, then taking Laplace transform, we can get
\[
\varphi^*(s) = \eta h^*(s) \varphi^*(s),
\]
i.e., \(h^*(s) = 0\), which is a contradiction, and this completes the proof.
\end{proof}

Based on Lemma~\ref{lem:zero}, we can say that Equation~\eqref{eqn:14} holds.
This is because if \(I_{\{N_{T_{1},\tau}(y-\tau) \geq 1\}} = 0\), then it has,
\[
\int_0^{y-\tau} \eta h(t-v) \, \mathrm{d}N_{T_{1},\tau}(v) = 0,
\]
that is to say,
\[
N_{1,y,0,\tau}(t) \triangleright \int_0^{y-\tau} \eta h(t+y-\tau-v) \, \mathrm{d}N_{T_{1},\tau}(v) + \int_0^t \eta h(t-v) \, \mathrm{d}N_{1,y,0,\tau}(v) = \int_0^t \eta h(t-v) \, \mathrm{d}N_{1,y,0,\tau}(v).
\]

Furthermore, the following equality can be derived:
\[
\frac{\mathrm{d} \mathbb{E}[N_{1,y,0,\tau}(t)]}{\mathrm{d}t} = \int_0^t \eta h(t-v) \frac{\mathrm{d} \mathbb{E}[N_{1,y,0,\tau}(v)]}{\mathrm{d}v} \, \mathrm{d}v,
\]
i.e.,
\[
\frac{\mathrm{d} \mathbb{E}[N_{1,y,0,\tau}(t)]}{\mathrm{d}t} \equiv 0,
\]
which implies that $N_{1,y,0,\tau}(t) \equiv 0$ a.s.
Thus Equation~\eqref{eqn:14} is proved.

Based on Equation~\eqref{eqn:14}, Equation~\eqref{eqn:13} can be rewritten as
\begin{equation}\label{eqn:15}
N_1(t) \mid [U_1 = y \geq T_1 = \tau] = 1 + N_{1,0}(t-\tau) + N_{T_1,\tau}(y-\tau) + N_{1,y,0,\tau}(t-y) + N_{1,1}(t-y).
\end{equation}

Now let us consider the term \(\Delta_{121} = \int_0^t \mathbb{E}[N_1(t) \mid y \geq T_1] \mathbb{P}\{y \geq T_1\} f(y) \, \mathrm{d}y\). Furthermore,  based on Equation~\eqref{eqn:15}, we have
\begin{equation}
\begin{aligned}
\Delta_{121} & = \int_0^t \mathbb{E}[N_1(t) \mid y \geq T_1] \mathbb{P}\{y \geq T_1\} f(y) \, \mathrm{d}y \\
\nonumber
& = \int_0^t \Bigg[\int_0^y [1 + \mathbb{E}[N_{1,0}(t-\tau)] + \mathbb{E}[N_{T_1,\tau}(y-\tau)]
\nonumber
\\
&\qquad\qquad\qquad\qquad\qquad
+ \mathbb{E}[N_{1,y,0,\tau}(t-y)] + \mathbb{E}[N_{1,1}(t-y)]] g(\tau) \, \mathrm{d}\tau \Bigg] f(y) \, \mathrm{d}y \\
\nonumber
& = \int_0^t \int_0^y  \left[1 + k_{1,0}(t-\tau) + k_{T_1,\tau}(y-\tau) + k_{1,y,0,\tau}(t-y) + m_1(t-y)\right] g(\tau) \, f(y) \, \mathrm{d}\tau \, \mathrm{d}y, \\
\nonumber
\end{aligned}
\end{equation}
where
\begin{align*}
& k_{1,0}(t-\tau) = \mathbb{E}[N_{1,0}(t-\tau)], \\
& k_{T_1,\tau}(y-\tau) = \mathbb{E}[N_{T_1,\tau}(y-\tau)], \\
& k_{1,y,0,\tau}(t-y) = \mathbb{E}[N_{1,y,0,\tau}(t-y)], \\
& m_1(t-y) = \mathbb{E}[N_1(t-y)].
\end{align*}

These expectations of point processes satisfy the following set of equations,
\begin{equation}\label{eqn:16}
\left\{
\begin{aligned}
& \frac{\mathrm{d}k_{1,0}(t)}{\mathrm{d}t} = \eta h(t) + \int_0^t \eta h(t-v) \, \mathrm{d}k_{1,0}(v), \quad \text{with } k_{1,0}(0) = 0, \\
& \frac{\mathrm{d}k_{T_1,\tau}(t)}{\mathrm{d}t} = \mu(t+\tau) + \int_0^t \eta h(t-v) \, \mathrm{d}k_{T_1,\tau}(v), \quad \text{with } k_{T_1,\tau}(0) = 0, \\
& \frac{\mathrm{d}k_{1,y,0,\tau}(t)}{\mathrm{d}t} = \int_0^{y-\tau} \eta h(t+y-\tau-v) \, \mathrm{d}k_{T_1,\tau}(v) + \int_0^t \eta h(t-v) \, \mathrm{d}k_{1,y,0,\tau}(v), \quad \text{with } k_{1,y,0,\tau}(0) = 0.
\end{aligned}
\right.
\end{equation}

Indeed, we can simplify Equation~\eqref{eqn:16}. For simplicity, we only show how to do manipulate on the third equation as follows.
An integral operation can be done for the third equation in Equation~\eqref{eqn:16}, we have
\begin{equation}
\begin{aligned}
k_{1,y,0,\tau}(t) & =  \int_0^t \left[\int_0^{y-\tau} \eta h(s+y-\tau-v) \, \mathrm{d}k_{T_1,\tau}(v) \right]\, \mathrm{d}s + \int_0^t \left[\int_0^s \eta h(s-v) \, \mathrm{d}k_{1,y,0,\tau}(v)\right] \, \mathrm{d}s \\
& = \int_0^{y-\tau} \left[ \int_0^{t} \eta h(s+y-\tau-v) \, \mathrm{d}s \right] \mathrm{d}k_{T_1,\tau}(v) + \int_0^t \left[ \int_v^t \eta h(s-v) \, \mathrm{d}s \right] \mathrm{d}k_{1,y,0,\tau}(v) \\
& = \eta H(t) k_{T_1,\tau}(y-\tau) - \int_0^{y-\tau} \eta [h(y-\tau-v) - h(t+y-\tau-v)] k_{T_1,\tau}(v) \, \mathrm{d}v \\
&\qquad\qquad\qquad + \int_0^t \eta h(t-v) k_{1,y,0,\tau}(v) \, \mathrm{d}v,
\end{aligned}
\nonumber
\end{equation}
where \( H(t):= \int_0^t h(v) \, \mathrm{d}v \).

Thus Equation~\eqref{eqn:16} can be given in an integral form by
\begin{equation}\label{eqn:17}
\left\{
\begin{aligned}
k_{1,0}(t) &= \int_0^t [1 + k_{1,0}(t-v)] \eta h(v) \, \mathrm{d}v, \quad \text{with } k_{1,0}(0) = 0, \\
k_{T_1,\tau}(t) &=  \int_0^t \mu(v+\tau) \, \mathrm{d}v + \int_0^t \eta h(v) k_{T_1,\tau}(t-v) \, \mathrm{d}v, \quad \text{with } k_{T_1,\tau}(0) = 0, \\
k_{1,y,0,\tau}(t) &= \eta H(t) k_{T_1,\tau}(y-\tau)  - \int_0^{y-\tau} \eta [h(y-\tau-v) - h(t+y-\tau-v)] k_{T_1,\tau}(v) \, \mathrm{d}v \\
&\quad + \int_0^t \eta h(t-v) k_{1,y,0,\tau}(v) \, \mathrm{d}v, \quad \text{with } k_{1,y,0,\tau}(0) = 0.
\end{aligned}
\right.
\end{equation}

Now let us consider the term \(\Delta_{122} = \int_t^\infty \mathbb{E}[N_1(t) \mid y \geq T_1] \mathbb{P}\{y \geq T_1\} f(y) \, \mathrm{d}y\).
\begin{equation}
\begin{aligned}
\Delta_{122} &= \int_t^\infty \mathbb{E}[N_1(t) \mid y \geq T_1] \mathbb{P}\{y \geq T_1\} f(y) \, \mathrm{d}y \\
&= \int_t^\infty \left[ \int_0^t \mathbb{E}[N_1(t)] g(\tau) \, \mathrm{d}\tau \right] f(y) \, \mathrm{d}y + \int_t^\infty \left[ \int_t^y \mathbb{E}[N_1(t)] g(\tau) \mathrm{d}\tau \right] f(y)   \mathrm{d}y \\
&= \mathbb{E}[N(t)]G(t)[1 - F(t)] + \mathbb{E}[N(t)]\int_t^\infty [G(y) - G(t)] f(y) \, \mathrm{d}y \\
&= \mathbb{E}[N(t)] \int_t^\infty G(y) f(y) \, \mathrm{d}y,
\end{aligned}
\nonumber
\end{equation}
where \( N(t) \triangleright \mu(t) + \int_0^t \eta h(t-v) \, \mathrm{d}N(v) \).

Summarizing all the above results, we get
\begin{equation}\label{eqn:18}
\begin{aligned}
m_1(t) &= \Delta_{11} + \Delta_{12} = \Delta_{11} + \Delta_{122} + \Delta_{121} \\
&= \int_0^t m_1(t-y) G_f(y) \, \mathrm{d}y + \mathbb{E}[N(t)] \int_t^\infty G(y) f(y) \, \mathrm{d}y \\
&\quad + \int_0^t \int_0^y [1 + k_{1,0}(t-\tau) + k_{T_1,\tau}(y-\tau) + k_{1,y,0,\tau}(t-y) + m_1(t-y)] g(\tau) f(y) \, \mathrm{d}\tau \, \mathrm{d}y \\
&= \int_0^t \int_0^y [1 + k_{1,0}(t-\tau) + k_{T_1,\tau}(y-\tau) + k_{1,y,0,\tau}(t- y)] g(\tau) f(y) \, \mathrm{d}\tau \, \mathrm{d}y \\
&\quad\qquad\qquad + \mathbb{E}[N(t)] \int_t^\infty G(y) f(y) \, \mathrm{d}y + \int_0^t m_1(t-y) f(y) \, \mathrm{d}y,
\end{aligned}
\end{equation}
where
\begin{equation}\label{eqn:19}
\frac{\mathrm{d} \mathbb{E}[N(t)]}{\mathrm{d}t} = \mu(t) + \int_0^t \eta h(t-v) \, \mathrm{d}\mathbb{E}[N(v)], \quad \text{with } \mathbb{E}[N(0)] = 0.
\end{equation}

Similarly, Equation~\eqref{eqn:19} can also be given in an integral form by
\begin{equation}
\mathbb{E}[N(t)] = \int_0^t \mu(v) \, \mathrm{d}v + \int_0^t \eta h(v) \mathbb{E}[N(t-v)] \, \mathrm{d}v.
\end{equation}

For Equation~\eqref{eqn:18}, we can simplify it by using Fubini's theorem, which is given as follows.
\begin{align}
& \int_0^t \int_0^y k_{1,0}(t-\tau) g(\tau) f(y) \, \mathrm{d}\tau \, \mathrm{d}y = \int_0^t \left[ \int_\tau^t f(y) \, \mathrm{d}y \right] k_{1,0}(t-\tau) g(\tau) \, \mathrm{d}\tau \nonumber \\
&= \int_0^t k_{1,0}(t-\tau) g(\tau) [F(t) - F(\tau)] \, \mathrm{d}\tau \nonumber \\
&= \int_0^t k_{1,0}(t-\tau) g(\tau) [1 - F(\tau)] \, \mathrm{d}\tau - [1 - F(t)] \int_0^t k_{1,0}(t-\tau) g(\tau) \, \mathrm{d}\tau \nonumber \\
&= \int_0^t k_{1,0}(t-\tau) F_g(\tau) \, \mathrm{d}\tau - [1 - F(t)] \int_0^t k_{1,0}(t-\tau) g(\tau) \, \mathrm{d}\tau,
\nonumber
\end{align}
where \( F_g(y) := [1 - F(y)] g(y) \).

Thus Equation~\eqref{eqn:18} can be simplified as
\begin{equation}
\begin{aligned}
m_1(t) &= \int_0^t G(y) f(y) \, \mathrm{d}y + \int_0^t k_{1,0}(t-\tau) F_g(\tau) \, \mathrm{d}\tau - [1 - F(t)] \int_0^t k_{1,0}(t-\tau) g(\tau) \, \mathrm{d}\tau \\
&\quad + \int_0^t \int_0^y k_{T_1,\tau}(y-\tau) g(\tau) f(y) \, \mathrm{d}\tau \, \mathrm{d}y + \int_0^t \int_0^y k_{1,y,0,\tau}(t-y) g(\tau) f(y) \, \mathrm{d}\tau \, \mathrm{d}y\\
&\quad\quad + \mathbb{E}[N(t)] \int_t^\infty G(y) f(y) \, \mathrm{d}y + \int_0^t m_1(t-y) f(y) \, \mathrm{d}y.
\end{aligned}
\end{equation}

Thus we can get the following integral equation to \( m_1(t) \) by
\begin{equation}\label{eqn:22}
m_1(t) = \Psi_1(t) + \int_0^t m_1(t-y) f(y) \, \mathrm{d}y = \Psi_1(t) + \int_0^t m_1(t-y) n_1(y) \, \mathrm{d}y,
\end{equation}
where $ n_1(y) = f(y) $, $ \Psi_1(t) $ is independent with  $N_1(t)$ , and it is given by
\begin{equation}\label{eqn:23}
\begin{aligned}
\Psi_1(t) &= \int_0^t G(y) f(y) \, \mathrm{d}y + \int_0^t k_{1,0}(t-\tau) F_g(\tau) \, \mathrm{d}\tau - [1 - F(t)] \int_0^t k_{1,0}(t-\tau) g(\tau) \, \mathrm{d}\tau \\
&\qquad + \int_0^t \int_0^y k_{T_1,\tau}(y-\tau) g(\tau) f(y) \, \mathrm{d}\tau \, \mathrm{d}y + \int_0^t \int_0^y k_{1,y,0,\tau}(t-y) g(\tau) f(y) \, \mathrm{d}\tau \, \mathrm{d}y \\
&\qquad\qquad + m(t)\int_t^\infty G(y) f(y) \, \mathrm{d}y,
\end{aligned}
\end{equation}
where $m(t):=\mathbb{E}[N(t)]$ satisfies the integral equation:
\begin{equation}
m(t) = \int_0^t \mu(v) \, \mathrm{d}v + \int_0^t \eta h(v) m(t-v) \, \mathrm{d}v.
\end{equation}
This completes the proof.
\end{proof}

\begin{proof}[Proof of Corollary~\ref{cor:1}]
Under the assumption that $f(y)=\gamma e^{-\gamma y}$ for some $\gamma>0$, we have
\begin{equation}
m_{1}(t)=\Psi_{1}(t)+\int_{0}^{t}m_{1}(y)\gamma e^{-\gamma(t-y)}\mathrm{d}y.
\end{equation}
By differentiating with respect to $t$, we obtain
\begin{equation}
m'_{1}(t)=\Psi'_{1}(t)+\gamma m_{1}(t)-\gamma\int_{0}^{t}m_{1}(y)\gamma e^{-\gamma(t-y)}\mathrm{d}y,
\end{equation}
which implies that
\begin{equation}
m'_{1}(t)=\Psi'_{1}(t)+\gamma m_{1}(t)-\gamma(m_{1}(t)-\Psi_{1}(t))=\Psi'_{1}(t)+\gamma\Psi_{1}(t),
\end{equation}
which yields that
\begin{equation}
m_{1}(t)=\Psi_{1}(t)+\gamma\int_{0}^{t}\Psi_{1}(s)\mathrm{d}s,
\end{equation}
where we used $m_{1}(0)=\Psi_{1}(0)=0$.
This completes the proof.
\end{proof}

\begin{proof}[Proof of Corollary~\ref{cor:1:exp}]
Under the assumption,
the offspring density function $h(t)=\beta e^{-\beta t}$ for some $\beta>0$.
By taking the Laplace transform with respect to the first integral equation for $k_{1,0}(t)$ in \eqref{k:equations}, we get
\begin{equation}
k^{\ast}_{1,0}(s)=\frac{\eta\beta}{s(s+\beta-\eta\beta)},
\end{equation}
which implies that
\begin{equation}
k_{1,0}(t)=\frac{\eta(e^{\beta(\eta-1)t}-1)}{\eta-1},
\end{equation}
with the understanding that $k_{1,0}(t)=\beta t$ with $\eta=1$.
Alternatively, one can see that the integral equation for $k_{1,0}(t)$ in \eqref{k:equations} is given by
\begin{equation}
k_{1,0}(t)=\eta\left(1-e^{-\beta t}\right)+\eta\beta\int_{0}^{t}k_{1,0}(v)e^{-\beta(t-v)}\mathrm{d}v.
\end{equation}
By differentiating with respect to $t$, we get
\begin{align}
k'_{1,0}(t)&=\eta\beta e^{-\beta t}+\eta\beta k_{1,0}(t)
-\eta\beta^{2}\int_{0}^{t}k_{1,0}(v)e^{-\beta(t-v)}\mathrm{d}v
\nonumber
\\
&=\eta\beta+(\eta-1)\beta k_{1,0}(t),
\end{align}
which implies that
\begin{equation}
k_{1,0}(t)=\frac{\eta(e^{\beta(\eta-1)t}-1)}{\eta-1}.
\end{equation}
The integral equation for $k_{T_1,\tau}(t)$ in \eqref{k:equations} is given by
\begin{equation}
k_{T_1,\tau}(t) = \int_0^t \mu(v+\tau) \, \mathrm{d}v + \int_0^t \eta \beta e^{-\beta(t-v)} k_{T_1,\tau}(v) \, \mathrm{d}v.
\end{equation}
By differentiating with respect to $t$, we get
\begin{align}
k'_{T_{1},\tau}(t)&=\mu(t+\tau)+\eta\beta k_{T_{1},\tau}(t)-\beta\int_0^t \eta \beta e^{-\beta(t-v)} k_{T_1,\tau}(v) \, \mathrm{d}v
\nonumber
\\
&=\mu(t+\tau)+(\eta-1)\beta k_{T_{1},\tau}(t)+\eta\beta\int_0^t \mu(v+\tau) \, \mathrm{d}v,
\end{align}
which implies that
\begin{equation}
k_{T_1,\tau}(t) =\int_{0}^{t}e^{(\eta-1)\beta(t-s)}\left[\mu(s+\tau)+\eta\beta\int_0^s \mu(v+\tau) \, \mathrm{d}v\right]\mathrm{d}s.
\end{equation}
The integral equation for $k_{1,y,0,\tau}(t) $ in \eqref{k:equations} is given by
\begin{align}
k_{1,y,0,\tau}(t) &= \eta \left(1-e^{-\beta t}\right) k_{T_1,\tau}(y-\tau) - \int_0^{y-\tau} \eta\beta \left[e^{-\beta(y-\tau-v)} -e^{-\beta(t+y-\tau-v)}\right] k_{T_1,\tau}(v) \, \mathrm{d}v+
\nonumber
\\
&\quad + \int_0^t \eta \beta e^{-\beta(t-v)} k_{1,y,0,\tau}(v) \, \mathrm{d}v, \quad \text{with } k_{1,y,0,\tau}(0) = 0.
\end{align}
Similar as before, we can compute that
\begin{align}
k_{1,y,0,\tau}(t)&=\int_{0}^{t}e^{(\eta-1)\beta(t-s)}\left[\eta\beta e^{-\beta s}k_{T_1,\tau}(y-\tau)
+\int_{0}^{y-\tau}\eta\beta^{2}e^{-\beta(s+y-\tau-v)}k_{T_{1},\tau}(v)\mathrm{d}v\right]\mathrm{d}s\nonumber
\\
&\qquad
+\int_{0}^{t}e^{(\eta-1)\beta(t-s)}\beta\Bigg[\eta (1-e^{-\beta s}) k_{T_1,\tau}(y-\tau)
\nonumber
\\
&\qquad\qquad\qquad
- \int_0^{y-\tau} \eta\beta \left[e^{-\beta(y-\tau-v)} -e^{-\beta(s+y-\tau-v)}\right] k_{T_1,\tau}(v) \, \mathrm{d}v \Bigg]\mathrm{d}s
\nonumber
\\
&=\int_{0}^{t}e^{(\eta-1)\beta(t-s)}\beta\eta k_{T_1,\tau}(y-\tau) \mathrm{d}s
\nonumber
\\
&\qquad\qquad\qquad
+\int_{0}^{t}e^{(\eta-1)\beta(t-s)}\int_0^{y-\tau} \eta\beta^{2}e^{-\beta(y-\tau-v)} k_{T_1,\tau}(v) \, \mathrm{d}v \mathrm{d}s
\nonumber
\\
&=\frac{\eta(e^{\beta(\eta-1)t}-1)}{\eta-1}\left[k_{T_1,\tau}(y-\tau) -\int_0^{y-\tau} \eta\beta e^{-\beta(y-\tau-v)} k_{T_1,\tau}(v) \, \mathrm{d}v\right].
\end{align}
Moreover, it follows form \eqref{m:t:eqn} that
\begin{equation}
m(t) = \int_0^t \mu(v) \, \mathrm{d}v + \int_0^t \eta\beta e^{-\beta(t-v)} m(v) \, \mathrm{d}v.
\end{equation}
By differentiating with respect to $t$, we get
\begin{align}
m'(t) &= \mu(t) +\eta\beta m(t)-\beta \int_0^t \eta\beta e^{-\beta(t-v)} m(v) \, \mathrm{d}v
\nonumber
\\
&=\mu(t) +\eta\beta m(t)-\beta m(t)+\beta\int_{0}^{t}\mu(v)\mathrm{d}v,
\end{align}
which implies that
\begin{equation}
m(t)=\int_{0}^{t}e^{(\eta-1)\beta(t-s)}\left[\mu(s)+\beta\int_{0}^{s}\mu(v)\mathrm{d}v\right]\mathrm{d}s.
\end{equation}
This completes the proof.
\end{proof}


\begin{proof}[Proof of Corollary~\ref{cor:1:exp:special}]
Under the assumptions, the offspring density function $h(t)=\beta e^{-\beta t}$ for some $\beta>0$,
$f(y)=\gamma e^{-\gamma y}$
for some $\gamma>0$ and $\mu(t)=\mu_{i}$ for every $t_{i-1}\leq t<t_{i}$.
Then, we have
\begin{equation}
k_{1,0}(t)=\frac{\eta(e^{\beta(\eta-1)t}-1)}{\eta-1}.
\end{equation}

Next, we recall that for any $t\geq 0$:
\begin{equation}
k_{T_1,\tau}(t) =\int_{0}^{t}e^{(\eta-1)\beta(t-s)}\left[\mu(s+\tau)+\eta\beta\int_0^s \mu(v+\tau) \, \mathrm{d}v\right]\mathrm{d}s.
\end{equation}
Let $i_{\ast}$ be the unique positive value such that $t_{i_{\ast}-1}\leq\tau<t_{i_{\ast}}$.
Then, for any $0\leq t<t_{i_{\ast}}-\tau$, we have
\begin{align}
k_{T_1,\tau}(t)
&=\int_{0}^{t}e^{(\eta-1)\beta(t-s)}\left[\mu_{t_{i_{\ast}}}+\eta\beta\mu_{t_{i_{\ast}}}s\right]\mathrm{d}s
\nonumber
\\
&=\frac{\mu_{t_{i_{\ast}}}(e^{(\eta-1)\beta t}-1)}{(\eta-1)\beta}+\eta\beta\mu_{t_{i_{\ast}}}
\left(\frac{t}{(1-\eta)\beta}-\frac{1-e^{(\eta-1)\beta t}}{(1-\eta)^{2}\beta^{2}}\right).
\end{align}
In addition, for any $t_{i}-\tau\leq t<t_{i+1}-\tau$ with $i\geq i_{\ast}$, we have
\begin{align}
&k_{T_1,\tau}(t)=\int_{0}^{t_{i_{\ast}}-\tau}e^{(\eta-1)\beta(t-s)}\left[\mu_{t_{i_{\ast}}}+\eta\beta\mu_{t_{i_{\ast}}}s\right]\mathrm{d}s
\nonumber
\\
&\quad
+\sum_{j=i_{\ast}}^{i-1}\int_{t_{j}-\tau}^{t_{j+1}-\tau}e^{(\eta-1)\beta(t-s)}\left[\mu_{j+1}+\eta\beta(t_{i_{\ast}}-\tau)\mu_{i_{\ast}}
+\eta\beta\sum_{\ell=i_{\ast}+1}^{j}(t_{\ell}-t_{\ell-1})\mu_{\ell}+\eta\beta(s+\tau-t_{j})\mu_{j+1}\right]\mathrm{d}s
\nonumber
\\
&\qquad\qquad
+\int_{t_{i}-\tau}^{t}e^{(\eta-1)\beta(t-s)}\left[\mu_{i+1}+\eta\beta(t_{i_{\ast}}-\tau)\mu_{i_{\ast}}
+\eta\beta\sum_{\ell=i_{\ast}+1}^{i}(t_{\ell}-t_{\ell-1})\mu_{\ell}+\eta\beta(s+\tau-t_{i})\mu_{i+1}\right]\mathrm{d}s
\nonumber
\\
&=\mu_{t_{i_{\ast}}}e^{(\eta-1)\beta t}\frac{e^{(1-\eta)\beta(t_{i_{\ast}}-\tau)}-1}{(1-\eta)\beta}
+\eta\beta\mu_{t_{i_{\ast}}}e^{(\eta-1)\beta t}\left(\frac{(t_{i_{\ast}}-\tau)e^{(1-\eta)\beta(t_{i_{\ast}}-\tau)}}{(1-\eta)\beta}-\frac{e^{(1-\eta)\beta(t_{i_{\ast}}-\tau)}-1}{(1-\eta)^{2}\beta^{2}}\right)
\nonumber
\\
&\qquad
+\sum_{j=i_{\ast}}^{i-1}e^{(\eta-1)\beta t}\frac{e^{(1-\eta)\beta(t_{j+1}-\tau)}-e^{(1-\eta)\beta(t_{j}-\tau)}}{(1-\eta)\beta}\left[\mu_{j+1}+\eta\beta(t_{i_{\ast}}-\tau)\mu_{i_{\ast}}
+\eta\beta\sum_{\ell=i_{\ast}+1}^{j}(t_{\ell}-t_{\ell-1})\mu_{\ell}\right]
\nonumber
\\
&\qquad\qquad
+\sum_{j=i_{\ast}}^{i-1}e^{(\eta-1)\beta(t+\tau-t_{j})}\eta\beta\mu_{j+1}
\left[\frac{(t_{j+1}-t_{j})e^{(1-\eta)\beta(t_{j+1}-t_{j})}}{(1-\eta)\beta}
-\frac{e^{(1-\eta)\beta(t_{j+1}-t_{j})}-1}{(1-\eta)^{2}\beta^{2}}\right]
\nonumber
\\
&\qquad\qquad\qquad
+e^{(\eta-1)\beta t}\frac{e^{(1-\eta)\beta t}-e^{(1-\eta)\beta(t_{i}-\tau)}}{(1-\eta)\beta}\left[\mu_{i+1}+\eta\beta(t_{i_{\ast}}-\tau)\mu_{i_{\ast}}
+\eta\beta\sum_{\ell=i_{\ast}+1}^{i}(t_{\ell}-t_{\ell-1})\mu_{\ell}\right]
\nonumber
\\
&\qquad\qquad\qquad\qquad
+e^{(\eta-1)\beta(t+\tau-t_{i})}\eta\beta\mu_{i+1}
\left[\frac{(t-t_{i})e^{(1-\eta)\beta(t-t_{i})}}{(1-\eta)\beta}
-\frac{e^{(1-\eta)\beta(t-t_{i})}-1}{(1-\eta)^{2}\beta^{2}}\right].
\end{align}

Next, we recall that for any $t\geq 0$:
\begin{equation}
k_{1,y,0,\tau}(t)=\frac{\eta(e^{\beta(\eta-1)t}-1)}{\eta-1} \left[k_{T_1,\tau}(y-\tau) -\int_0^{y-\tau} \eta\beta e^{-\beta(y-\tau-v)} k_{T_1,\tau}(v) \, \mathrm{d}v\right].
\end{equation}
We also recall that for any $t\geq 0$:
\begin{equation}
k_{T_1,\tau}(t) = \int_0^t \mu(v+\tau) \, \mathrm{d}v + \int_0^t \eta \beta e^{-\beta(t-v)} k_{T_1,\tau}(v) \, \mathrm{d}v,
\end{equation}
which implies that
\begin{equation}
 \int_0^{y-\tau} \eta \beta e^{-\beta(y-\tau-v)} k_{T_1,\tau}(v) \, \mathrm{d}v
 = \int_0^{y-\tau}\mu(v+\tau) \, \mathrm{d}v -k_{T_1,\tau}(y-\tau).
\end{equation}
Therefore, for any $t\geq 0$:
\begin{equation}
k_{1,y,0,\tau}(t)=\frac{\eta(e^{\beta(\eta-1)t}-1)}{\eta-1}\left[(1-\beta)^2 k_{T_1,\tau}(y-\tau) -\eta\int_{0}^{y-\tau}\mu(v+\tau)\mathrm{d}v\right].
\end{equation}
When $y<t_{i_{\ast}}$, we have
\begin{align}
k_{1,y,0,\tau}(t)&=\frac{\eta(e^{\beta(\eta-1)t}-1)}{\eta-1}(1-\beta)^2
\nonumber
\\
&\qquad\qquad\cdot\left[\frac{\mu_{t_{i_{\ast}}}(e^{(\eta-1)\beta(y-\tau)}-1)}{(\eta-1)\beta}+\eta\beta\mu_{t_{i_{\ast}}}
\left(\frac{y-\tau}{(1-\eta)\beta}-\frac{1-e^{(\eta-1)\beta(y-\tau)}}{(1-\eta)^{2}\beta^{2}}\right)\right]
\nonumber
\\
&\qquad
-\frac{\eta(e^{\beta(\eta-1)t}-1)}{\eta-1}\eta (y-\tau)\mu_{i_{\ast}}.
\end{align}
When $t_{i}\leq y<t_{i+1}$ for some $i\geq i_{\ast}$, we have
\begin{align}
k_{1,y,0,\tau}(t)&=\frac{\eta(e^{\beta(\eta-1)t}-1)}{\eta-1}(1-\beta)^2
\Bigg\{
\mu_{t_{i_{\ast}}}e^{(\eta-1)\beta(y-\tau)}\frac{e^{(1-\eta)\beta(t_{i_{\ast}}-\tau)}-1}{(1-\eta)\beta}
\nonumber
\\
&+\eta\beta\mu_{t_{i_{\ast}}}e^{(\eta-1)\beta(y-\tau)}\left(\frac{(t_{i_{\ast}}-\tau)e^{(1-\eta)\beta(t_{i_{\ast}}-\tau)}}{(1-\eta)\beta}-\frac{e^{(1-\eta)\beta(t_{i_{\ast}}-\tau)}-1}{(1-\eta)^{2}\beta^{2}}\right)
\nonumber
\\
&\qquad
+\sum_{j=i_{\ast}}^{i-1}e^{(\eta-1)\beta(y-\tau)}\frac{e^{(1-\eta)\beta(t_{j+1}-\tau)}-e^{(1-\eta)\beta(t_{j}-\tau)}}{(1-\eta)\beta}
\nonumber
\\
&\qquad\qquad\qquad\cdot\left[\mu_{j+1}+\eta\beta(t_{i_{\ast}}-\tau)\mu_{i_{\ast}}
+\eta\beta\sum_{\ell=i_{\ast}+1}^{j}(t_{\ell}-t_{\ell-1})\mu_{\ell}\right]
\nonumber
\\
&\qquad\qquad
+\sum_{j=i_{\ast}}^{i-1}e^{(\eta-1)\beta(y-t_{j})}\eta\beta\mu_{j+1}
\left[\frac{(t_{j+1}-t_{j})e^{(1-\eta)\beta(t_{j+1}-t_{j})}}{(1-\eta)\beta}
-\frac{e^{(1-\eta)\beta(t_{j+1}-t_{j})}-1}{(1-\eta)^{2}\beta^{2}}\right]
\nonumber
\\
&\qquad\qquad\qquad
+\frac{1-e^{(1-\eta)\beta(t_{i}-y)}}{(1-\eta)\beta}\left[\mu_{i+1}+\eta\beta(t_{i_{\ast}}-\tau)\mu_{i_{\ast}}
+\eta\beta\sum_{\ell=i_{\ast}+1}^{i}(t_{\ell}-t_{\ell-1})\mu_{\ell}\right]
\nonumber
\\
&\qquad\qquad\qquad\qquad
+e^{(\eta-1)\beta(y-t_{i})}\eta\beta\mu_{i+1}
\left[\frac{(y-\tau-t_{i})e^{(1-\eta)\beta(y-\tau-t_{i})}}{(1-\eta)\beta}
-\frac{e^{(1-\eta)\beta(t-t_{i})}-1}{(1-\eta)^{2}\beta^{2}}\right]\Bigg\}
\nonumber
\\
&\qquad
-\frac{\eta(e^{\beta(\eta-1)t}-1)}{\eta-1}\eta \left[(t_{i_{\ast}}-\tau)\mu_{i_{\ast}}+\sum_{j=i_{\ast}}^{i-1}(t_{j+1}-t_{j})\mu_{j+1}
+(y-t_{i})\mu_{i+1}\right].
\end{align}

Moreover, for any $t_{i}\leq t<t_{i+1}$,
\begin{align}
m(t)&=\sum_{j=1}^{i}\int_{t_{j-1}}^{t_{j}}e^{(\eta-1)\beta(t-s)}\left[\mu(s)+\beta\int_{0}^{s}\mu(v)\mathrm{d}v\right]\mathrm{d}s
+\int_{t_{i}}^{t}e^{(\eta-1)\beta(t-s)}\left[\mu(s)+\beta\int_{0}^{s}\mu(v)\mathrm{d}v\right]\mathrm{d}s
\nonumber
\\
&=\sum_{j=1}^{i}\int_{t_{j-1}}^{t_{j}}e^{(\eta-1)\beta(t-s)}\left[\mu_{j}+\beta\sum_{\ell=1}^{j-1}\mu_{\ell}(t_{\ell}-t_{\ell-1})
+\beta(s-t_{j-1})\mu_{j}\right]\mathrm{d}s
\nonumber
\\
&\qquad\qquad\qquad
+\int_{t_{i}}^{t}e^{(\eta-1)\beta(t-s)}\left[\mu_{i+1}+\beta\sum_{\ell=1}^{i}\mu_{\ell}(t_{\ell}-t_{\ell-1})
+\beta(s-t_{i})\mu_{i+1}\right]\mathrm{d}s
\nonumber
\\
&=e^{(\eta-1)\beta t}\sum_{j=1}^{i}\left[\mu_{j}+\beta\sum_{\ell=1}^{j-1}\mu_{\ell}(t_{\ell}-t_{\ell-1})\right]\frac{e^{(1-\eta)\beta t_{j}}-e^{(1-\eta)\beta t_{j-1}}}{(1-\eta)\beta}
\nonumber
\\
&\qquad
+\frac{e^{(\eta-1)\beta t}}{\beta(\eta-1)^{2}}\sum_{j=1}^{i}\mu_{j}\left[1+(\eta-1)\beta t_{j-1}-\left[1+(\eta-1)\beta t_{j}\right] e^{(1-\eta)\beta (t_{j}-t_{j-1})}\right]
\nonumber
\\
&\qquad\qquad
+e^{(\eta-1)\beta t}\left[\mu_{i+1}+\beta\sum_{\ell=1}^{i}\mu_{\ell}(t_{\ell}-t_{\ell-1})\right]
\frac{e^{(1-\eta)\beta t}-e^{(1-\eta)\beta t_{i}}}{(1-\eta)\beta}
\nonumber
\\
&\qquad\qquad\qquad
+\frac{e^{(\eta-1)\beta t}}{\beta(\eta-1)^{2}}\mu_{i+1}\left[1+(\eta-1)\beta t_{i}-\left[1+(\eta-1)\beta t\right] e^{(1-\eta)\beta (t-t_{i})}\right].
\end{align}
This completes the proof.
\end{proof}

\subsection{Proofs of the Results in Section~\ref{sec:R2:expectation}}

\begin{proof}[Proof of Theorem~\ref{thm:2}]
First, we have the decomposition:
\begin{equation}\label{eqn:25}
m_2(t) =\mathbb{E}\left[N_{2}(t)1_{V_{1}=Y_{1}}\right]
+\mathbb{E}\left[N_{2}(t)1_{V_{1}=T_{1}}\right].
\end{equation}

For the first term in Equation~\eqref{eqn:25}, conditioning on $ \{Y_1 = y\}$, we have
\begin{equation}\label{eqn:26}
\begin{aligned}
\mathbb{E}\left[N_{2}(t)1_{V_{1}=Y_{1}}\right]
&= \int_0^t \mathbb{E}[N_2(t) \mid V_1 = y][1 - G(y)] f(y) \, \mathrm{d}y \\
&= \int_0^t \mathbb{E}[N_2(t-y)][1 - G(y)] f(y) \, \mathrm{d}y \\
&= \int_0^t m_2(t-y) G_f(y) \, \mathrm{d}y.
\end{aligned}
\end{equation}

For the second term in Equation~\eqref{eqn:25}, conditioning on \( \{T_1 = \tau\} \), we have
\begin{equation}
\begin{aligned}
\mathbb{E}\left[N_{2}(t)1_{V_{1}=T_{1}}\right]&= \int_0^t \mathbb{E}[N_2(t) \mid V_1 = \tau][1 - F(\tau)] g(\tau) \, \mathrm{d}\tau \\
&= \int_0^t \mathbb{E}[N_2(t) \mid V_1 = \tau] F_g(\tau) \, \mathrm{d}\tau,
\end{aligned}
\end{equation}
where $ F_g(t) := [1 - F(t)] g(t) $.
Furthermore, we have
\begin{equation}
N_2(t) \mid [V_1 = \tau] =
\begin{cases}
0, & \text{if } \tau > t, \\
1 + N_{2,0}(t-\tau) + N_{2,1}(t-\tau), & \text{if } \tau \leq t,
\end{cases}
\nonumber
\end{equation}
where $ N_{2,0}(t-\tau) $ $ (t \geq \tau )$ is the point process formed by the offspring of the immigrant at time $\tau$, which is a self-excited point process with basic intensity $\eta h(t-\tau)$ and excitation function $\eta h(t-\tau)$,
and $N_{2,1}(t-\tau)$ ($t \geq \tau$) is the point process formed by the immigrants after time $\tau$ and their offspring, which is an $\textup{R}_2$Hawkes process having the same distribution with $N_{2}(t-\tau)$.
That is to say, we have
\begin{align*}
N_{2,0}(t) &\triangleright \eta h(t) + \int_0^t \eta h(t-v) \, \mathrm{d}N_{2,0}(v), \\
N_{2,1}(t) &\stackrel{d}{=} N_2(t).
\end{align*}

Let $k_{2,0}(t) = \mathbb{E}[N_{2,0}(t)]$. Thus we have
\begin{equation}
\begin{aligned}
&\int_0^t \mathbb{E}[N_2(t) \mid V_1 = T_1 = \tau] P \{T_1 = \tau\} \, \mathrm{d}\tau \\
&= \int_0^t \mathbb{E}[N_2(t) \mid V_1 = \tau][1 - F(t)] g(\tau) \, \mathrm{d}\tau + \int_0^t \mathbb{E}[N_2(t) \mid V_1 = \tau][F(t) - F(\tau)] g(\tau) \, \mathrm{d}\tau \\
&= \int_0^t [1 + k_{2,0}(t - \tau) + m_2(t - \tau)] F_g(\tau) \, \mathrm{d}\tau.
\end{aligned}
\nonumber
\end{equation}

Based on Equation~\eqref{eqn:26} and the results presented above, we can get the following integral equation for $m_2(t)$,
\begin{equation}\label{eqn:28}
\begin{aligned}
m_2(t) &= \int_0^t m_2(t-y) G_f(y) \, \mathrm{d}y + \int_0^t [1 + k_{2,0}(t-\tau) + m_2(t-\tau)] F_g(\tau) \, \mathrm{d}\tau \\
&= \int_0^t [1 + k_{2,0}(t-y)] F_g(y) \, \mathrm{d}y + \int_0^t m_2(t-y) [G_f(y) + F_g(y)] \, \mathrm{d}y.
\end{aligned}
\end{equation}
That is to say,
\begin{equation}\label{eqn:29}
m_2(t) = \Psi_2(t) + \int_0^t m_2(t-y) [G_f(y) + F_g(y)] \, \mathrm{d}y = \Psi_2(t) + \int_0^t m_2(t-y) n_2(y) \, \mathrm{d}y,
\end{equation}
where $\Psi_2(t) = \int_0^t [1 + k_{2,0}(t-y)] F_g(y) \, \mathrm{d}y$ and $n_2(y) = G_f(y) + F_g(y)$.

As defined above for the function $k_{2,0}(t)$, it is the expectation of a self-excited Hawkes process with the following intensity function
\begin{equation}
\lambda_{2,0}(t) = \eta h(t) + \eta \int_0^t h(t-v) \, \mathrm{d}N_{2,0}(v).
\nonumber
\end{equation}

Based on Theorem~1 in \cite{Cui2020}, we have
\begin{equation}
k_{2,0}(t) = \mathbb{E}[N_{2,0}(t)] = \int_0^t \mathbb{E}[\lambda_{2,0}(v)] \mathrm{d}v.
\end{equation}

Similar to \cite{Chen2018}, using Fubini's theorem and integration by parts, we get
\begin{equation}
\int_0^t \mathbb{E}[\lambda_{2,0}(t)] \mathrm{d}t = \int_0^t \eta h(v) \mathrm{d}v + \int_0^t \eta h(v) \mathbb{E}[N_{2,0}(t-v)] \mathrm{d}v.
\nonumber
\end{equation}

Thus we get the following integral equation for $k_{2,0}(t)$,
\begin{equation}\label{eqn:31}
k_{2,0}(t) = \int_0^t [1 + k_{2,0}(t-v)] \eta h(v) \mathrm{d}v.
\end{equation}

Combining Equations~\eqref{eqn:28} and \eqref{eqn:31}, we get a set of integral equations for \( m_2(t) \) and \( k_{2,0}(t) \) as follows,
\begin{equation}
\begin{cases}
m_2(t) = \int_0^t [1 + k_{2,0}(t-v)] F_g(v) \mathrm{d}v + \int_0^t m_2(t-v) [G_f(v) + F_g(v)] \mathrm{d}v, \\
k_{2,0}(t) = \int_0^t [1 + k_{2,0}(t-v)] \eta h(v) \mathrm{d}v.
\end{cases}
\end{equation}

Furthermore, using Laplace transform, we get a set of linear equations for \( m_2^*(s) \) and \( k_{2,0}^*(s) \) as follows,
\begin{equation}\label{eqn:33}
\begin{cases}
m_2^*(s) = m_2^*(s) G_f^*(s) + \left[ \frac{1}{s} + k_{2,0}^*(s) + m_2^*(s) \right] F_g^*(s), \\
k_{2,0}^*(s) = \eta \left[ \frac{1}{s} + k_{2,0}^*(s) \right] h^*(s),
\end{cases}
\end{equation}
where the superscript symbol * denotes the Laplace transform, i.e., \( w^*(s) = \int_0^\infty e^{-st} w(t) \mathrm{d}t \).

Solving Equation~\eqref{eqn:33}, we get
\begin{equation}\label{eqn:34}
m_2^*(s) = \frac{F_g^*(s)}{s[1 - \eta h^*(s)][1 - G_f^*(s) - F_g^*(s)]}.
\end{equation}
This completes the proof.
\end{proof}


\begin{proof}[Proof of Corollary~\ref{cor:2}]
Since $\mu(t)=\mu_{i}$ for every $t_{i-1}\leq t<t_{i}$, we have that for any $t_{i}\leq t<t_{i+1}$,
\begin{equation}
1-G(t)=e^{-\int_{0}^{t}\mu(v)\mathrm{d}v}=e^{-\sum_{j=1}^{i}\mu_{j}(t_{j}-t_{j-1})}e^{-\mu_{i+1}(t-t_{i})},
\end{equation}
and
\begin{equation}
g(t)=\mu_{i+1}e^{-\sum_{j=1}^{i}\mu_{j}(t_{j}-t_{j-1})}e^{-\mu_{i+1}(t-t_{i})}.
\end{equation}
It follows that
\begin{align}
n_{2}(t)
&=(1-G(t))f(t)+(1-F(t))g(t)
\nonumber
\\
&=(\gamma+\mu_{i+1})e^{-\sum_{j=1}^{i}\mu_{j}(t_{j}-t_{j-1})}e^{-\mu_{i+1}(t-t_{i})}e^{-\gamma t}
\nonumber
\\
&=\alpha_{i}(\gamma+\mu_{i+1})e^{-(\gamma+\mu_{i+1})t},
\end{align}
where $\alpha_{i}:=e^{-\sum_{j=1}^{i}\mu_{j}(t_{j}-t_{j-1})}e^{\mu_{i+1}t_{i}}$.
Thus, we have
\begin{equation}
m_{2}(t)=\Psi_{2}(t)+\int_{0}^{t}m_{2}(v)\alpha_{i}(\gamma+\mu_{i+1})e^{-(\gamma+\mu_{i+1})(t-v)}\mathrm{d}v.
\end{equation}
By differentiating the above equation with respect to $t$, we obtain
\begin{equation}
m'_{2}(t)=\Psi'_{2}(t)+\alpha_{i}(\gamma+\mu_{i+1})m_{2}(t)
-(\gamma+\mu_{i+1})\int_{0}^{t}m_{2}(v)\alpha_{i}(\gamma+\mu_{i+1})e^{-(\gamma+\mu_{i+1})(t-v)}\mathrm{d}v,
\end{equation}
which implies that
\begin{align}
m'_{2}(t)&=\Psi'_{2}(t)+\alpha_{i}(\gamma+\mu_{i+1})m_{2}(t)-(\gamma+\mu_{i+1})(m_{2}(t)-\Psi_{2}(t))
\nonumber
\\
&=\Psi'_{2}(t)+(\gamma+\mu_{i+1})\Psi_{2}(t)+(\alpha_{i}-1)(\gamma+\mu_{i+1})m_{2}(t).
\end{align}
One can solve this first-order linear ODE to obtain for any $t_{i}\leq t<t_{i+1}$:
\begin{equation}
m_{2}(t)=e^{(1-\alpha_{i})(\gamma+\mu_{i+1})(t_{i}-t)}m_{2}(t_{i})
+\int_{t_{i}}^{t}(\Psi'_{2}(s)+(\gamma+\mu_{i+1})\Psi_{2}(s))e^{(1-\alpha_{i})(\gamma+\mu_{i+1})(s-t)}\mathrm{d}s,
\end{equation}
for any $i=0,1,2,\ldots$ with $m_{2}(0)=0$.
This completes the proof.
\end{proof}

\subsection{Proofs of the Results in Section~\ref{sec:R3:expectation}}

\begin{proof}[Proof of Theorem~\ref{thm:3}]
For the $\textup{R}_3$Hawkes process,
we get the following decomposition for  $m_3(t) = \mathbb{E}[N_3(t)]$, i.e.,
\begin{equation}
m_3(t) = \mathbb{E}\left[N_{3}(t)1_{W_{1}=T_{1}}\right]
+\mathbb{E}\left[N_{3}(t)1_{W_{1}=Y_{1}}\right].
\end{equation}

For the term $\mathbb{E}\left[N_{3}(t)1_{W_{1}=T_{1}}\right]$, conditioning on $\{T_1 = \tau\}$, we get
\begin{equation}
\begin{aligned}
&\mathbb{E}\left[N_{3}(t)1_{W_{1}=T_{1}}\right]\\
&= \int_0^t \mathbb{E}[N_3(t-\tau) | W_1 = \tau\} \mathbb{P}\{Y_1 \leq \tau\} g(\tau) \mathrm{d}\tau= \int_0^t \mathbb{E}[N_3(t-\tau)] F(\tau) g(\tau) \mathrm{d}\tau.
\end{aligned}
\end{equation}

For the term $\mathbb{E}\left[N_{3}(t)1_{W_{1}=Y_{1}}\right]$, we have
\begin{equation}
\begin{aligned}
& \mathbb{E}\left[N_{3}(t)1_{W_{1}=Y_{1}}\right] \\
&= \int_0^t \int_0^y [1 + k_{3,0}(t-\tau) + k_{T_1,\tau}(y-\tau) + k_{3,y,0,\tau}(t-y) + m_3(t-y)] g(\tau) f(y) \mathrm{d}\tau \mathrm{d}y \\
&\quad + \mathbb{E}[N(t)] \int_t^\infty G(y) f(y) \mathrm{d}y,
\end{aligned}
\end{equation}
and
\begin{equation}
\begin{aligned}
m_3(t) &= \int_0^t \int_0^y [1 + k_{3,0}(t-\tau) + k_{T_1,\tau}(y-\tau) + k_{3,y,0,\tau}(t-y) + m_3(t-y)] g(\tau) f(y) \mathrm{d}\tau \mathrm{d}y \\
&\quad + \mathbb{E}[N(t)] \int_t^\infty G(y) f(y) \mathrm{d}y + \int_0^t \mathbb{E}[N_3(t-\tau)] F(\tau) g(\tau) \mathrm{d}\tau \\
&= \int_0^t   G(y) f(y) \mathrm{d}y + \int_0^t k_{3,0}(t-\tau) F_g(\tau) \mathrm{d}\tau - [1 - F(t)] \int_0^t k_{3,0}(t-\tau) g(\tau) \mathrm{d}\tau \\
&\quad + \int_0^t \int_0^y k_{T_1,\tau}(y-\tau) g(\tau) f(y) \mathrm{d}\tau \mathrm{d}y + \int_0^t \int_0^y k_{3,y,0,\tau}(t-y) g(\tau) f(y) \, \mathrm{d}\tau \, \mathrm{d}y  \\
&\quad + \mathbb{E}[N(t)]\int_t^\infty G(y) f(y) \mathrm{d}y + \int_0^t m_3(t-y) [G(y) f(y) + F(y) g(y)] \mathrm{d}y,
\end{aligned}
\nonumber
\end{equation}
where $k_{3,0}(t)$, $k_{T_1,\tau}(t)$ and $k_{3,y,0,\tau}(t)$ satisfy the following set of equations (similar to Equation~\eqref{eqn:17}), i.e.,
\begin{equation}
\begin{cases}
k_{3,0}(t) = \int_0^t [1 + k_{3,0}(t-v)] \eta h(v) \mathrm{d}v, \\
k_{T_1,\tau}(t) = \int_0^t \mu(v+\tau) \mathrm{d}v + \eta \int_0^t k_{T_1,\tau}(t-v) h(v) \mathrm{d}v, \\
k_{3,y,0,\tau}(t) = \eta H(t) k_{T_3,\tau}(y-\tau) - \int_0^{y-\tau} \eta [h(y-\tau-v) - h(t+y-\tau-v)] k_{T_1,\tau}(v) \, \mathrm{d}v \\
\quad\quad + \int_0^t \eta h(t-v) k_{3,y,0,\tau}(v) \, \mathrm{d}v, \quad \text{with } k_{3,y,0,\tau}(0) = 0,
\end{cases}
\end{equation}
with $k_{3,0}(0) = k_{T_1,\tau}(0) = k_{3,y,0,\tau}(0) = 0$.

Thus we can get the following integral equation for $m_3(t)$:
\begin{equation}\label{eqn:41}
m_3(t) = \Psi_3(t) + \int_0^t m_3(t-y) [G(y) f(y) + F(y) g(y)] \mathrm{d}y = \Psi_3(t) + \int_0^t m_3(t-y) n_3(y) \mathrm{d}y,
\end{equation}
where $n_3(y) = G(y) f(y) + F(y) g(y)$,
$\Psi_3(t)$ is independent with $N_3(t)$, and it is given by
\begin{equation}
\begin{aligned}
\Psi_3(t) &= \int_0^t G(y) f(y) \mathrm{d}y + \int_0^t k_{3,0}(t-\tau) F_g(\tau) \mathrm{d}\tau - [1 - F(t)] \int_0^t k_{3,0}(t-\tau) g(\tau) \mathrm{d}\tau \\
&\quad + \int_0^t  \int_0^y k_{T_1,\tau}(y-\tau) g(\tau) f(y) \mathrm{d}\tau \mathrm{d}y + \int_0^t \int_0^y k_{3,y,0,\tau}(t-y) g(\tau) f(y) \, \mathrm{d}\tau \, \mathrm{d}y  \\
&\quad\quad + \mathbb{E}[N(t)] \int_t^\infty G(y) f(y) \mathrm{d}y,
\end{aligned}
\nonumber
\end{equation}
where we recall that $m(t)=\mathbb{E}[N(t)]$ satisfies the integral equation \eqref{m:t:eqn}.
This completes the proof.
\end{proof}


\begin{proof}[Proof of Corollary~\ref{cor:3}]
Since $\mu(t)=\mu_{i}$ for every $t_{i-1}\leq t<t_{i}$, for any $t_{i}\leq t<t_{i+1}$,
\begin{equation}
G(t)=1-\alpha_{i}e^{-\mu_{i+1}t},
\qquad
g(t)=\mu_{i+1}\alpha_{i}e^{-\mu_{i+1}t},
\end{equation}
where $\alpha_{i}:=e^{-\sum_{j=1}^{i}\mu_{j}(t_{j}-t_{j-1})}e^{\mu_{i+1}t_{i}}$.
It follows that
\begin{align}
n_{3}(t)
&=G(t)f(t)+F(t)g(t)
\nonumber
\\
&=\left(1-\alpha_{i}e^{-\mu_{i+1}t}\right)\gamma e^{-\gamma t}
+(1-e^{-\gamma t})\mu_{i+1}\alpha_{i}e^{-\mu_{i+1}t}
\nonumber
\\
&=\gamma e^{-\gamma t}+\mu_{i+1}\alpha_{i}e^{-\mu_{i+1}t}
-(\gamma+\mu_{i+1})\alpha_{i}e^{-(\mu_{i+1}+\gamma)t}.
\end{align}
Thus, we have
\begin{align}
m_{3}(t)
&=\Psi_{3}(t)+\int_{0}^{t}m_{3}(v)\gamma e^{-\gamma(t-v)}\mathrm{d}v
\nonumber
\\
&\qquad+\int_{0}^{t}m_{3}(v)\mu_{i+1}\alpha_{i}e^{-\mu_{i+1}(t-v)}\mathrm{d}v
-\int_{0}^{t}m_{3}(v)(\gamma+\mu_{i+1})\alpha_{i}e^{-(\mu_{i+1}+\gamma)(t-v)}\mathrm{d}v.
\end{align}
We can re-write the above equation as
\begin{equation}
m_{3}(t)=\Psi_{3}(t)+g_{1}(t)+g_{2}(t)-g_{3}(t),
\end{equation}
where
\begin{align}
&g_{1}(t):=\int_{0}^{t}m_{3}(v)\gamma e^{-\gamma(t-v)}\mathrm{d}v,
\\
&g_{2}(t):=\int_{0}^{t}m_{3}(v)\mu_{i+1}\alpha_{i}e^{-\mu_{i+1}(t-v)}\mathrm{d}v,
\\
&g_{3}(t):=\int_{0}^{t}m_{3}(v)(\gamma+\mu_{i+1})\alpha_{i}e^{-(\mu_{i+1}+\gamma)(t-v)}\mathrm{d}v.
\end{align}
We can compute that
\begin{align}
&m'_{3}(t)=\Psi'_{3}(t)+\gamma(1-\alpha_{i})m_{3}(t)-\gamma g_{1}(t)-\mu_{i+1}g_{2}(t)+(\gamma+\mu_{i+1})g_{3}(t),
\\
&m''_{3}(t)=\Psi''_{3}(t)+\gamma(1-\alpha_{i})m'_{3}(t)+\left((\alpha_{i}-1)\gamma^{2}+2\gamma\mu_{i+1}\alpha_{i}\right)m_{3}(t)
\nonumber
\\
&\qquad\qquad\qquad\qquad\qquad
+\gamma^{2}g_{1}(t)+\mu_{i+1}^{2}g_{2}(t)-(\gamma+\mu_{i+1})^{2}g_{3}(t),
\\
&m'''_{3}(t)=\Psi'''_{3}(t)+\gamma(1-\alpha_{i})m''_{3}(t)+\left((\alpha_{i}-1)\gamma^{2}+2\gamma\mu_{i+1}\alpha_{i}\right)m'_{3}(t)
\nonumber
\\
&\qquad\qquad\qquad\qquad\qquad
+\left(\gamma^{3}+\mu_{i+1}^{3}\alpha_{i}-(\gamma+\mu_{i+1})^{3}\alpha_{i}\right)m_{3}(t)
\nonumber
\\
&\qquad\qquad\qquad\qquad\qquad\qquad\qquad
-\gamma^{3}g_{1}(t)-\mu_{i+1}^{3}g_{2}(t)
+(\gamma+\mu_{i+1})^{3}g_{3}(t).
\end{align}
Moreover, by considering $0=t_{0}\leq t<t_{1}$,
we have $m_{3}(0)=0$
and $g_{1}(0)=g_{2}(0)=g_{3}(0)=0$ such that
$m'_{3}(0)=\Psi'_{3}(0)$
and $m''_{3}(0)=\Psi''_{3}(0)+\gamma(1-\alpha_{0})m'_{3}(0)
=\Psi''_{3}(0)+\gamma(1-\alpha_{0})\Psi'_{3}(0)$,
where $\alpha_{0}=1$.

Next, let us define:
\begin{equation}
M_{3}(t):=
\left[\begin{array}{c}
m_{3}(t)
\\
m'_{3}(t)
\\
m''_{3}(t)
\end{array}\right],
\qquad
U_{3}(t):=
\left[\begin{array}{c}
\Psi_{3}(t)
\\
\Psi'_{3}(t)
\\
\Psi''_{3}(t)
\end{array}\right],
\qquad
G(t):=
\left[\begin{array}{c}
g_{1}(t)
\\
g_{2}(t)
\\
g_{3}(t)
\end{array}\right].
\end{equation}
Then, we have
\begin{equation}\label{M:3:t:eqn}
M_{3}(t)=U_{3}(t)+A_{i}M_{3}(t)+C_{i}G(t),
\end{equation}
where
\begin{equation}
A_{i}:=
\left[
\begin{array}{ccc}
0 & 0 & 0
\\
\gamma(1-\alpha_{i}) & 0 & 0
\\
(\alpha_{i}-1)\gamma^{2}+2\gamma\mu_{i+1}\alpha_{i} & \gamma(1-\alpha_{i}) & 0
\end{array}
\right],
\qquad
C_{i}:=
\left[
\begin{array}{ccc}
1 & 1 & -1
\\
-\gamma & -\mu_{i+1} & \gamma+\mu_{i+1}
\\
\gamma^{2} & \mu_{i+1}^{2} & -(\gamma+\mu_{i+1})^{2}
\end{array}
\right],
\end{equation}
and
\begin{equation}\label{M:prime:3:t:eqn}
M'_{3}(t)=U'_{3}(t)+B_{i}M_{3}(t)+D_{i}G(t),
\end{equation}
where
\begin{align}
B_{i}:=
\left[
\begin{array}{ccc}
\gamma(1-\alpha_{i}) & 0 & 0
\\
(\alpha_{i}-1)\gamma^{2}+2\gamma\mu_{i+1}\alpha_{i} & \gamma(1-\alpha_{i}) & 0
\\
\gamma^{3}+\mu_{i+1}^{3}\alpha_{i}-(\gamma+\mu_{i+1})^{3}\alpha_{i} & (\alpha_{i}-1)\gamma^{2}+2\gamma\mu_{i+1}\alpha_{i} & \gamma(1-\alpha_{i})
\end{array}
\right],
\end{align}
and
\begin{equation}
D_{i}:=
\left[
\begin{array}{ccc}
-\gamma & -\mu_{i+1} & \gamma+\mu_{i+1}
\\
\gamma^{2} & \mu_{i+1}^{2} & -(\gamma+\mu_{i+1})^{2}
\\
-\gamma^{3} & -\mu_{i+1}^{2} & (\gamma+\mu_{i+1})^{3}
\end{array}
\right].
\end{equation}
It follows from \eqref{M:3:t:eqn} that
\begin{equation}\label{G:t:eqn}
G(t)=C_{i}^{-1}M_{3}(t)-C_{i}^{-1}U_{3}(t)-C_{i}^{-1}A_{i}M_{3}(t).
\end{equation}
By substituting \eqref{G:t:eqn} into \eqref{M:prime:3:t:eqn}, we obtain
\begin{equation}\label{M:3:ODE}
M'_{3}(t)=\left(B_{i}+D_{i}C_{i}^{-1}-D_{i}C_{i}^{-1}A_{i}\right)M_{3}(t)+U'_{3}(t)-D_{i}C_{i}^{-1}U_{3}(t).
\end{equation}
By solving the linear ODE \eqref{M:3:ODE}, we get for any $t_{i}\leq t<t_{i+1}$:
\begin{align}
M_{3}(t)&=e^{(B_{i}+D_{i}C_{i}^{-1}-D_{i}C_{i}^{-1}A_{i})(t-t_{i})}M_{3}(t_{i})
\nonumber
\\
&\qquad\qquad\qquad
+\int_{t_{i}}^{t}e^{(B_{i}+D_{i}C_{i}^{-1}-D_{i}C_{i}^{-1}A_{i})(t-s)}\left(U'_{3}(s)-D_{i}C_{i}^{-1}U_{3}(s)\right)\mathrm{d}s,
\end{align}
for any $i=0,1,2,\ldots$ with
\begin{equation}
M_{3}(0)=
\left[\begin{array}{c}
0
\\
\Psi'_{3}(0)
\\
\Psi''_{3}(0)
\end{array}\right].
\end{equation}
This completes the proof.
\end{proof}

\subsection{Proofs of the Results in Section~\ref{sec:applications}}

\begin{proof}[Proof of Proposition\ref{prop:C:1}]
Since the instants $iT$ ($i=1,2,\ldots$) are renewal points,
based on the renewal limit theorem (see e.g. page 115 in \cite{Ross1996}),
if the distribution of renewal cycle is non-lattice,
$C_{1}(T)$ is the ratio of expected loss during a renewal cycle
and the expected length of a renewal cycle which yields that
\begin{equation}
C_{1}(T)
=\frac{c_{f}\mathbb{E}[N(T)]+c_{Ip}\mathbb{E}[N_{R}(T)]+c_{p}}{T},
\end{equation}
where $N(t)$ is a renewal Hawkes process, and it is represented by $N_{i}(t)$ for $\textup{R}_i$Hawkes process, $i=1,2,3$,
that are introduced in Section~\ref{sec:three:types}, and their expectations are computed in Theorems~\ref{thm:1}, \ref{thm:2} and \ref{thm:3},
and $N_{R}(t)$ is a renewal process for immigrants with interarrivals $Y_{1},Y_{2},\ldots$.
From the properties of renewal processes, we have
\begin{equation}
\mathbb{E}[N_{R}(t)]=\sum_{n=1}^{\infty}F^{(n)}(t),
\end{equation}
where $F^{(n)}(t)$ is the $n$-fold convolution of $F$ with itself.
This completes the proof.
\end{proof}


\begin{proof}[Proof of Proposition~\ref{prop:C:2}]
First, $N_{R}(t)$ is a renewal process for immigrants with interarrivals $Y_{1},Y_{2},\ldots$ such that
by Proposition~\ref{prop:C:1}, we have
$\mathbb{E}[N_{R}(t)]=\sum_{n=1}^{\infty}F^{(n)}(t)$, where $F^{(n)}(t)$ is the $n$-fold convolution of $F$ with itself,
where $F$ is the distribution function of interarrivals $Y_{1},Y_{2},\ldots$.

Next, by the property of non-homogenous Poisson processes, we have
\begin{equation}
\mathbb{E}[N_{\mathrm{NHP}}(Y_{1})]=\int_{0}^{\infty}\mathbb{E}[N_{\mathrm{NHP}}(y)|Y_{1}=y]\mathrm{d}F(y)
=\int_{0}^{\infty}\int_{0}^{y}\mu(v)\mathrm{d}v\mathrm{d}F(y).
\end{equation}
Based on the definition of $N_{I}(t)$, we have
\begin{align}
\mathbb{E}[N_{I}(t)]
&=\mathbb{E}\left[\sum_{i=1}^{N_{R}(t)}N_{\mathrm{NHP}}(Y_{i})I_{\left\{\sum_{j=1}^{N_{R}(t)}Y_{j}\leq t\right\}}\right]
+\mathbb{E}\left[N_{\mathrm{NHP}}\left(t-\sum_{j=1}^{N_{R}(t)}Y_{j}\right)I_{\left\{\sum_{j=1}^{N_{R}(t)}Y_{j}\leq t\right\}}\right]
\nonumber
\\
&=\sum_{n=0}^{\infty}\mathbb{E}\left[\sum_{i=1}^{N_{R}(t)}N_{\mathrm{NHP}}(Y_{i})I_{\left\{\sum_{j=1}^{N_{R}(t)}Y_{j}\leq t\right\}}|N_{R}(t)=n\right]
\mathbb{P}(N_{R}(t)=n)
\nonumber
\\
&\qquad\qquad
+\sum_{n=0}^{\infty}\mathbb{E}\left[N_{\mathrm{NHP}}\left(t-\sum_{j=1}^{N_{R}(t)}Y_{j}\right)I_{\left\{\sum_{j=1}^{N_{R}(t)}Y_{j}\leq t\right\}}|N_{R}(t)=n\right]
\mathbb{P}(N_{R}(t)=n).
\nonumber
\end{align}
By applying the identity $\{N_{R}(t)=n\}=\left\{\sum_{j=1}^{n}Y_{j}\leq t\right\}$,
we can further get
\begin{align}
\mathbb{E}[N_{I}(t)]
&=\sum_{n=0}^{\infty}\mathbb{E}\left[\sum_{i=1}^{n}N_{\mathrm{NHP}}(Y_{i})|N_{R}(t)=n\right]
\mathbb{P}(N_{R}(t)=n)
\nonumber
\\
&\qquad\qquad
+\sum_{n=0}^{\infty}\mathbb{E}\left[N_{\mathrm{NHP}}\left(t-\sum_{j=1}^{n}Y_{j}\right)|N_{R}(t)=n\right]
\mathbb{P}(N_{R}(t)=n)
\nonumber
\\
&=\mathbb{E}[N_{\mathrm{NHP}}(Y_{1})]\sum_{n=0}^{\infty}n\mathbb{P}(N_{R}(t)=n)
+\int_{0}^{t}\mathbb{E}\left[N_{\mathrm{NHP}}(t-y)\right]\sum_{n=0}^{\infty}f^{(n)}(y)\mathbb{P}(N_{R}(t)=n)\mathrm{d}y
\nonumber
\\
&=\mathbb{E}[N_{\mathrm{NHP}}(Y_{1})]\mathbb{E}[N_{R}(t)]+\int_{0}^{t}\mathbb{E}\left[N_{\mathrm{NHP}}(t-y)\right]\mathbb{E}\left[f^{(N_{R}(t))}(y)\right]\mathrm{d}y,
\end{align}
where $f(y)=dF(y)/dy$ and $f^{(n)}(y)$ denotes the $n$-th convolution of $f(y)$ with itself.

Finally, we have
\begin{equation}
\mathbb{E}[N_{O}(t)]=\mathbb{E}[N(t)-N_{I}(t)]=\mathbb{E}[N(t)]-\mathbb{E}[N_{I}(t)],
\end{equation}
where $N_{I}(t)$ is the number of failures in part A by time $t$, which is modeled by a point process for immigrants with intensity function
$\mu(t)$ renewed at times $Y_{1}$, $\sum_{i=1}^{2}Y_{i}$, $\sum_{i=1}^{3}Y_{i}$, etc.,
and $N_{O}(t)$ is the number
of failures in part B by time $t$, which is modeled by a self-exciting process for offspring of $N(t)$ with intensity function
$\eta\int_{0}^{t}h(t-s)dN(s)$.
Therefore, we get
\begin{align}
C_{2}(T)&=\frac{c_{I}\mathbb{E}[N_{I}(T)]+c_{O}\mathbb{E}[N_{O}(T)]+c_{Ip}\mathbb{E}[N_{R}(T)]+c_{p}}{T}
\nonumber
\\
&=\frac{(c_{I}-c_{O})\mathbb{E}[N_{I}(T)]+c_{O}\mathbb{E}[N(T)]+c_{Ip}\mathbb{E}[N_{R}(T)]+c_{p}}{T}.\label{C:2:T:formula}
\end{align}
This completes the proof.
\end{proof}


\begin{proof}[Proof of Corollary~\ref{N:I:T:complicated}]
We recall from \eqref{N:I:T:original} that
\begin{equation}\label{N:I:plug:1}
\mathbb{E}[N_{I}(T)]
=\mathbb{E}[N_{\mathrm{NHP}}(Y_{1})]\mathbb{E}[N_{R}(T)]+\int_{0}^{T}\mathbb{E}\left[N_{\mathrm{NHP}}(T-y)\right]\mathbb{E}\left[f^{(N_{R}(T))}(y)\right]\mathrm{d}y,
\end{equation}
where
\begin{equation}\label{N:I:plug:2}
\mathbb{E}[N_{R}(t)]=\sum_{n=1}^{\infty}F^{(n)}(t),
\end{equation}
and we recall from \eqref{NHP:Y:1:formula} that
\begin{equation}\label{N:I:plug:3}
\mathbb{E}[N_{\mathrm{NHP}}(Y_{1})]=\int_{0}^{\infty}\int_{0}^{y}\mu(v)\mathrm{d}v\mathrm{d}F(y).
\end{equation}
Since $N_{\mathrm{NHP}}(t)$ is a non-homogeneous Poisson process with intensity function $\mu(t)$, we have
\begin{equation}\label{N:I:plug:4}
\mathbb{E}\left[N_{\mathrm{NHP}}(T-y)\right]=\int_{0}^{T-y}\mu(s)\mathrm{d}s.
\end{equation}
Next, we can compute that
\begin{align}
\mathbb{E}\left[f^{(N_{R}(T))}(y)\right]=\sum_{n=0}^{\infty}f^{(n)}(y)\mathbb{P}(N_{R}(T)=n),
\end{align}
where we can compute that
\begin{align}
\mathbb{P}(N_{R}(T)=n)&=\mathbb{P}(Y_{1}+\cdots+Y_{n}\leq T<Y_{1}+\cdots+Y_{n}+Y_{n+1})
\nonumber
\\
&=\mathbb{P}(Y_{1}+\cdots+Y_{n}\leq T)-\mathbb{P}(Y_{1}+\cdots+Y_{n}+Y_{n+1}\leq T)
\nonumber
\\
&=\int_{0}^{T}f^{(n)}(v)\mathrm{d}v-\int_{0}^{T}f^{(n+1)}(v)\mathrm{d}v.\label{N:I:plug:5}
\end{align}
By substituting \eqref{N:I:plug:2}, \eqref{N:I:plug:3}, \eqref{N:I:plug:4}, \eqref{N:I:plug:5} into \eqref{N:I:plug:1},
we get
\begin{align}
\mathbb{E}[N_{I}(T)]
&=\sum_{n=1}^{\infty}F^{(n)}(t)\int_{0}^{\infty}\int_{0}^{y}\mu(v)\mathrm{d}v\mathrm{d}F(y)
\nonumber
\\
&\qquad
+\int_{0}^{T}\int_{0}^{T-y}\mu(s)\mathrm{d}s
\sum_{n=0}^{\infty}f^{(n)}(y)\left(\int_{0}^{T}f^{(n)}(v)\mathrm{d}v-\int_{0}^{T}f^{(n+1)}(v)\mathrm{d}v\right)\mathrm{d}y,
\end{align}
which completes the proof.
\end{proof}


\begin{proof}[Proof of Corollary~\ref{N:I:T}]
One can compute that $\mathbb{E}[N_{\mathrm{NHP}}(Y_{1})]
=\int_{0}^{\infty}y^{2}\gamma e^{-\gamma y}\mathrm{d}y=\frac{2}{\gamma^{2}}$,
$\mathbb{E}[N_{R}(T)]=\gamma T$,
and moreover, conditional on $N_{R}(T)$,
$f^{(N_{R}(T))}(y)$ is the probability density function
of an Erlang distribution with scale parameter $1/\gamma$
and shape parameter $N_{R}(T)$ such that
\begin{align}
\int_{0}^{T}\mathbb{E}\left[N_{\mathrm{NHP}}(T-y)\right]\mathbb{E}\left[f^{(N_{R}(T))}(y)\right]\mathrm{d}y
&=\int_{0}^{T}(T-y)^{2}\mathbb{E}\left[f^{(N_{R}(T))}(y)\right]\mathrm{d}y
\nonumber
\\
&=\mathbb{E}\left[\int_{0}^{T}\left(T^{2}-2Ty+y^{2}\right)f^{(N_{R}(T))}(y)\mathrm{d}y\right]
\nonumber
\\
&=\mathbb{E}\left[\int_{0}^{T}\left(T^{2}-2Ty+y^{2}\right)\frac{\gamma^{N_{R}(T)}y^{N_{R}(T)-1}e^{-\gamma y}}{(N_{R}(T)-1)!}\mathrm{d}y\right].
\end{align}
We can further compute that
\begin{align}
&\mathbb{E}\left[\int_{0}^{T}\left(T^{2}-2Ty+y^{2}\right)\frac{\gamma^{N_{R}(T)}y^{N_{R}(T)-1}e^{-\gamma y}}{(N_{R}(T)-1)!}\mathrm{d}y\right]
\nonumber
\\
&=T^{2}\mathbb{E}\left[\int_{0}^{T}\frac{\gamma^{N_{R}(T)}y^{N_{R}(T)-1}e^{-\gamma y}}{(N_{R}(T)-1)!}\mathrm{d}y\right]
-2T\mathbb{E}\left[\frac{N_{R}(T)}{\gamma}\int_{0}^{T}\frac{\gamma^{N_{R}(T)+1}y^{(N_{R}(T)+1)-1}e^{-\gamma y}}{((N_{R}(T)+1)-1)!}\mathrm{d}y\right]
\nonumber
\\
&\qquad+\mathbb{E}\left[\frac{N_{R}(T)(N_{R}(T)+1)}{\gamma^{2}}\int_{0}^{T}\frac{\gamma^{N_{R}(T)+2}y^{(N_{R}(T)+2)-1}e^{-\gamma y}}{((N_{R}(T)+2)-1)!}\mathrm{d}y\right]
\nonumber
\\
&=T^{2}\mathbb{E}\left[1-\sum_{n=0}^{N_{R}(T)-1}\frac{1}{n!}e^{-\gamma T}(\gamma T)^{n}\right]
-2T\mathbb{E}\left[\frac{N_{R}(T)}{\gamma}\left(1-\sum_{n=0}^{N_{R}(T)}\frac{1}{n!}e^{-\gamma T}(\gamma T)^{n}\right)\right]
\nonumber
\\
&\qquad
+\mathbb{E}\left[\frac{N_{R}(T)(N_{R}(T)+1)}{\gamma^{2}}\left(1-\sum_{n=0}^{N_{R}(T)+1}\frac{1}{n!}e^{-\gamma T}(\gamma T)^{n}\right)\right]
\nonumber
\\
&=T^{2}-\frac{2T}{\gamma}\mathbb{E}[N_{R}(T)]
+\frac{1}{\gamma^{2}}\mathbb{E}[N_{R}(T)(N_{R}(T)+1)]
-T^{2}\mathbb{E}\left[\sum_{n=0}^{N_{R}(T)-1}\frac{1}{n!}e^{-\gamma T}(\gamma T)^{n}\right]
\nonumber
\\
&\qquad
+\frac{2T}{\gamma}\mathbb{E}\left[N_{R}(T)\sum_{n=0}^{N_{R}(T)}\frac{1}{n!}e^{-\gamma T}(\gamma T)^{n}\right]
-\frac{1}{\gamma^{2}}\mathbb{E}\left[N_{R}(T)(N_{R}(T)+1)\sum_{n=0}^{N_{R}(T)+1}\frac{1}{n!}e^{-\gamma T}(\gamma T)^{n}\right],
\end{align}
which implies that
\begin{align}
&\mathbb{E}\left[\int_{0}^{T}\left(T^{2}-2Ty+y^{2}\right)\frac{\gamma^{N_{R}(T)}y^{N_{R}(T)-1}e^{-\gamma y}}{(N_{R}(T)-1)!}\mathrm{d}y\right]
\nonumber
\\
&=T^{2}-\frac{2T}{\gamma}(\gamma T)
+\frac{1}{\gamma^{2}}((\gamma T)^{2}+2\gamma T)
-T^{2}\sum_{k=0}^{\infty}\frac{e^{-\gamma T}(\gamma T)^{k}}{k!}\sum_{n=0}^{k-1}\frac{1}{n!}e^{-\gamma T}(\gamma T)^{n}
\nonumber
\\
&\qquad
+\frac{2T}{\gamma}\sum_{k=0}^{\infty}\frac{e^{-\gamma T}(\gamma T)^{k}}{k!}k\sum_{n=0}^{k}\frac{1}{n!}e^{-\gamma T}(\gamma T)^{n}
-\frac{1}{\gamma^{2}}\sum_{k=0}^{\infty}\frac{e^{-\gamma T}(\gamma T)^{k}}{k!}k(k+1)\sum_{n=0}^{k+1}\frac{1}{n!}e^{-\gamma T}(\gamma T)^{n},
\end{align}
where we used the closed-form formula for the cumulative distribution function of an Erlang distribution
and the fact that $N_{R}(T)$ is Poisson distributed with mean $\gamma T$.
This completes the proof.
\end{proof}


\begin{proof}[Proof of Proposition~\ref{cor:C:2}]
It follows from Proposition~\ref{prop:C:2} that
\begin{align}
\mathbb{E}[N_{I}(t)]
&=\mathbb{E}[N_{\mathrm{NHP}}(Y_{1})]\mathbb{E}[N_{R}(t)]+\int_{0}^{t}\mathbb{E}\left[N_{\mathrm{NHP}}(t-y)\right]\mathbb{E}\left[f^{(N_{R}(t))}(y)\right]\mathrm{d}y
\nonumber
\\
&\geq\mathbb{E}[N_{\mathrm{NHP}}(Y_{1})]\mathbb{E}[N_{R}(t)].\label{L:bound}
\end{align}
On the other hand,
we can compute that
\begin{align*}
&\sum_{n=0}^{\infty}\mathbb{E}\left[N_{\mathrm{NHP}}\left(t-\sum_{j=1}^{N_{R}(t)}Y_{j}\right)I_{\left\{\sum_{j=1}^{N_{R}(t)}Y_{j}\leq t\right\}}|N_{R}(t)=n\right]
\mathbb{P}(N_{R}(t)=n)
\\
&\leq\sum_{n=0}^{\infty}\mathbb{E}\left[N_{\mathrm{NHP}}(Y_{1})\right]\mathbb{P}(N_{R}(t)=n)
\\
&=\mathbb{E}\left[N_{\mathrm{NHP}}(Y_{1})\right],
\end{align*}
which implies that
\begin{equation}
\mathbb{E}[N_{I}(t)]
\leq\mathbb{E}\left[N_{\mathrm{NHP}}(Y_{1})\right]\left(\mathbb{E}[N_{R}(t)]+1\right).\label{U:bound}
\end{equation}
Hence, we conclude from \eqref{L:bound} and \eqref{U:bound} that
\begin{equation}\label{upper:lower:bdd}
\mathbb{E}\left[N_{\mathrm{NHP}}(Y_{1})\right]\mathbb{E}[N_{R}(t)]
\leq\mathbb{E}[N_{I}(t)]
\leq\mathbb{E}\left[N_{\mathrm{NHP}}(Y_{1})\right]\left(\mathbb{E}[N_{R}(t)]+1\right).
\end{equation}
Finally, by applying \eqref{upper:lower:bdd} and \eqref{C:2:T:formula}, we obtain
\begin{align}
&\frac{(c_{I}-c_{O})\mathbb{E}[N_{\mathrm{NHP}}(Y_{1})]\mathbb{E}[N_{R}(T)]+c_{O}\mathbb{E}[N(T)]+c_{Ip}\mathbb{E}[N_{R}(T)]+c_{p}}{T}
\nonumber
\\
&\leq
C_{2}(T)
\leq
\frac{(c_{I}-c_{O})\mathbb{E}[N_{\mathrm{NHP}}(Y_{1})]\left(\mathbb{E}[N_{R}(T)]+1\right)+c_{O}\mathbb{E}[N(T)]+c_{Ip}\mathbb{E}[N_{R}(T)]+c_{p}}{T},
\end{align}
where $N_{\mathrm{NHP}}(t)$ is a non-homogeneous Poisson process with intensity function $\mu(t)$,
$N_{R}(t)$ is a renewal process for immigrants with interarrivals $Y_{1},Y_{2},\ldots$.
This completes the proof.
\end{proof}


\section{Further Illustrations of Renewal Hawkes Processes}\label{sec:illustration}

To understand three classes of renewal Hawkes processes proposed in Section~\ref{sec:three:types}, one of their realized intensity function trajectories is plotted in Figure~\ref{fig1}. The purpose using the following concrete examples is to show how the intensity functions have evolved, in which the related functions and parameters are assumed as follows.
Let $\mu(t) = 1.2t$, $h(t) = 2e^{-2t}$, $\eta = \frac{3}{4}$, and the exogenous factor $Y$ has the realized values: $y_1 = 0.67$, $\tau_1=0.5$, and thus the first renewal time is $u_1=0.67$. By further sampling, we have
$y_2 = 0.89$, $\tau_{21}=0.7$, $\tau_{22}=0.3$, and thus the second renewal time is $u_2=0.67+0.89=1.56$, where $\tau_{21},\tau_{22}$ are first and second realized values for $T_2$. By continuing sampling, we have $y_3 = 0.6$, $\tau_3=1.12$,
the third renewal time is $u_3=1.56+0.6=2.16$. Finally, there are not occurrence of any offspring or renewal events by time $t=3$.
Then the realized intensity function is given by
\begin{equation}
\lambda_1(t) =
\begin{cases}
1.2t, & 0 \leq t < 0.5, \\
1.2t + 1.5e^{-2(t - 0.5)}, & 0,5 \leq t <0.67,  \\
1.2(t - 0.67) + 1.5e^{-2(t - 0.5)}, & 0.67 \leq t < 1.37, \\
1.2(t - 0.67) + 1.5e^{-2(t - 0.5)} + 1.5e^{-2(t - 1.37)}, & 1.37 \leq t < 1.56, \\
1.2(t - 1.56) + 1.5e^{-2(t - 0.5)} + 1.5e^{-2(t - 1.37)} , & 1.56 \leq t <2.16, \\
1.2(t - 2.16) + 1.5e^{-2(t - 0.5)} + 1.5e^{-2(t - 1.37)} , & 2.16 \leq t <3.0 .
\end{cases}
\nonumber
\end{equation}

The $\text{R}_1$-Hawkes process has three renewal times during interval $[0,3.0]$, i.e., $u_1 = y_1 = 0.67$, $u_2 = y_1+y_2=0.67+0.89=1.56$, $u_3 = u_2+y_3 = 1.56+0.6=2.16$.
 There are two first-generation offspring generated at times $ \tau_{1}=0.5, \tau_{2}=u_{1}+0.7=1.37$. Note, we assume that the second generation offspring, third generation offspring, etc., do not occur by time $t=3.0$.

Similarly, we have the sampling values: $y_1=0.67,\tau_{1}=0.5, y_2=0.89, \tau_{2}=0.7, y_3=0.6, \tau_{3}=1.12, y_4=1.3, \tau_{4}=0.9 $, and the intensity function is given at the interval $[0,3.0]$ by
\begin{equation}
\lambda_2(t) =
\begin{cases}
1.2t, & 0 \leq t < 0.5, \\
1.2(t - 0.5) + 1.5e^{-2(t - 0.5)}, & 0.5 \leq t < 1.2, \\
1.2(t - 1.2) + 1.5e^{-2(t - 0.5)} + 1.5e^{-2(t - 1.2)}, & 1.2 \leq t < 1.8, \\
1.2(t - 1.8) + 1.5e^{-2(t - 0.5)} + 1.5e^{-2(t - 1.2)} , & 1.8 \leq t < 2.7, \\
1.2(t - 2.7) + 1.5e^{-2(t - 0.5)} + 1.5e^{-2(t - 1.2)} + 1.5e^{-2(t - 2.7)}, & 2.7 \leq t < 3.0. \\
\end{cases}
\nonumber
\end{equation}

Indeed, the $\text{R}_2$-Hawkes process has four renewal times during interval $[0,3]$, i.e., $v_1 = \min\{y_1, \tau_1\} = 0.5$, $v_2 = \min\{v_1 + y_2, v_1 + \tau_2\} = 1.2$, $v_3 = \min\{v_2 + y_3, v_2 + \tau_3\} = 1.8,$
 $v_4 = \min\{v_3 + y_4, v_3 + \tau_4\} = 2.7. $ There are three first-generation offspring generated at times $ \tau_{1}=0.5, \tau_{2}=v_{1}+0.7=1.2, \tau_{4}=v_{3}+0.9=2.7$.
The realized intensity function for the $\text{R}_3$-Hawkes process is given by
\begin{equation}
\lambda_3(t) =
\begin{cases}
1.2t, & 0 \leq t < 0.5, \\
1.2t + 1.5e^{-2(t - 0.5)}, & 0,5 \leq t <0.67,  \\
1.2(t - 0.67) + 1.5e^{-2(t - 0.5)}, & 0.67 \leq t < 1.37, \\
1.2(t - 1.37) + 1.5e^{-2(t - 0.5)} + 1.5e^{-2(t - 1.37)}, & 1.37 \leq t < 1.56, \\
1.2(t - 1.56) + 1.5e^{-2(t - 0.5)} + 1.5e^{-2(t - 1.37)} , & 1.56 \leq t <2.68, \\
1.2(t - 2.68) + 1.5e^{-2(t - 0.5)} + 1.5e^{-2(t - 1.37)}+1.5e^{-2(t - 2.68)} , & 2.68 \leq t <3.0 .
\end{cases}
\nonumber
\end{equation}

The $\text{R}_3$-Hawkes process has three renewal times during interval $[0,3.0]$, i.e., $w_1 = y_1 = 0.67$, $w_2 = w_1+y_2=0.67+0.89=1.56$, $w_3 = w_2+\tau_{3} = 1.56+1.13=2.68$.
 There are three first-generation offspring generated at times $ \tau_{1}=0.5, \tau_{2}=w_{1}+0.7=1.37, \tau_{3}=w_{2}+1.12=2.68$. Based on the three formulas, we plot their realized intensity functions in Figure~\ref{fig1}.

\begin{figure}[h]
    \centering
    \includegraphics[scale=0.42]{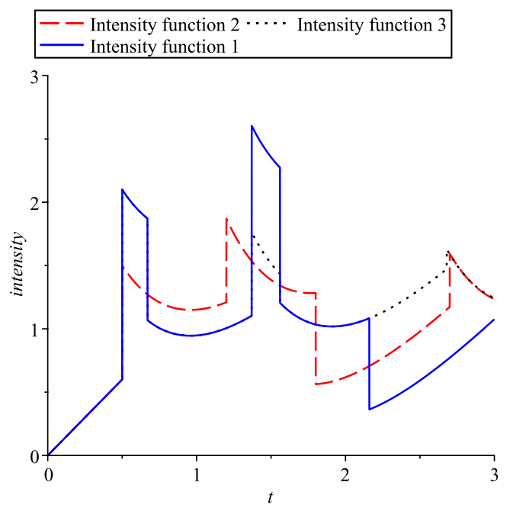}
    \caption{The realized intensity functions of three classes of renewal Hawkes processes.}
    \label{fig1}
\end{figure}

\end{document}